\documentclass[11pt]{amsart}
\usepackage{fullpage}
\usepackage{amsmath,amssymb,amsfonts}
\usepackage{enumerate}
\usepackage{amscd, amssymb, latexsym, amsmath, amscd}
\usepackage{graphicx}
\usepackage{aurical}
\usepackage[T1]{fontenc}
\usepackage{mathrsfs}
\usepackage{url}
\usepackage{youngtab}
\usepackage[small,nohug,heads=vee]{diagrams}
\diagramstyle[labelstyle=\scriptstyle]
\numberwithin{equation}{section}

\theoremstyle{plain}
\newtheorem{theorem}{Theorem}[section]
\newtheorem*{thm}{Theorem}
\newtheorem{proposition}[theorem]{Proposition}
\newtheorem{lemma}[theorem]{Lemma}

\newtheorem{corollary}[theorem]{Corollary}

\theoremstyle{definition}
\newtheorem{definition}[theorem]{Definition}
\newtheorem*{dfn}{Definition}

\newtheorem{example}[theorem]{Example}

\theoremstyle{remark}
\newtheorem{remark}[theorem]{Remark}
\newtheorem{remarks}[theorem]{Remarks}

\makeatletter

\newcommand{\Rmnum}[1]{\expandafter\@slowromancap\romannumeral #1@}
\makeatother

\newcommand{\Z}{\mathbb{Z}}

\newcommand{\N}{\mathbb{N}}
\newcommand{\g}{\mathfrak{g}}
\newcommand{\n}{\mathfrak{n}}

\newcommand{\p}{\mathfrak{p}}

\newcommand{\m}{\mathfrak{m}}
\newcommand{\e}{\mathfrak{e}}
\newcommand{\f}{\mathfrak{f}}
\newcommand{\V}{\mathscr{V}}
\newcommand{\D}{\mathscr{D}}
\renewcommand{\S}{\mathscr{S}}
\newcommand{\Sch}{\mathscr{C}}
\newcommand{\Y}{\mathscr{Y}}
\renewcommand{\u}{\mathfrak{u}}
\newcommand{\set}[2]{ \left\{ {#1} \, \left| \, {#2} \right\}\right.}
\newcommand{\A}{\mathcal{A}}
\newcommand{\X}{\mathcal{X}}
\renewcommand{\Im}{\operatorname{Im}}

\newcommand{\HH}{\mathscr{H}}

\newcommand{\z}{\mathfrak{z}}
\renewcommand{\sl}{\mathfrak{sl}}
\newcommand{\R}{\mathbb{R}}
\newcommand{\C}{\mathbb{C}}
\renewcommand{\H}{\mathcal{H}}
\newcommand{\Ind}{\operatorname{Ind}}
\newcommand{\Tr}{\operatorname{Tr}}
\newcommand{\Orb}{\mathcal{O}}
\renewcommand{\k}{\mbox{\Fontauri k}}
\newcommand{\Ad}{\operatorname{Ad}}
\newcommand{\ad}{\operatorname{ad}}
\newcommand{\Hom}{\operatorname{Hom}}

\newcommand{\End}{\operatorname{End}}
\newcommand{\supp}{\operatorname{supp}}
\newcommand{\Stab}{\mbox{Stab}}
\newcommand{\Lie}{\operatorname{Lie}}
\newcommand{\GL}{\operatorname{GL}}
\newcommand{\Sp}{\operatorname{Sp}}
\newcommand{\Mp}{\operatorname{Mp}}
\renewcommand{\a}{\mathfrak{a}}
\newcommand{\Wh}{\operatorname{Wh}}
\newcommand{\Span}{\operatorname{Span}}
\newcommand{\WF}{\operatorname{WF}}
\newcommand{\AC}{\operatorname{AC}}
\newcommand{\Lin}{\operatorname{Lin}}
\newcommand{\Diff}{\operatorname{Diff}}
\newcommand{\Max}{\operatorname{Max}}
\newcommand{\Her}{\operatorname{Her}}
\newcommand{\vsp}{{\vspace{0.2in}}}

\title[Local theta lifting of generalized Whittaker models]{Local theta lifting of generalized Whittaker models associated to nilpotent orbits}

\author [R. Gomez] {Raul Gomez}
\address{Department of Mathematics\\
National University of Singapore\\
Block S17, 10 Lower Kent Ridge Road\\
Singapore 119076} 
\curraddr{Department of Mathematics, 593 Malott Hall, Cornell University, Ithaca, NY14853, USA}
\email{rg558@cornell.edu}

\author [C.-B. Zhu] {Chen-Bo Zhu}
\address{Department of Mathematics\\
National University of Singapore\\
Block S17, 10 Lower Kent Ridge Road\\
Singapore 119076} \email{matzhucb@nus.edu.sg}

\begin{document}


\thanks{Both authors are supported by MOE2010-T2-2-113}

\begin{abstract} Let $(G,\tilde{G})$ be a reductive dual pair over a local field $\mbox{\Fontauri k}$ of characteristic 0, and denote by $V$ and $\tilde{V}$ the standard modules
of $G$ and $\tilde{G}$, respectively. Consider the set $\Max \Hom (V,\tilde{V})$ of full rank elements in $\Hom(V,\tilde{V})$, and the nilpotent orbit correspondence $\Orb \subset \g$ and $\Theta (\Orb)\subset \tilde{\g}$ induced by elements of $\Max \Hom (V,\tilde{V})$ via the moment maps. Let $(\pi,\mathscr{V})$ be a smooth
irreducible representation of $G$. We show that there is a correspondence of the generalized Whittaker models of $\pi $ of type $\Orb$ and
of $\Theta (\pi)$ of type $\Theta (\Orb)$, where $\Theta (\pi)$ is the full theta lift of $\pi $. When $(G,\tilde{G})$ is in the stable range with $G$ the smaller member, every nilpotent orbit $\Orb \subset \g$ is in the image of the moment map from $\Max \Hom (V,\tilde{V})$. In this case, and for $\k$ non-Archimedean, the result has been previously obtained by M\oe{}glin in a different approach.
\end{abstract}

\maketitle

\tableofcontents

\section{Introduction and main result}

Let $\k$ be a local field of characteristic 0, for which we fix
a non-trivial unitary character $\psi$. Let $G$ be a reductive group
over $\k$, $\g$ its Lie algebra, on which we fix an
$\Ad G$-invariant non-degenerate bilinear form $\kappa$. Let
$\gamma=\{X,H,Y\}\subset \g$ be an $\sl_{2}$-triple, that is
\[
[H,X]=2X,\qquad [H,Y]=-2Y, \qquad [X,Y]=H.
\]
Set $\g_{j}=\set{Z\in \g}{\ad(H)Z=jZ}$, for $j\in \Z$. Then, from standard
$\sl_{2}$-theory, we have a finite direct sum $\g=\oplus_{j\in \Z}\g_{j}$.
Define the Lie subalgebras $\u=\oplus_{j\leq -2}
\g_{j}$, $\n=\oplus_{j\leq -1} \g_{j}$, $\p=\oplus_{j\leq 0} \g_{j}$
and $\m=\g_{0}$. Let $U$, $N$, $P$, and $M$ be the corresponding
subgroups of $G$. Thus $U=\exp \u$, $N=\exp \n$, $P=\set{p\in G}{\Ad(p)\p
\subset \p}$ and $M=\set{m\in G}{\mbox{$\Ad(m)H=H$}}$.
Let
\begin{equation}
\label{defchi}
\chi_{\gamma}(\exp Z)=\psi (\kappa(X,Z)), \ \ \ \ \forall \ Z\in \u.
\end{equation}
Since $\kappa(X,[\u,\u])=0$, $\chi_{\gamma}$ defines a character on
$U$. Let $M_{X}=\set{m\in M}{\mbox{$\Ad(m)X=X$}}$. Then it is known
\cite[Section 3.4]{CM92} that
\[
M_{X}=G_{\gamma}:=\set{g\in G}{\mbox{$\Ad(g)X=X$, $\Ad(g)Y=Y$, $\Ad(g)H=H$}}.
\]
In particular $M_{X}$ is reductive. For the moment assume that $\g_{-1}\neq 0$, or equivalently $\mathfrak{u}\subsetneq \mathfrak{n}$. In this case
$\ad(X)|_{\g_{-1}}:\g_{-1}\longrightarrow \g_{1}$ is an isomorphism,
and we may define a symplectic structure on $\g_{-1}$ by setting
\begin{equation}
\label{defsymg-1}
\kappa_{-1}(S,T)=\kappa(\ad(X)S,T)=\kappa(X,[S,T]), \qquad \mbox{for all  $S$, $T\in \g_{-1}$}.
\end{equation}
We may exhibit a canonical surjective group homomorphism from $N$ to the associated Heisenberg group $\H$ which maps $\exp Z$ to $\kappa(X,Z)$ in the center of $\H$, for $Z\in \u$. Then, according to the Stone-von
Neumann theorem, there exists a unique, up to equivalence, smooth
irreducible (unitarizable) representation
$(\rho_{\chi_{\gamma}},\S_{\chi_{\gamma}})$ of $N$ such that $U$
acts by the character $\chi_{\gamma}$. See Section \ref{subsec:GWM} for details. Since $M_{X}$ preserves
$\gamma $ and thus the symplectic form $\kappa_{-1}$, it is well-known \cite{Weil} that there exists a central
cover of $M_{X}$, to be denoted by $M_{\chi_{\gamma}}$, and a
representation of a semi-direct product $M_{\chi_{\gamma}}\ltimes N$
on $\S_{\chi_{\gamma}}$ which extends the representation
$\rho_{\chi_{\gamma}}$ of $N$. We refer to the representation of
$M_{\chi_{\gamma}}\ltimes N$ on $\S_{\chi_{\gamma}}$ as the smooth
oscillator-Heisenberg (or Weil) representation.
If $\g_{-1}=0$, then we define
$M_{\chi_{\gamma}}$ to be just $M_{X}$. For notational convenience, we also denote by $(\rho_{\chi_{\gamma}},\S_{\chi_{\gamma}})$ the $1$-dimensional
representation of $N=U$ given by the character $\chi_{\gamma}$. We may again view $(\rho_{\chi_{\gamma}},\S_{\chi_{\gamma}})$ as a representation of
$M_{X}\ltimes N$, with the trivial $M_{X}$ action.

In this article, a smooth representation of a reductive group over
$\k$ will mean a smooth representation in the usual sense for $\k$
non-Archimedean, namely it is locally constant, and a
Casselman-Wallach representation for $\k$ Archimedean. The reader
may consult \cite[Chapter 11]{Wa92} for more information about
Casselman-Wallach representations.

\begin{dfn} Fix an $\sl_{2}$-triple $\gamma=\{X,H,Y\}\subset \g$.
Let $(\pi,\mathscr{V})$ be a smooth representation of $G$. We define
the \emph{space of generalized Whittaker models of $\pi$ associated
to $\gamma $} to be
\begin{equation}
\label{defwhittaker}
\Wh_{\gamma}(\pi)=\Hom_{N}(\mathscr{V},\S_{\chi_{\gamma}}).
\end{equation}
Note that $\Wh_{\gamma}(\pi)$ is naturally an
$M_{\chi_{\gamma}}$-module:
\[(m\cdot \lambda )(v)=m\cdot \lambda (\bar{m}^{-1}\cdot v), \ \ m\in M_{\chi_{\gamma}}, \
\lambda \in \Hom_{N}(\mathscr{V},\S_{\chi_{\gamma}}), \ v\in \mathscr{V}, \]
where $m\mapsto \bar{m}$ is the covering map from $M_{\chi_{\gamma}}$ onto $M_{X}$.
\end{dfn}

From the well-known results of Jacobson-Morozov and Kostant \cite[Chapter 3]{CM92}, the map $\gamma=\{X,H,Y\}\mapsto \Orb =\Ad G
\cdot X$ yields a 1-1 correspondence between
\[
\left\{\begin{array}{c} \mbox{$\Ad G$ conjugacy classes of}\\
\mbox{$\sl_{2}$-triples in $\g$}\end{array}\right\}
\longleftrightarrow \left\{\begin{array}{c} \mbox{Nonzero nilpotent $\Ad G$-orbits}\\
\mbox{$\Orb \subset \g$}\end{array}\right\}.
\]
If the conjugacy class of an $\sl_{2}$-triple $\gamma$ corresponds
to a nilpotent orbit $\Orb\subset \g$  we say that $\gamma$ is an
\emph{$\sl_{2}$-triple of type $\Orb$.} It is clear that given two
conjugate $\sl_{2}$-triples $\gamma$, $\gamma '$, there exists an
isomorphism $\phi:\Wh_{\gamma}(\pi)\longrightarrow
\Wh_{\gamma '}(\pi)$ that intertwines the action of
$M_{\chi_{\gamma}}$ and $M_{\chi_{\gamma '}}$. For this reason,
sometimes we will abuse notation and denote
$\Wh_{\gamma}(\pi)$ just by $\Wh_{\Orb}(\pi)$ and we will
call it the \emph{space of generalized Whittaker models associated
to $\Orb$}, or the \emph{space of generalized Whittaker models of
type $\Orb$}. Note that this definition is (essentially) equivalent
to the one given in \cite{Ya86, MW87}.

The study of Whittaker and generalized Whittaker models for
representations of reductive groups over local fields evolved in
connection with the theory of automorphic forms (via their Fourier
coefficients), and has found important applications in both areas.
See for example \cite{Sh74,NPS73,Ko78,Ka85,Ya86, WaJI}, and a recent
article of Jiang \cite{Ji07} discussing its role in a general theory
of periods of automorphic forms. From the point of view of
representation theory, the space of generalized Whittaker models may
be viewed as one kind of nilpotent invariant associated to smooth
representations.
Another nilpotent invariant is the wave front cycle:
\[
\WF(\pi)=\sum_{\scriptstyle \begin{array}{c} \Orb\subset{\g}\\ \mbox{nilpotent}\end{array}} c_{\Orb}(\pi)[\Orb],
\]
defined by Harish-Chandra in the non-Archimedean case  (\cite{HC78}; see also \cite{Hei85,Pr91b}) and by Howe and Barbasch-Vogan in the Archimedean case (\cite{HoWF,BV80}; see also
\cite{Ro95,SV00}). For $\k$ non-Archimedean, M\oe{}glin and Waldspurger \cite{MW87} have established that $\WF(\pi)$ completely controls the spaces of generalized Whittaker models of interest, namely, if $\Orb$ is a nilpotent orbit which is maximal in the wave front set of $\pi$ \cite{HoWF}, then
\[
c_{\Orb}(\pi)=\dim \Wh_{\Orb}(\pi).
\]
For $\k$ Archimedean, the corresponding phenomenon is not yet (fully) understood, except for the representations with the largest Gelfand-Kirillov dimension \cite{Vo78,Ma92} and unitary highest weight modules \cite{Ya01}. For the latter, the wave front set was computed earlier in \cite{Pr91a}.

On the other hand, when $\k$ is Archimedean, there is another nilpotent invariant defined by Vogan \cite{Vo91}, which is called the associated cycle:
\[
\AC(\pi) =\sum_{\scriptstyle \begin{array}{c} \Orb\subset{\mathfrak{s}}\\ \mbox{nilpotent}\end{array}} d_{\Orb}(\pi)[\Orb].
\]
Here $\g_{\C}=\mathfrak{k}\oplus
\mathfrak{s}$ is the complexified Cartan decomposition. By the work of Schmid and Vilonen \cite{SV00}, we know that the wave front cycle and the associate cycle determine
each other through the Kostant-Sekiguchi correspondence \cite{Se87}.

\vsp

Now we come to the other part of the title regarding local theta lifting. Let $(G,\tilde{G})$ be a reductive dual pair in
the sense of Howe \cite{Ho79}. Let $(\omega , \Y )$ be the smooth oscillator representation
associated to the dual pair $(G,\tilde{G})$ and to the character
$\psi$ of $\k$. It is well-known that $(\omega , \Y)$ yields a representation
of $G\times \tilde{G}$ except when $G$ is the isometry group of an
odd dimensional quadratic space, in which case it is a representation of a double cover of $\tilde{G}$. (In general one may need to tensor with a genuine character of the relevant double cover to achieve this.) In the exceptional case, we shall simply redefine $\tilde{G}$ to be the induced double
cover, so as to simplify notation. Also in this case we shall only consider genuine representations of $\tilde{G}$,
namely those representations such that the non-trivial element of the covering map acts by $-1$. For all other cases, representations are called genuine by convention.

Let $(\pi,\mathscr{V})$ be a smooth irreducible genuine representation of
$G$, and consider the maximal $\pi$-isotypic quotient of
$\Y$, called the Howe quotient of $\pi $. If it is non-zero, it is of the form
$\mathscr{V} \otimes \Theta(\mathscr{V})$ (algebraic tensor product for $\k$ non-Archimedean, and
completed projective tensor product for $\k$ Archimedean), where $\Theta(\mathscr{V})$ carries a smooth
representation $\Theta(\pi)$ of $\tilde{G}$. The representation $\Theta(\pi)$ is sometimes referred to as the
full theta lift (or colloquially the big theta lift) of $\pi$. Results of Howe \cite{Ho89} and Waldspurger \cite{Wald} say that $\Theta(\pi)$
is an admissible representation of finite length, moreover, with the possible exception when $\k$ is a dyadic field,
$\Theta(\pi)$ has a unique irreducible quotient $\theta(\pi)$ called the (local) theta
lift of $\pi$. In this article we shall focus on the full theta lift $\Theta(\pi)$, and shall thus place no restriction
on the residue characteristic of $\k$.

Let $V$, $\tilde{V}$, be the standard modules of $G$, $\tilde{G}$,
respectively; and consider the moment maps $\varphi$,
$\tilde{\varphi}$, from $\Hom(V,\tilde{V})$ to $\g$ and
$\tilde{\g}$, respectively:
\[
\begin{diagram}
 & &  \Hom(V,\tilde{V}) & &   \\
&  \ldTo^{\varphi} & &\rdTo^{\tilde{\varphi}} & \\
\g&  & & &\tilde{\g}
\end{diagram}
\]
See Section \ref{subsec:liftOrbit} for the definition of the moment maps. If $T\in \Hom(V,\tilde{V})$, then $\varphi (T)$ is nilpotent if and
only if $\tilde{\varphi}(T)$ is nilpotent. As usual, this yields a
notion of correspondence for nilpotent orbits $\Orb \subset \g$ and
$\tilde{\Orb}\subset \tilde{\g}$. (In general this may not be one to
one.) Denote by $\Max \Hom (V,\tilde{V})$ the set of full rank
elements in $\Hom(V,\tilde{V})$. Without any loss of generality, we
assume that $\dim V \leq \dim \tilde{V}$, and elements of $\Max \Hom
(V,\tilde{V})$ are then represented by injective maps from $V$ to
$\tilde{V}$. In this article, we shall only be concerned with
nilpotent orbits $\Orb \subset \g$ in the image of $\Max \Hom
(V,\tilde{V})$ under the moment map $\varphi$, namely it satisfies
\begin{equation}
\label{assumption0}
\varphi ^{-1}(\Orb) \cap \Max \Hom (V,\tilde{V})\ne \emptyset.
\end{equation}
This will be our standing assumption. We write the resulting
nilpotent orbit correspondence as $\Orb \mapsto
\tilde{\Orb}=\Theta(\Orb)$. See Section \ref{subsec:liftOrbit} for
details.

As noted in the beginning, the space of generalized Whittaker
models depends on an $\sl_{2}$-triple. To relate the two
$\sl_{2}$-triples in $\g$ and $\tilde{\g}$, we make the following

\begin{dfn}
Let $\gamma=\{X,H,Y\}\subset \g$ and $\tilde{\gamma}=\{\tilde{X},\tilde{H},\tilde{Y}\}\subset \tilde{\g}$ be two $\sl_{2}$-triples of type $\Orb$ and $\tilde{\Orb}$, respectively. We say that $T\in \Hom(V,\tilde{V})$ \emph{lifts} $\gamma$ to $\tilde{\gamma}$ if $\varphi(T)=X$, $\tilde{\varphi}(T)=\tilde{X}$ and $T(V_{j})\subset \tilde{V}_{j+1}$ for all $j$. Here
$V_{j}=\{v\in V\, | \, Hv=jv\}$, and likewise for $\tilde{V}_{j+1}$.
\end{dfn}

Set
\[
\Orb_{\gamma,\tilde{\gamma}}=\{T\in \Hom(V,\tilde{V})\, | \, \mbox{$T$ lifts $\gamma$ to $\tilde{\gamma}$}\},\]
and
\begin{equation}
\label{defoggMax0}
\Orb^{\Max}_{\gamma,\tilde{\gamma}}=\Orb_{\gamma,\tilde{\gamma}}\cap
\Max \Hom (V,\tilde{V}).
\end{equation}
(By Lemma \ref{lem:oggMax}, this is a single $M_{X}\times \tilde{M}_{\tilde{X}}$-orbit.)

Our main result, in a slightly less concrete form than what will be proved in Section \ref{liftWhittaker}, is

\begin{theorem}
\label{MainThm}
Let $(G,\tilde{G})$ be a reductive dual pair, and let $(\omega , \Y )$ be the smooth oscillator representation
associated to the dual pair $(G,\tilde{G})$ and to the character
$\psi$ of $\k$. Let $(\pi,\mathscr{V})$ be a smooth
irreducible genuine representation of $G$. Let $\Orb\subset \g$ be a
nilpotent $G$-orbit in the image of $\Max \Hom (V,\tilde{V})$ under the moment map $\varphi$ and let $\tilde{\Orb}=\Theta(\Orb)\subset \tilde{\g}$ be the corresponding
nilpotent $\tilde{G}$-orbit. Then
\[
\Wh_{\tilde{\Orb}}(\Theta(\pi))\cong \Wh_{\Orb}(\check{\pi}),
\]
where $\check{\pi}$ is the contragredient representation of $\pi$. More precisely, let $\gamma$ and $\tilde{\gamma}$ be two $\mathfrak{sl}_{2}$-triples of type $\Orb$ and $\tilde{\Orb}$, respectively. Then, given any $T=T_{\gamma,\tilde{\gamma}}\in \Orb^{\Max}_{\gamma,\tilde{\gamma}}$, there is a surjective homomorphism $\phi _T:\widetilde{M}_{\chi_{\tilde{\gamma}}} \twoheadrightarrow M_{\chi_{\gamma}}$ (depending on $T$), inducing an action of $\widetilde{M}_{\chi_{\tilde{\gamma}}}$ on $ \Wh_{\gamma}(\check{\pi})$, such that
\[
\Wh_{\tilde{\gamma}}(\Theta(\pi))\cong \Wh_{\gamma}(\check{\pi})
\]
as  $\widetilde{M}_{\chi_{\tilde{\gamma}}}$-modules.
\end{theorem}

\noindent {\bf Remarks}:  (a) Nilpotent orbits in $\g$ may be
parameterized by equivalence classes of admissible
$\epsilon$-Hermitian Young tableaux (Section
\ref{sub:Young-tableaux}). Using this parametrization, the
correspondence $\Orb \mapsto \Theta(\Orb)$ can be described
explicitly and in simple terms. Note that the assumption in
(\ref{assumption0}) allows us to ``embed" the sesquilinear Young
tableau parameterizing $\Orb $ into the sesquilinear Young tableau
parameterizing $\Theta(\Orb)$. See Section \ref{subsec:liftOrbit}.

\noindent (b) One particularly important circumstance is when $(G,\tilde{G})$
is in the stable range with $G$ the smaller member (\cite{HoST}; cf.
\cite{Li89}). Then every nilpotent orbit $\Orb \subset \g$ is in the
image of $\Max \Hom (V,\tilde{V})$ under the moment map $\varphi$,
and for any smooth irreducible genuine representation
$(\pi,\mathscr{V})$ of $G$, we have $\Theta(\pi)\ne 0$
\cite{MVW,PP}. In this case and for $\k$ non-Archimedean, Theorem
\ref{MainThm} is due to M\oe{}glin \cite{Mo98}. Our approach,
explained in the following paragraph, is in some sense more
conceptual. For the type of questions considered in this article, it
is also well-known that substantially more effort is required to
treat the Archimedean case.

\noindent (c) Assume that we are in the stable
range and $\k$ is Archimedean. There is a similar notion of theta
lifting of  nilpotent $K_{\C}$-orbits and one similarly expects a
correspondence of associate cycles \cite{NOT}. We refer the reader to
\cite{Ya01, NZ04} for some results in this direction, and the
definitive result in the recent paper of Loke and Ma \cite{LM13}. We
also refer to several earlier works of Przebinda \cite{Pr91a,Pr93,DP96,Pr00} on
correspondence of wave front sets.

\noindent (d) Generically $\Theta (\pi)$
should be irreducible. One then obtains the space of generalized
Whittaker models for $\theta (\pi)$.  For example it is expected
that under the assumption of stable range, $\Theta (\pi)$ is
irreducible whenever $\pi $ is unitarizable. For $\k$ Archimedean,
this is established in \cite{LM13}.

\noindent (e) As a direct consequence of Theorem \ref{MainThm}, we
have $\Theta(\pi)\ne 0$, if $\Wh_{\Orb}(\check{\pi})\ne 0$, for some
$\Orb$ in the image of $\Max \Hom (V,\tilde{V})$ under the moment map $\varphi$.

\noindent (f) For ``small" $\Orb$, the condition in (\ref{assumption0}) is actually quite restrictive. For example for the zero orbit, the condition implies that the dual pair $(G,\tilde{G})$ is in the stable range with $G$ the smaller member. For this reason there is a need to consider orbit correspondence covered by a $G\times \tilde{G}$ stable set larger than $\Max \Hom (V,\tilde{V})$, and to investigate how their generalized Whittaker models behave. This will be taken up in a forthcoming work of the authors.

\vsp

The main ingredient in the proof of Theorem \ref{MainThm} is our
description of the space of covariants $\Y_{\tilde{U},
\chi_{\tilde{\gamma}}}:=\Hom_{\tilde{U}}(\Y, \chi_{\tilde{\gamma}})$ of the smooth oscillator representation
$\Y$, where $(\tilde{U}, \chi_{\tilde{\gamma}})$ for the
$\sl_{2}$-triple $\tilde{\gamma}$ in $\tilde{\g}$ is as
$(U,\chi_{\gamma})$ for the $\sl_{2}$-triple $\gamma$, and
$\tilde{\gamma}$ is a lift of $\gamma$. (This is reminiscent of the computation of
the Jacquet modules of the smooth oscillator representation by Kudla \cite{Ku86}, though not in method.)
We show it is isomorphic to
a certain space $\Sch(N\backslash G;\HH_{\gamma, \tilde{\gamma}})$ of
rapidly decreasing functions on $N\backslash G$ with values in
$\HH_{\gamma, \tilde{\gamma}}$, where $\HH_{\gamma, \tilde{\gamma}}$ is the
space of the smooth oscillator-Heisenberg representation associated
to a certain symplectic subspace $W_{\gamma, \tilde{\gamma}}$ of
$\Hom (V,\tilde{V})$.  This is the key Proposition
\ref{prop:tildechicoinvariants}. An important point is that we may
construct an isomorphism of $\g_{-1}\oplus \tilde{\g}_{-1}$ with
$W_{\gamma, \tilde{\gamma}}$, which depends on $T\in
\Orb^{\Max}_{\gamma,\tilde{\gamma}}$.  The key proposition basically
says that we can define a natural surjective (matrix coefficient)
map (using ``homogeneous components" of $T$), from $\Y$ to
$\Sch(N\backslash G;\HH_{\gamma, \tilde{\gamma}})$, which induces an
isomorphism on $\Y_{\tilde{U},\chi_{\tilde{\gamma}}}$!

\vsp

Here are some additional words on the organization and contents of
this article. In Section \ref{classical}, we describe classical
groups as the isometry groups of $\epsilon$-Hermitian $D$-modules,
where $D$ is a division algebra over $\k$. In Section \ref{orbit},
we review the well-known parametrization of nilpotent orbits in the
classical Lie algebras, following the book by Collingwood and
McGovern \cite{CM92}. We also introduce the generalized Whittaker
models associated to the nilpotent orbits. In Section \ref{frechet},
we introduce some Fr\'echet spaces of functions on $G$ to realize
the space of generalized Whittaker models for $\k$ Archimedean. In
Section \ref{sec:liftOrbit}, we recall the notion of lifting of
nilpotent orbits in the setting of dual pairs (via the moment maps),
and we describe the fine structure of lifting for those orbits
$\Orb$ and $\tilde{\Orb}$ which correspond via injective maps from
$V$ to $\tilde{V}$. We also prove an explicit isomorphism of $\g_{-1}\oplus
\tilde{\g}_{-1}$ with a symplectic subspace $W_{\gamma,
\tilde{\gamma}}$ of $\Hom(V,\tilde{V})$. In Section
\ref{liftWhittaker}, we relate the generalized Whittaker models of
$\check{\pi}$ and $\Theta (\pi)$. As mentioned previously, its main
ingredient is the description of the space of covariants
$\Y_{\tilde{U}, \chi_{\tilde{\gamma}}}$ of the smooth oscillator
representation $\Y$. To arrive at this, we make extensive use of the
gradation in the standard modules $V$ and $\tilde{V}$ given by the
semisimple elements $H$, $\tilde{H}$ of the two $\sl_{2}$-triples
$\gamma\subset \g$ and $\tilde{\gamma}\subset\tilde{\g}$ of type
$\Orb$ and $\Theta(\Orb)$, respectively. On the one hand, it gives
rise to totally isotropic subspaces and thus convenient realizations
of $\Y$. On the other hand, it facilitates an inductive argument
based on the heights of the gradations. Together with the
isomorphism of $\g_{-1}\oplus \tilde{\g}_{-1}$ with $W_{\gamma,
\tilde{\gamma}}$, this implies the relationship between the
generalized Whittaker models of $\check{\pi}$ and $\Theta(\pi)$.

\vsp

\noindent {\bf Acknowledgements}: the authors thank W. T. Gan, D. Jiang, and B. Sun for their interests and comments. The authors also wish to express their gratitude to the anonymous referee for his critical comments as well as numerous detailed suggestions.

\section{Classical groups as isometry groups of $\varepsilon$-Hermitian modules}
\label{classical}

\subsection{Hermitian $D$-modules}
\label{Hermitianmodules}
Let $\k$ be a local field, $|\cdot|$  its absolute
value, and let $\psi$ be a fix non-trivial unitary character of $\k$. Let $D$ be one of the following division algebras over $\k$:
the field $\k$ itself, a quadratic field extension of $\k$ or the
quaternion division $\k$-algebra. Observe that $D$ comes equipped
with a canonical involutive anti-automorphism (the identity map, the non-trivial Galois element, or
the main involution, respectively) which we will denote by $x \mapsto
\overline{x}$.
Throughout this article, we will only consider finitely generated
modules over $D$.

Let $V$ and $W$ be two right $D$-modules. We will denote the set of
right $D$-module morphisms from $V$ to $W$ by
\[
 \Hom_{D}(V,W)=\{T:V\longrightarrow W \,| \, \mbox{$T(v_{1}a+v_{2}b)=T(v_{1})a+T(v_{2})b$ for all $v_{1}$, $v_{2} \in V$, $a$, $b\in D$}\}.
\]
If $V=W$, we will denote this set by $\End_{D}(V)$. Set
\[
 \GL(V,D)=\{T\in \End_{D}(V) \, | \, \mbox{$T$ is invertible}\}.
\]
When it is clear from the context what the division algebra $D$ is,
we may just omit $D$ in various of these notations.

 Let $V'$ be the set of right $D$-linear functionals on $V$. There is a natural left
 $D$-module structure on $V'$ given by setting
 \[  (a\lambda)(v)=a\lambda(v), \quad \text{for all $a\in D$, $v\in V$, and
 $\lambda \in V'$.} \]
Observe that with this structure, $W\otimes_{D}V'$ is naturally
isomorphic to $\Hom_{D}(V,W)$ as a $\k$-vector space. Given $T\in
\Hom_{D}(V,W)$, we will specify an element in $\Hom_{D}(W',V')$
(analogously defined), which we will also denote $T$, by setting
$(\lambda T)(v):=\lambda(Tv)$, for $v\in V$ and $\lambda \in W'$. This correspondence gives rise to
natural isomorphisms between $\End_{D}(V)$ and $\End_{D}(V')$, and
between $\GL_{D}(V)$ and $\GL_{D}(V')$.

\begin{definition}
Let $\varepsilon=\pm 1$. We say that $(V,B)$ is a right $\varepsilon$-Hermitian $D$-module if $V$ is a right $D$-module and $B$ is an $\varepsilon$-Hermitian form, i.e., $B:V\times V \longrightarrow D$ is a map such that
\begin{enumerate}
 \item $B$ is \emph{sesquilinear}: for all $v_{1}$, $v_{2}$, $v_{3}\in V$, $a$, $b\in D$,
  \[ B(v_{1},v_{2}a+v_{3}b)=B(v_{1},v_{2})a+B(v_{1},v_{3})b \quad \text{and} \quad B(v_{1}a + v_{2}b,v_{3})=\overline{a}B(v_{1},v_{3})+\overline{b}B(v_{2},v_{3}).\]
\item $B$ is $\varepsilon$-\emph{Hermitian}:
$$
B(v,w)=\varepsilon\overline{B(w,v)} \qquad \mbox{ for all $v,w\in V$.}
$$
\item $B$ is \emph{non-degenerate}.
\end{enumerate}
\end{definition}

Given a right $\varepsilon$-Hermitian $D$-module $(V,B)$, we may
construct a left $\varepsilon$-Hermitian $D$-module
$(V^{\ast},B^{\ast})$ in the following way: as a set, $V^{\ast}$
will be the set of symbols $\{v^{\ast}\, |\, v\in V\}$. Then we give to
$V^{\ast}$ a left $D$-module structure by setting, for all $v$, $w
\in V$, $a\in D$,
\[  \text{$v^{\ast}+w^{\ast}=(v+w)^{\ast}$ and $av^{\ast}=(v\overline{a})^{\ast}$.} \]
 Finally, we set
\[
 B^{\ast}(v^{\ast},w^{\ast})=\overline{B(w,v)} \qquad \mbox{for all $v$, $w\in V$.}
\]
 In an analogous way, we may define the $*$ operation on left $\varepsilon$-Hermitian $D$-modules. Then $V^{\ast\ast}$ is naturally isomorphic with $V$.
 Given $T\in \End_{D}(V)$, we define $T^{\ast}\in \End_{D}(V^{\ast})$ by setting $v^{\ast}T^{\ast}:=(Tv)^{\ast}$. With this definition, it is easily seen that $(TS)^{\ast}=S^{\ast}T^{\ast}$, for all $S$, $T\in \End_{D}(V)$. Therefore the map $g\mapsto (g^{\ast})^{-1}$ defines a group isomorphism between $\GL_{D}(V)$ and $\GL_{D}(V^{\ast})$.

Observe that the form $B$ induces a left $D$-module isomorphism
$B^{\flat}:V^{\ast}\longrightarrow V'$ given by
$B^{\flat}(v^{\ast})(w)=B(v,w)$ for $v$, $w\in V$. In what follows,
we will make implicit use of this map to identify these two spaces.
With this identification, for any $T\in \End_{D}(V)$, we can think of $T^{\ast}$ as an element in
$\End_{D}(V)$ defined by
$v^{\ast}(T^{\ast}w):=(v^{\ast}T^{\ast})(w)$, i.e., $T^{\ast}$ is
defined by the usual condition that
\[
 B(v,T^{\ast}w)=B(Tv,w) \qquad \mbox{for all $v$, $w \in V$}.
\]

A $D$-submodule $E\subset V$ is said to be \emph{totally isotropic}
if $B|_{E\times E}=0$. If $E$ is a totally isotropic submodule, then
there exists a totally isotropic submodule $F\subset V$ such that
$B|_{(E\oplus F)\times (E\oplus F)}$ is non-degenerate.  If we set
\[
 U=(E\oplus F)^{\perp}:=\{u\in V\, | \, \mbox{$B(u,w)=0$ for all $w\in E\oplus F$}\},
\]
then $V=E\oplus F \oplus U$, and $B|_{U\times U}$ is non-degenerate.
In this case we say that $E$ and $F$ are totally isotropic,
\emph{complementary} submodules. Observe that then
$B^{\flat}|_{F^{\ast}}:F^{\ast}\longrightarrow E'$ is an
isomorphism. As before we will make implicit use of this isomorphism
to identify $F^{\ast}$ with $E'$.

\subsection{Isometry groups}
\label{Isometry}

Given a right $\varepsilon$-Hermitian $D$-module $(V,B)$, we define
its isometry group
\[
G(V,B)=\{g\in \GL(V) \, | \, \mbox{$B(g v,g w)=B(v,w)$ for all $v$,
$w\in V$}\}.
\]
When there is no risk of confusion regarding $B$, we will denote
this group just by $G(V)$ or even just as $G$. Observe that if $g\in
G$, then $g^{*}=g^{-1}$.

Associated to the group $\GL(V)$ we have the Lie algebra
$\mathfrak{gl}(V)=\End(V)$ with bracket $[T,S]:=TS-ST$, for all
$T$, $S \in \mathfrak{gl}(V)$. Similarly, associated to the group
$G$ we have the Lie algebra
\begin{eqnarray*}
\g &=&\{T\in \mathfrak{gl}(V) \, | \, \mbox{$B(T v,w) + B(v,Tw)=0$ for all $v$, $w\in V$}\}\\
   &=&\{T\in \mathfrak{gl}(V) \, | \, T^*=-T\}.
\end{eqnarray*}
We define a bilinear form on $\g$ by
\begin{equation}
\label{knormalized}
\kappa(T,S)=\Tr(T^{\ast}S)/2,
\end{equation}
for all $T$, $S \in \g$. Here $\Tr(T^{\ast}S)$ is the trace
 of $T^{\ast}S$ as a $\k$-linear transformation. This bilinear form is non-degenerate
 and \emph{invariant} with respect to
 the adjoint action of $G$ on $\g$.

\begin{remark} In the literature, the isometry groups
$G(V,B)$ are called type I. Given a division algebra $D$, one may
consider the algebra $\mathbb{D}=D\oplus D$, which has the canonical
involution $\overline{(a,b)}=(b,a)$, for $a$, $b\in D$. Then one may
similarly define right $\epsilon $-Hermitian $\mathbb{D}$-modules
($\epsilon =\pm 1$) and their isometry groups which are easily seen
to be isomorphic to $\GL(V,D)$, where $V$ are right $D$-modules.
They are called type II. With this modification (from $D$ to
$\mathbb{D}$), all results in this article, which are stated for type I dual pairs,
are expected to carry over to the case of type II dual pairs. We leave this to the interested readers.
\end{remark}

\section{Nilpotent orbits of classical groups}
\label{orbit}

\subsection{Nilpotent orbits and $\sl_{2}$-triples} \label{subsection:nilpotentorbits}
In order to set up notation, and facilitate the exposition of the
reminder of this article, we will review the basic structural results on nilpotent orbits and $\sl_{2}$-triples in $\g$. The exposition given here is based on the
book of Collingwood and McGovern \cite{CM92}. Another standard reference is the book of Carter \cite[Chapter 5]{Ca85}.

Let $X\in \g$ be a nonzero nilpotent element. Then, by the Jacobson-Morozov
theorem, there exists an $\sl_{2}$-triple
$\gamma=\{H,X,Y\} \subset \g$, containing $X$ as the nilpositive element, namely
\[
 [H,X]=2X, \qquad [H,Y]=-2Y, \qquad [X,Y]=H.
\]
Let $\g_{i}=\{Z \in \g \, | \,
\ad(H)(Z)=iZ\}$, for $i \in \Z$. Then, from standard $\sl_{2}$-theory, we have that
\begin{equation}
\g=\bigoplus_{i\in \Z}\g_{i}.
\end{equation}
Now let
\[
 \p=\bigoplus_{i \leq 0}\g_{i}, \qquad  \overline{\p}=\bigoplus_{i \geq 0}\g_{i},
\]
and set
\[
P=\{g\in G \, | \, \Ad(g)\p\subset \p\},\qquad    \overline{P}=\{g\in G \, | \, \Ad(g)\overline{\p}\subset \overline{\p}\}.
\]
 $\overline{P}$ (resp.\ $\overline{\p}$) is called the Jacobson-Morozov parabolic subgroup (resp. parabolic subalgebra) associated to $X$. Observe that $P$ (resp.\ $\p$) is a parabolic subgroup  opposite to $\overline{P}$ (resp. parabolic subalgebra opposite to $\overline{\p}$). (The subgroup $\overline{P}$ depends only on $X$, but the subgroup $P$ depends on the choice of $\sl_{2}$-triple $\gamma$ containing $X$.) Let $M=\{m\in G\, | \, \Ad(m)H=H\}$. Its Lie algebra is $\m =\{Z\in \g\, | \, \ad(Z)H=0\}$, which is exactly $\g _{0}$. Let
\[
\n=\bigoplus_{i\leq -1} \g_{i} \qquad \mbox{and} \qquad \u=\bigoplus_{i\leq -2} \g_{i}.
\]
Observe that if $Z\in \n$, then $Z$ is nilpotent and hence $\exp Z=\sum_{j=0}^{\infty}Z^{j}/j!$ is a well defined element in $\End(V)$. Let $N=\exp \n=\set{\exp Z}{Z\in \n}$ and $U=\exp \u=\set{\exp Z}{Z\in \u}$; then $U$,  $N$ are subgroups of $G$, and $P=MN$. Similarly we have the Levi decomposition $\overline{P}=M\overline{N}$, with $\overline{N}$ the unipotent radical of $\overline{P}$.

For $i\in \Z$, let $V_{i}=\{v\in V\, | \, Hv=iv\}$. Then, again from standard $\sl_{2}$-theory,
\begin{equation}
V=\bigoplus_{i\in \Z} V_{i}.   \label{eq:Vdirectsumdecomposition}
\end{equation}
We may also characterize $M$ as the set of $m\in G$ that preserves
the direct sum decomposition given in
(\ref{eq:Vdirectsumdecomposition}). Now observe that, since
$H^{*}=-H$, $B|_{V_{0}\times V_{0}}$ is non-degenerate and $B$
establishes a perfect pairing between $V_{i}$ and $V_{-i}$ for all
$i > 0$. Using this pairing we define for any $T\in \End(V_{i})$
a map $T^{*}\in \End(V_{-i})$ given by
 \[
 B(Tv,w)=B(v,T^{*}w), \qquad \mbox{for all $v\in V_{i}$, $w\in V_{-i}$}.
 \]
It follows immediately that, if $g\in \GL(V_{i})$, then
$B(gv,(g^{*})^{-1}w)=B(v,w)$, for all $v\in V_{i}$, $w\in V_{-i}$.
From this we conclude that there is an embedding of $\GL(V_{i})$
into $M$ for all $i >0$. We will denote the image of this embedding
by $M_{i}$. Proceeding in a similar manner we may also define a
natural embedding of $G(V_{0},B|_{V_{0}\times V_{0}})$ into $M$,
whose image we will denote by $M_{0}$. Then
\begin{equation}
M = \prod_{i \geq 0} M_{i} \cong G(V_0)\times \prod_{i>0} \GL(V_{i}). \label{eq:descriptionofM}
\end{equation}

Set $\g_{\gamma}=\Span_{\k}\{X,H,Y\}\subset \g$. We have the $\g_{\gamma}$-isotypic decomposition
\begin{equation}
\label{isotypic}
 V=\bigoplus_{j=1}^{l} V^{\gamma,t_{j}},
\end{equation}
where $V^{\gamma,t_{j}}$ is a direct sum of irreducible
$t_{j}$-dimensional $\g_{\gamma}$-modules, and $t_{1}> t_{2}> \ldots > t_{l}>0$. Let
$V^{\gamma,t_{j}}_{i}=V^{\gamma,t_{j}}\cap V_{i}$. It is nonzero if and only if $|i|<t_{j}$ and $i\equiv t_j-1$ ($\mbox{mod } 2$).
Then it is clear that
\begin{equation}
V_{i}=\bigoplus_{j,\, t_{j}>|i|} V^{\gamma,t_{j}}_{i}. \label{eq:inducedhermitianform}
\end{equation}
Using standard results from the representation theory of
$\sl_{2}$, we have that for all $i\geq 0$, the map
$(X|_{V_{-i}})^{i}:V_{-i}\longrightarrow V_{i}$ is invertible. (Here
we are using the convention that $(X|_{V_{0}})^{0}$ is the identity
map on $V_{0}$.) The statement remains true for $i<0$ if we
interpret a negative power of an invertible map as the positive
power of its inverse.

Now define a Hermitian form $B_{i}$ on $V_{i}$ by setting $B_{i}(v,w)=B(X|_{V_{i}}^{-i}v,w)$ for all $v$,
$w\in V_{i}$. Since $B$ establishes a perfect pairing between
$V_{i}$ and $V_{-i}$, it is clear that $B_{i}$ is non-degenerate.
The analysis applied to $(V,B)$ applies equally to $(V^{\gamma,t_{j}},B^{\gamma,t_{j}})$, where $B^{\gamma,t_{j}}$ is the restriction of $B$ to $V^{\gamma,t_{j}}$.
Thus $B_{i}$ is in fact non-degenerate when restricted to any
$V^{\gamma,t_{j}}_{i}$. We will denote by $B_{i}^{\gamma,t_{j}}$ the restriction of $B_{i}$ to $V^{\gamma,t_{j}}_{i}$.
The following result is also clear:
\begin{equation}
\label{eq:XfromVitoViplustwo}
\mbox{If $-t_{j}+1\leq i <i+2 \leq t_{j}-1$,} \qquad \mbox{then $B_{i+2}(Xv,Xw)=-B_{i}(v,w)$} \qquad \mbox{for all $v$, $w\in V^{\gamma,t_{j}}_{i}$}.
\end{equation}
In particular, all the $B_{i}^{\gamma,t_{j}}$'s are determined by $B_{t_{j}-1}^{\gamma,t_{j}}$ (on $V_{t_{j}-1}^{\gamma,t_{j}}$, the space of highest weight vectors in $V^{\gamma,t_{j}}$).

Let $M_{X}=\{m\in M\,| \, \Ad(m)X=X\}$. The following result is due to Kostant. See \cite[Section 3.4]{CM92} or \cite[Proposition 5.5.9]{Ca85}.

\begin{thm}\label{thm:kostant} Let $G_{X}=\{g\in G\, | \, \mbox{$\Ad(g)X=X$}\}$. Then
\[
G_{X}=M_{X} \overline{N}_{X},
\]
where $\overline{N}_{X}=\{n\in \overline{N}\, | \,
\Ad(n)X=X\}$. Furthermore
\[
M_{X} =G_{\gamma}:=\{g\in G\, | \,  \mbox{$\Ad(g)X=X$, $\Ad(g)H=H$,
 $\Ad(g)Y=Y$}\}.
\]
\end{thm}

We now have all the ingredients to give a description of $M_{X}$ in
the spirit of the description of $M$ given in equation
(\ref{eq:descriptionofM}). Observe that $M_{X}$ acts on
$V^{\gamma,t_{j}}_{t_{j}-1}$ preserving $B^{\gamma,t_{j}}_{t_{j}-1}$. From this observation and
equation (\ref{eq:XfromVitoViplustwo}) we conclude that there is an
embedding of $G(V^{\gamma,t_{j}}_{t_{j}-1},B^{\gamma,t_{j}}_{t_{j}-1})$ into $M_{X}$. Let
$M_{X,t_{j}-1}$ be the image of this embedding. Then
\begin{equation}\label{eq:productisometrygroups}
M_{X}=\prod_{j=1}^{l}M_{X,t_{j}-1}\cong \prod_{j=1}^{l} G(V^{\gamma,t_{j}}_{t_{j}-1}).
\end{equation}

\begin{remark}
\label{rmkzero}
For $X=0$, we may take $\g_{0}=\g$, and $\g_{i}=0$ for $i\ne 0$. With this convention, all the subgroups defined in this section will make sense for any $X$. We will adopt this convention in the sequel, sometimes without mentioning the appropriate (minor) adjustment for the special case of the zero orbit.
\end{remark}

\subsection{$\epsilon$-Hermitian Young tableaux}
\label{sub:Young-tableaux}
A partition of $n$ is a tuple
 $[d_{1},\ldots,d_{k}]$ of positive integers such that
 $d_{1}\geq d_{2} \geq \ldots \geq d_{k}$ and $d_{1}+\cdots+d_{k}=n$. Partitions of
 $n$ are frequently represented by \emph{Young diagrams} in the following way:
 given a partition $\mathbf{d}=[d_{1},\ldots,d_{k}]$ we construct a left-justified
 array of empty boxes such that the $i$-th row has $d_{i}$ boxes. This array is the
 Young diagram associated to $\mathbf{d}$. Another way of describing a partition $\mathbf{d}=[d_{1},\ldots,d_{k}]$ of $n$ is using the \emph{exponential notation} $\mathbf{d}=[t_{1}^{i_{1}},\ldots,t_{l}^{i_{l}}]$,
 which means that the number $t_{k}$ appears $i_{k}$ times (the multiplicity) in the
 partition, and $t_1>\cdots >t_l>0$.

\begin{definition} A \emph{sesquilinear Young tableau} is a pair
$\Gamma=(\mathbf{d}^{\Gamma},(V^{\Gamma},B^{\Gamma}))$ such that
$\mathbf{d}^{\Gamma}=[t_{1}^{i_{1}},\ldots,t_{l}^{i_{l}}]$ is a partition
of $n$, and $(V^{\Gamma},B^{\Gamma})$ is an assignment, for each $1 \leq j\leq
l$, of an $\epsilon_{j}$-Hermitian module
$(V^{\Gamma,t_{j}}_{t_{j}-1},B^{\Gamma,t_{j}}_{t_{j}-1})$ of dimension ${i_{j}}$.
\end{definition}

\begin{example} The following picture represents a sesquilinear Young tableau $\Gamma=(\mathbf{d}^{\Gamma},(V^{\Gamma},B^{\Gamma}))$, where $\mathbf{d}^{\Gamma}=[3^{2},2^{3},1^{2}]$ is a partition of $14$, and $V_{2}^{\Gamma,3}$, $V_{1}^{\Gamma,2}$, $V_{0}^{\Gamma,1}$ have dimensions $2$, $3$ and $2$, respectively.
\[
\raisebox{35pt}{$\begin{array}{l} \raisebox{9pt}{$ (V_{2}^{\Gamma,3},B_{2}^{\Gamma,3})\left\{ \raisebox{15pt}[12pt][0pt]{} \right.$} \\ \raisebox{15pt}{$ (V_{1}^{\Gamma,2},B_{1}^{\Gamma,2}) \left\{ \raisebox{0pt}[14pt][12pt]{} \right.$} \\\raisebox{10pt}{$( V_{0}^{\Gamma,1},B_{0}^{\Gamma,1}) \left\{ \raisebox{0pt}[12pt][-10pt]{} \right.$} \end{array}$} \yng(3,3,2,2,2,1,1)
\]
\end{example}
 Let $(\rho_{m},\k^{m})$ be the irreducible representation of $\sl_{2}$ of
 dimension $m$, and fix, for each $m$, a non-degenerate invariant bilinear form $(\cdot,\cdot)_{m}$
 on $\k^{m}$. Recall that this form is unique up to scalar and that
 $(f_{1},f_{2})_{m}=(-1)^{m-1}(f_{2},f_{1})_{m}$, for all $f_{1}$, $f_{2} \in \k^{m}$. This follows from the analysis surrounding equation (\ref{eq:XfromVitoViplustwo}).

 Given a sesquilinear Young tableau $\Gamma=(\mathbf{d}^{\Gamma},(V^{\Gamma},B^{\Gamma}))$ we will set
 \begin{equation}
 V^{\Gamma,t_{j}}:=V^{\Gamma,t_{j}}_{t_{j}-1}\otimes_{\k}\k^{t_{j}}.
 \end{equation}
 On this space we define a $(-1)^{t_{j}-1}\epsilon_{j}$-Hermitian form $B^{\Gamma,t_{j}}$ by
\begin{equation}
\label{Bgt}
B^{\Gamma,t_{j}}(v_{1}\otimes f_{1}, v_{2}\otimes f_{2})=B^{\Gamma,t_{j}}_{t_{j}-1}(v_{1},v_{2})
 (f_{1},f_{2})_{t_{j}}, \qquad \mbox{$v_{1}$, $v_{2}\in V^{\Gamma,t_{j}}_{t_{j}-1}$, and $f_{1}$, $f_{2} \in \k^{t_{j}}$}.
\end{equation}

\begin{definition} Let $\Gamma=(\mathbf{d}^{\Gamma},(V^{\Gamma},B^{\Gamma}))$ and $\Phi=(\mathbf{d}^{\Phi},(V^{\Phi},B^{\Phi}))$ be two sesquilinear Young Tableaux.
\begin{itemize}
\item[(i)] We say that $\Gamma$ and $\Phi$ are \emph{equivalent} if $\mathbf{d}^{\Gamma}=\mathbf{d}^{\Phi}=[t_{1}^{i_{1}},\ldots,t_{l}^{i_{l}}]$ and $(V^{\Gamma,t_{j}}_{t_{j-1}},B^{\Gamma,t_{j}}_{t_{j-1}})$ is isomorphic to $(V^{\Phi,t_{j}}_{t_{j-1}},B^{\Phi,t_{j}}_{t_{j-1}})$ for all $j=1,\ldots,l$.
\item [(ii)] We say that $\Gamma$ is \emph{$\epsilon$-Hermitian} if $(V^{\Gamma,t_{j}},B^{\Gamma,t_{j}})$ is an $\epsilon$-Hermitian module for all $j=1,\ldots,l$.
\item [(iii)] Given an $\epsilon$-Hermitian module $(V,B)$, we say that  $\Gamma$  is \emph{admissible} for $(V,B)$, or \emph{$(V,B)$-admissible}, if $(\oplus_{j}V^{\Gamma,t_{j}},\oplus_{j}B^{\Gamma,t_{j}})$ is isomorphic to $(V,B)$.
\end{itemize}
\end{definition}

Recall that given an $\sl_{2}$-triple $\gamma=\{X,H,Y\} \subset \g$ there exists $t_{1}>\ldots >t_{l}$ such that $V=\oplus_{j} V^{\gamma,t_{j}}$. Using this decomposition we may define an $\epsilon$-Hermitian Young tableaux $\Gamma_{\gamma}=(\mathbf{d}^{\Gamma_{\gamma}},(V^{\Gamma_{\gamma}},B^{\Gamma_{\gamma}}))$ by setting $\mathbf{d}^{\Gamma_{\gamma}}=[t_{1}^{i_{1}},\ldots,t_{l}^{i_{l}}]$, where $i_{j}=\dim V^{\gamma,t_{j}}_{t_{j}-1}$, and
\begin{equation}
\label{defvgt}
(V^{\Gamma_{\gamma},t_{j}}_{t_{j}-1},B^{\Gamma_{\gamma},t_{j}}_{t_{j}-1}) := (V^{\gamma,t_{j}}_{t_{j}-1},B^{\gamma,t_{j}}_{t_{j}-1}),
\end{equation}
for all $j=1,\ldots,l$.

Methods of \cite[Section 9.3]{CM92} imply that this assignment gives a
bijection between the set of $\sl_{2}$-triples in $\g$
up to the Adjoint action of $G$ and equivalence classes of
admissible $\epsilon$-Hermitian Young tableaux. We thus have

\begin{proposition}
There is a $1$-$1$ correspondence between the following sets:
\[
\{\mbox{Nilpotent orbits in $\g$}\} \longleftrightarrow
\left\{\begin{array}{c} \mbox{Equivalence classes of admissible}\\
\mbox{$\epsilon$-Hermitian Young tableaux}\end{array}\right\}.
\]
\end{proposition}

\begin{remark} If $\k$ is the field of real numbers then Hermitian Young
tableaux can be more concretely described by \emph{signed} Young diagrams
\cite[Section 9.3]{CM92}.
\end{remark}

\subsection{Generalized Whittaker models associated to nilpotent orbits}
\label{subsec:GWM}

Let $\gamma=\{X,H,Y\}\subset \g$ be an $\sl_{2}$-triple. As in the introduction, we define the character $\chi_{\gamma}$ of $U$, by $\chi_{\gamma}(\exp Z)=\psi(\kappa(X,Z))$ for all $Z\in \u$; and a symplectic structure on $\g_{-1}$ by setting
\[
\kappa_{-1}(S,T)=\kappa(\ad(X)S,T)=\kappa(X,[S,T]) \qquad \mbox{for all $S$, $T \in \g_{-1}$.}
\]
(Note the similarity between this definition and the definition of the forms $B_{i}$ on $V_{i}$ in Section \ref{subsection:nilpotentorbits}.) Let $\H_{\gamma}$ be the Heisenberg group associated to the symplectic space $(\g_{-1}, \kappa_{-1})$. That is $\H_{\gamma}=\g_{-1}\times \k$, $\{0\}\times \k$ is central, and $(T,0)(S,0)=(T+S,\kappa _{-1}(T,S)/2)$ for all $T$, $S \in \g_{-1}$. Then, according to the Stone-von Neumann theorem, there exists a unique, up to equivalence,
 \emph{smooth} irreducible (unitarizable) representation $(\rho_{\gamma},\S_{\gamma})$ of $\H_{\gamma}$ such that the center of $\H_{\gamma}$ acts by the character $\psi$. Here smooth means that it is locally constant if $\k$ is non-Archimedian, and if $\k$ is Archimedian, $(\rho_{\gamma},\S_{\gamma})$ is the smoothing of the usual irreducible unitary representation of $\H_{\gamma}$ with the central character $\psi$.

Let $\alpha_{\gamma}:U\longrightarrow \k$ be the map given by $\alpha_{\gamma}(\exp Z)=\kappa(X,Z)$, for all $Z\in \u$. It is standard to check that $\alpha_{\gamma}$ defines a surjective group homomorphism. Furthermore, it extends to a group homomorphism $\alpha_{\gamma}:N\mapsto \H_{\gamma}$ given by
\begin{equation}
\alpha_{\gamma}(\exp T \exp Z)=(T,\kappa(X,Z)), \qquad \mbox{for all $T\in \g_{-1}$, $Z\in \u$}. \label{eq:alphagammadefinition}
\end{equation}
By composition, this yields a representation $(\rho_{\chi_{\gamma}},\S_{\chi_{\gamma}})$ of $N$, where $\S_{\chi_{\gamma}}:=\S_{\gamma}$ and
\begin{equation}
\label{eq:defrcg}
\rho_{\chi_{\gamma}}(n)v=\rho_{\gamma}(\alpha_{\gamma}(n))v, \qquad
\mbox{for all $n\in N$, $v\in \S_{\chi_{\gamma}}$.}
\end{equation}
Observe that then, for all $Z\in \u$, $v\in \S_{\chi_{\gamma}}$,
\[
\rho_{\chi_{\gamma}}(\exp Z)v=\rho_{\gamma}(0,\kappa(X,Z))v=\psi(\kappa(X,Z))=\chi_{\gamma}(\exp Z).
\]
In particular, if $\g_{-1}=0$, then $N=U$ acts on the $1$-dimensional space $\S_{\chi_{\gamma}}$  by the character $\chi_{\gamma}$. Since $M_{X}$ preserves $\gamma $,
it is well-known \cite{Weil} that there exists a central cover of
$M_{X}$, to be denoted by $M_{\chi_{\gamma}}$, and a representation
of a semi-direct product $M_{\chi_{\gamma}}\ltimes N$ on
$\S_{\chi_{\gamma}}$ which extends the representation
$\rho_{\chi_{\gamma}}$ of $N$. (When the central cover splits, one may take $M_{\chi_{\gamma}}$ to be $M_{X}$ itself. See \cite{RR93} for the explicit description of the central cover.) We refer to the representation $(\rho_{\chi_{\gamma}},\S_{\chi_{\gamma}})$ of $M_{\chi_{\gamma}}\ltimes N$ as the smooth oscillator-Heisenberg representation associated to $\chi_{\gamma}$. We remark that there is a notion of ``smooth Fr\'{e}chet representations of moderate growth" for groups of the type $M_{\chi_{\gamma}}\ltimes N$. See \cite[Definition 1.4.1]{du} or \cite[Section 2]{Su}.

\begin{definition}
\label{def:gwm}
Let $(\pi,\mathscr{V})$ be a smooth representation of $G$, and let $\gamma=\{X,H,Y\}\subset \g$ be an $\sl_{2}$-triple. We define
the \emph{space of generalized Whittaker models of $\pi$ associated
to $\gamma $} to be
\begin{equation}
\label{defwhittaker}
\Wh_{\gamma}(\pi)=\Hom_{N}(\mathscr{V},\S_{\chi_{\gamma}}).
\end{equation}
Note that $\Wh_{\gamma}(\pi)$ is naturally an
$M_{\chi_{\gamma}}$-module.
\end{definition}

We also make the following definition.

\begin{definition}
Let $\Orb\subset \g$ be a nonzero nilpotent orbit. We say that $\gamma=\{X,H,Y\}$ is an \emph{$\sl_{2}$-triple of type $\Orb$} if $X\in \Orb$, where $X$ is the nilpositive element of the $\sl_{2}$-triple.
\end{definition}

From the well-known results of Jacobson-Morozov and Kostant \cite[Chapter 3]{CM92}, the map $\gamma=\{X,H,Y\}\mapsto \Orb =\Ad G
\cdot X$ yields a 1-1 correspondence between
\[
\left\{\begin{array}{c} \mbox{$\Ad G$ conjugacy classes of}\\
\mbox{$\sl_{2}$-triples in $\g$}\end{array}\right\}
\longleftrightarrow \left\{\begin{array}{c} \mbox{Nonzero nilpotent $\Ad G$-orbits}\\
\mbox{$\Orb \subset \g$}\end{array}\right\}.
\]
As noted in the introduction, it is clear that given two
conjugate $\sl_{2}$-triples $\gamma$, $\gamma '$, there will be an obvious
isomorphism $\phi:\Wh_{\gamma}(\pi)\longrightarrow
\Wh_{\gamma '}(\pi)$ that intertwines the action of
$M_{\chi_{\gamma}}$ and $M_{\chi_{\gamma '}}$. By abuse of notation, we will denote
$\Wh_{\gamma}(\pi)$ just by $\Wh_{\Orb}(\pi)$ and we will
call it the \emph{space of generalized Whittaker models associated
to $\Orb$}, or the \emph{space of generalized Whittaker models of
type $\Orb$}.

\begin{remark} Recall the convention in Remark \ref{rmkzero} for $X=0$. With this convention the expression in \eqref{defwhittaker} and therefore Definition \ref{def:gwm} will make sense for all nilpotent orbits.
\end{remark}

\section{A realization of generalized Whittaker models: $\k$ Archimedean}

Let $(\pi,\V)$ be a smooth representation of $G$, and let
$\gamma=\{X,H,Y\}\subset\g$ be an $\sl_{2}$-triple. In this section
we will give convenient realizations of the spaces
$\S_{\chi_{\gamma}}$ and $\Wh_{\gamma}(\pi)$ associated to $\gamma$
and $\pi$. As we will see later in this section, the analysis
required for $\k$ Archimedian is substantially more involved than
the case where $\k$ is non-Archimedian. For this reason, we will
devote this section to the Archimedean case and will be contented to
just indicate how the analog results work in the non-Archimedian
case.

\label{frechet}

\subsection{Norms on $G$}\label{norms} Assume that $\k=\R$ or $\C$.
Let $(V,B_{V})$ be an $\epsilon$-Hermitian module of dimension $n$. Then we
have the following possibilities for $G=G(V)$:

\begin{enumerate}
\item $\k=\R$ and $D=\R$. In this case, either
\[
\mbox{$G\cong O(p,q)$,  $p+q=n$,  if $\epsilon=1$,} \qquad \mbox{or}\qquad  \mbox{$G\cong Sp(n,\R)$ if $\epsilon=-1$ and $n$ is even.}
\]

\item $\k=\R$ and $D=\C$. In this case
\[
\mbox{$G\cong U(p,q)$, $p+q=n$, \qquad \qquad \qquad \qquad \mbox{regardless of} \, $\epsilon \ \ (=\pm 1$).}
\]

\item $\k=\R$ and $D=\mathbb{H}$. In this case either
\[
\mbox{$G\cong Sp(p,q)$,  $p+q=n$, if $\epsilon=1$,} \qquad \mbox{or}\qquad  \mbox{$G\cong O^{\ast}(2n)$, if $\epsilon=-1$}.
\]

\item $\k=\C$ and $D=\C$. Then
\[
\mbox{$G\cong O(n,\C)$ if $\epsilon=1$} \qquad \mbox{or} \qquad \mbox{$G \cong Sp(n,\C)$ if $\epsilon=-1$ and $n$ is even.}
\]
\end{enumerate}

Let $V=E\oplus U\oplus F$, with $E$ and $F$ totally isotropic,
complementary submodules of maximal dimension. Thus $U$ is anisotropic. Let
$P_{E}=\Stab_{E}$, the stabilizer of $E$. Then $P_{E}=M_{E}N_{E}$, is a Langlands decomposition of $P_{E}$,
where $M_{E}=\{m\in P_{E}\, | \, m\cdot F\subset F\}$ and the Lie algebra of $N_{E}$ is
\[
\n_{E}\cong \Hom(U,E)\oplus \z,
\]
with $\z\cong \{Z:F \rightarrow E\, | \, Z^{\ast}=-Z\}$. It follows from this description that $M_{E}\cong \GL(E)\times G(U)$ and that $G(U)$ is compact. Thus there exists
a real inner product $B^{+}_{U}$ on $U$ such that $G(U,B)\subset G(U,B^{+}_{U})$.

Let $\{e_{1},\ldots, e_{l}\}$, and $\{f_{1},\ldots,f_{l}\}$ be basis of $E$ and $F$, respectively, such that
\[
B_{V}(e_{i},f_{j})=\delta_{i,l-j+1}, \qquad \mbox{for all $1\leq i,j\leq l$.}
\]
Then we can extend $B^{+}_{U}$ to an inner product $B^{+}_{V}$ on $V$ by setting $\{e_{1},\ldots, e_{l},f_{1},\ldots,f_{l}\}$ to be an orthonormal basis of $E\oplus F$, and $(E\oplus F)^{\perp}=U$.

If we set $K=G(V,B)\cap G(V,B^{+}_{V})$, then $K$ is a maximal compact subgroup of $G(V,B_{V})$. Moreover, if we set
\[a(\lambda _1,...,\lambda_l)=\left[\begin{smallmatrix} \lambda_{1} & & & & & & \\
                                 & \ddots & & & & & \\
                                 &  &\lambda_{l} & & & & \\
                                 &  & &I_{\dim U} & & & \\
                                 &  & & & \lambda_{l}^{-1}& & \\
                                 &  & & & & \ddots & \\
                                 &  & & & & & \lambda_{1}^{-1} \end{smallmatrix}\right],\]
then \[
A=\left. \left\{ a(\lambda _1,...,\lambda_l)\, \right| \, \mbox{$\lambda_{i}\in \R^{\ast}$, for $i=1,\ldots,l$}\right\}
\]
is a maximal split torus in $G(V, B_{V})$ and, according to the Cartan decomposition, for any $g\in G(V,B_{V})$, there exists $k_{1}$, $k_{2}\in K$, and $a\in A$ such that $g=k_{1}ak_{2}$.

Given $g\in G(V, B_{V})$, let
\begin{equation}
\label{normg}
\|g\|=\sup_{B^{+}_{V}(v,v)=1} [B^{+}_{V}(gv,gv)]^{\frac{1}{2}}
\end{equation}
be the operator norm restricted to $G(V, B)$. Observe that if $g=k_{1}ak_{2}$, with $k_{1}$, $k_{2}\in K$ and $a=a(\lambda _1,...,\lambda_l)$, then
\[
\|g\|=\max\{|\lambda_{1}|,\ldots,|\lambda_{l}|,|\lambda_{1}|^{-1},\ldots,|\lambda_{l}|^{-1}\}.
\]
From this observation it is immediate that $\|g\|=\|g^{-1}\|$ and $\|\exp tX\|=\|\exp X\|^{t}$ for all $t\geq 0$, and all $X\in \a=\Lie(A)$. Since $\|\cdot \|$ is the operator norm in $\End(V)$ we also have that $\|g_{1}g_{2}\|\leq \|g_{1}\|\|g_{2}\|$ for all $g_{1}$, $g_{2}\in G$, and hence $\|\cdot \|$ satisfies all the properties of a norm on $G$ \cite[2.A.2]{Wa88}.

\subsection{Inequalities regarding norms}

Let $\gamma=\{X,H,Y\}\subset \g$ be an $\sl_{2}$-triple, and let
$P=MN$ be as in Section \ref{subsection:nilpotentorbits}. Assume that $V=\oplus_{k=-r}^{r} V_{k}$, and set
$\End_{i}(V)=\oplus_{k=-r}^{r}
\Hom(V_{k},V_{k+i})$, where $V_{k}=0$ if $|k|>r$. Then
\[
\End(V)=\bigoplus_{i=-2r}^{2r} \End_{i}(V),
\]
and $\g_{i}=\g\cap \End_{i}(V)$.  Recall that, if $\g_{-1}\neq 0$,
then there is a symplectic structure on $\g_{-1}$ given by $\kappa
_{-1}(S,T)=\frac{1}{2}\Tr(X[S,T])$. Let $\g_{-1}=\e\oplus \f$ be a complete
polarization of $\g_{-1}$ with respect to this symplectic structure,
and let $\Her_{-1}(V)=\set{T\in \End_{-1}(V)}{T^{\ast}=T}$. (Here
$\Her_{-1}$ stands for Hermitian of degree $-1$.) Then we have a
decomposition $\End_{-1}(V)=\g_{-1}\oplus \Her_{-1}(V)=\e\oplus \f
\oplus \Her_{-1}(V)$. Define a norm $\|\cdot\|_{-1}$ on
$\End_{-1}(V)$ by setting
\[
\|T_{\e}+T_{\f}+T_{\Her_{-1}(V)}\|_{-1}=\|T_{\e}\|+\|T_{\f}\|+\|T_{\Her_{-1}(V)}\|,
\]
where, as before, $\|\cdot \|$ is the operator norm, $T_{\e}\in \e$,
$T_{\f}\in \f$, and $T_{\Her_{-1}(V)}\in \Her_{-1}(V)$ are
arbitrary. Now given $T\in \End(V)$, let
\[
\|T\|_{\gamma}=\|T_{-1}\|_{-1}+\sum_{i\neq -1} \|T_{i}\|,
\]
where $T=\oplus T_{i}$, with $T_{i}\in \End_{i}(V)$. Then $\|\cdot \|_{\gamma}$ defines a norm on $\End(V)$ and, since all norms on a finite dimensional vector space are equivalent, there exists constants $C_{1}$, $C_{2}>0$ such that for all $T\in \End(V)$
\begin{equation}\label{eq:equivalenceofnorms}
C_{1}\|T\| \leq \|T\|_{\gamma} \leq C_{2}\|T\|.
\end{equation}

Recall that $\n =\u \oplus \g_{-1}=\u \oplus \e\oplus \f$. Let $\n_{\e}=\u\oplus\e$ and $\n_{\f}=\u\oplus\f$, which are ideals of $\n$.
Let
\begin{equation}
\label{eq:nenf}
N_{\e}=\exp\n_{\e} \ \ \ \text{and} \ \ \ N_{\f}=\exp \n_{\f}
\end{equation}
be the corresponding normal subgroups of $N$. Observe that for every $n\in N$ there exists unique $u\in U$, $Z_{\f}\in \f$ and $Z_{\e}\in \e$ such that $n=u(\exp Z_{\f})(\exp Z_{\e})$. It follows immediately that  $N_{\f}\backslash N\cong \e$.

\begin{lemma} For all $Z\in \e$, $\tilde{n}\in N_{\f}$, we have
\begin{equation}
\label{eq:ntildeZ}
\|\tilde{n}\, \exp Z\|\geq C_1(1+\|Z\|).
\end{equation}
\end{lemma}

\begin{proof}
Observe that $\tilde{n}=\exp Z_{1} \exp Z_{2}$, for some $Z_{1}\in
\u$, $Z_{2}\in \f$. Hence, as an element of $\End(V)$,  $\tilde{n}\, \exp
Z= \exp Z_{1} \exp Z_{2} \exp Z =1+Z+Z_{2}+\tilde{Z}$, for
some $\tilde{Z}\in \oplus_{k=2}^{2r} \End_{-k}(V)$. Therefore,
\begin{eqnarray*}
\|\tilde{n}\, \exp Z\|_{\gamma} & = & \|1+Z+Z_{1}+\tilde{Z}\|_{\gamma} \nonumber \\
                                           & = & 1+\|Z\|+\|Z_{1}\|+\|\tilde{Z}\|_{\gamma} \nonumber\\
                                           &\geq & 1+\|Z\|.
\end{eqnarray*}
The lemma then follows from (\ref{eq:equivalenceofnorms}).
\end{proof}

\begin{lemma}\label{lemma:ntildeninequality} There exists constants $\tilde{d}$, $\tilde{C} >0$ such that for all $Z\in \e$, $\tilde{n}\in N_{\f}$
\[
\|\tilde{n} (\exp Z)\|^{\tilde{d}}\geq \tilde{C}\|\tilde{n}\|\|(\exp Z)\|.
\]
\end{lemma}

\begin{proof}
Since $\|\cdot \|$ is the operator norm on $\End(V)$, we have that
\begin{equation}\label{eq:nZinequality}
\|\exp Z\| \leq 1+\|Z\|+\frac{\|Z\|^{2}}{2!}+\cdot + \frac{\|Z\|^{2r}}{2r!}\leq C_{3}(1+\|Z\|)^{2r},
\end{equation}
for some constant $C_{3} >0$. Here we have used that $V=\oplus_{k=-r}^{r} V_{k}$. Combining equations  (\ref{eq:ntildeZ}) and (\ref{eq:nZinequality})  we get that
\begin{equation}\label{eq:inequalityntildenn}
\|\tilde{n} (\exp Z)\|^{2r}\geq C_{1}^{2r}(1+\|Z\|)^{2r}\geq \frac{C_{1}^{2r}}{C_{3}} \|\exp Z\|.
\end{equation}
On the other hand, since the restriction of $\|\cdot\|$ to $G$ defines a norm on $G$, we have that $\|\tilde{n}\|\leq \|(\exp Z)^{-1}\|\|\tilde{n} (\exp Z)\|=\|\exp Z\|\|\tilde{n} (\exp Z)\|
$. From this and equation (\ref{eq:inequalityntildenn})
\[
\|\tilde{n} (\exp Z)\| \geq \frac{\|\tilde{n}\|}{\|\exp Z\|}\geq \frac{C_{1}^{2r}}{C_{3}} \frac{\|\tilde{n}\|}{\|\tilde{n} (\exp Z)\|^{2r}},
\]
and hence $\|\tilde{n} (\exp Z)\|^{2r+1}\geq \frac{C_{1}^{2r}}{C_{3}}\|\tilde{n}\|$. But now this inequality and equation  (\ref{eq:inequalityntildenn}) imply that, if we set $\tilde{C}=(\frac{C_{1}^{2r}}{C_{3}})^{2}$ and $\tilde{d}=4r+1$, then
\[
\|\tilde{n} (\exp Z)\|^{\tilde{d}}\geq \tilde{C}\|\tilde{n}\| \|\exp Z\|.
\]
\end{proof}

\begin{corollary}\label{cor:nmktildeninequality}
There exists constants $d_{0}$, $C_{0} >0$  such that, for all $k\in K$, $m\in M$, $Z\in \e$, and $\tilde{n}\in N_{\f}$
\[
\|\tilde{n} (\exp Z)mk\|^{d_{0}} \geq C_{0}\|\tilde{n}\|\|\exp Z\|\|m\|.
\]
\end{corollary}
\begin{proof}
By definition of $\|\cdot\|$, $\|\tilde{n} (\exp Z)mk\|=\|\tilde{n} (\exp Z)m\|$. Now, from the proof of \cite[Theorem 7.2.1]{Wa88}, we have that
\begin{equation}
\label{eq:ntildenminequality}
\|\tilde{n} (\exp Z)m\| \geq \|m\| \ \mbox{and}\ \|\tilde{n} (\exp Z)m\|^{2} \geq \|\tilde{n} (\exp Z)\|.
\end{equation}
Let $\tilde{C}$ and $\tilde{d}$ be as in Lemma \ref{lemma:ntildeninequality}. Then, if we set $d_{0}=2\tilde{d}+1 $ and $C_{0}=\tilde{C}$, we have that
\[
\|\tilde{n} (\exp Z)mk\|^{d_{0}}\geq \|\tilde{n} (\exp Z)\|^{\tilde{d}}\|m\|\geq C_{0}\|\tilde{n}\| \|\exp Z\| \|m\|.
\]
\end{proof}

\subsection{Some spaces of rapidly decreasing functions on $G$}

We retain all of the notation from the previous section.

Let $(\rho_{\chi_{\gamma}},\S_{\chi_{\gamma}})$ be the smooth
Heisenberg representation of $N$ associated to the character
$\chi_{\gamma}$ of $U$, as in (\ref{eq:defrcg}). We first give a
realization of $\S_{\chi_{\gamma}}$, as follows: fix
$\g_{-1}=\e\oplus \f$ (a complete polarization of $\g_{-1}$). Set
$\D(\e)$ to be the space of constant-coefficient differential
operators on $\e$, and let
\[\S(\e)=\set{f\in C^{\infty}(\e)}{\mbox{$p_{Z,d}(f) <\infty$ for
all $Z\in \D(\e)$, $d\in \N$}}\]
be the Schwartz space of $\e$, where
\begin{equation}
p_{Z,d}(f)=\sup_{T\in \e}|Zf(T)|(1+\|T\|)^{d}. \label{eq:pZddefinition}
\end{equation}
 Extend $\chi_{\gamma}$ to $N_{\f}$ by setting
 \[
\chi_{\gamma}(u\exp Z)=\chi_{\gamma}(u), \qquad \mbox{for all  $u\in U$, $Z\in \f$.}
\]

We shall adopt the following notation. For a smooth representation $\sigma$ of a closed subgroup $B$ of a Lie group $A$, let
\begin{equation}
\label{sminduction}
C^{\infty}(B\backslash A; \sigma)=\{f\in C^{\infty}(A; \sigma) \, | \, f(ba)=\sigma (b)f(a), \text{for all $b\in B$ and $a\in A$}\}.
\end{equation}
Through right multiplication, this becomes a representation of $A$ (smoothly induced from $\sigma$). Later, we will also need to consider the space $C_{c}^{\infty}(B\backslash A; \sigma)$ consisting of those elements in $C^{\infty}(B\backslash A; \sigma)$ with compact support modulo $B$.

Given $f\in\C^{\infty}(N_{\f}\backslash N; \chi_{\gamma})$, define $\hat{f}\in C^{\infty}(\e)$ by $\hat{f}(Z)=f(\exp Z)$, for all $Z\in \e$. Conversely, given $f\in C^{\infty}(\e)$, set $\check{f}(n\exp Z)=\chi_{\gamma}( n)f(Z)$ for all $n\in N_{\f}$, $Z\in \e$. Clearly these two maps are inverse of each other, and if we set
\[
\S (N_{\f}\backslash N; \chi_{\gamma})=\{f\in C^{\infty}(N_{\f}\backslash N; \chi_{\gamma}) \, | \, \, \, \hat{f}\in \S(\e)\},
\]
then
\begin{equation}
\S_{\chi_{\gamma}}\cong \S (N_{\f}\backslash N; \chi_{\gamma}). \label{eq:Schigammarealization}
\end{equation}
In the rest of this article we will frequently use this realization of $\S_{\chi_{\gamma}}$ (implicitly).

\begin{remarks} (a) In the literature, the Lie algebras $\u$, $\n$ are frequently denoted by $\n_{2}$ and $\n_{1}$, respectively. When that is the case, the Lie algebra $\n_{\f}$ is often denoted by $\n_{1.5}$.

\noindent (b) When $\k$ is non-Archimedian, we set $\S(\e)$ to be the space of all the locally constant, compactly supported functions on $\e$ (known as the Bruhat-Schwartz space on $\e$). With this definition equation (\ref{eq:Schigammarealization}) remains valid for $\k$ non-Archimedian.
\end{remarks}

\vsp
Let $U(\g)$ be the universal enveloping algebra of $\g$. Given $f\in C^{\infty}(G)$, $Z\in U(\g)$ and $d\in \N$, we set
\[
q_{Z,d}(f)=\sup_{g\in G} |R_{Z}f(g)|\|g\|^{d},
\]
where $R_{Z}$ acts on $C^{\infty}(G)$ via the right regular representation of $U(\g)$. Let
\[
\Sch(G)=\set{f\in C^{\infty}(G)}{\mbox{$q_{X,d}(f) < \infty$, for all $X\in U(\g)$, $d\in \N$}}.
\]
It is easy to check that $\Sch(G)$ is a Fr\'echet space and that
$q_{Z,d}$ is a semi-norm on $\Sch(G)$ for all $Z\in U(\g)$, $d\in \N$.
We call $\Sch(G)$ the \emph{space of rapidly decreasing functions} on
$G$.

We now define certain space of rapidly decreasing functions on $N\backslash G$. Given $f\in C^{\infty}(N\backslash G;\S_{\chi_{\gamma}})$, $Z_{1}\in
U(\g)$, $Z_{2}\in \D(\e)$, $d_{1}$, $d_{2}\in \N$, set
\begin{equation}\label{eq:seminormonGmodN}
q_{Z_{1},Z_{2},d_{1},d_{2}}(f) = \sup_{k\in K,\, m\in M,\, T\in \e} |Z_{2}(R_{Z_{1}}f(mk))(T)|(1+\|T\|)^{d_{2}}\|m\|^{d_{1}}.
\end{equation}
Then we define
\[
\Sch(N \backslash G;\S_{\chi_{\gamma}})=\{f\in C^{\infty}(N\backslash G;\S_{\chi_{\gamma}}) \, | \,\mbox{$q_{Z_{1},Z_{2},d_{1},d_{2}}(f)<\infty$ for all $q_{Z_{1},Z_{2},d_{1},d_{2}}$ as in (\ref{eq:seminormonGmodN})}\}.
\]
Observe that
\[
q_{Z_{1},Z_{2},d_{1},d_{2}}(f)=\sup_{k\in K, m\in M}p_{Z_{2},d_{2}}(R_{Z_{1}}f(mk))\|m\|^{d_{1}},
\]
where $p_{Z_{2},d_{2}}$ is as in (\ref{eq:pZddefinition}). In general, given a smooth representation of $N$, $(\tau,\HH)$, and a seminorm $\rho$ on $\HH$, we set
\begin{equation}
q_{Z,d,\rho}(f)=\sup_{k\in K, m\in M}\rho(R_{Z}f(mk))\|m\|^{d}, \label{eq:generalNG}
\end{equation}
and define
\[
\Sch(N \backslash G;\HH)=\{f\in C^{\infty}(N\backslash G;\HH) \, | \,\mbox{$q_{Z,d,\rho}(f)<\infty$ for all $q_{Z,d,\rho}$ as in (\ref{eq:generalNG})}\}.
\]

Recall that there exists a covering $M_{\chi_{\gamma}} \twoheadrightarrow M_{X}$ such that $(\rho_{\chi_{\gamma}},\S_{\chi_{\gamma}})$ extends to a representation of $M_{\chi_{\gamma}}\ltimes N$. Using this extension, we define a natural action of $M_{\chi_{\gamma}}$ on $\Sch(N \backslash G;\S_{\chi_{\gamma}})$ by
\begin{equation}
m\cdot f(g)=\rho_{\chi_{\gamma}}(m)f(\bar{m}^{-1}g), \qquad \mbox{for all $m\in M_{\chi_{\gamma}}$, $g\in G$.} \label{eq:mactiononNG}
\end{equation}
Here $\bar{m}$ is the image of $m$ under the map $M_{\chi_{\gamma}}\twoheadrightarrow M_{X}$.

Now, given $f\in C^{\infty}(N_{\f} \backslash G;\chi_{\gamma})$, $Z\in U(\g)$ and $d\in \N$, we set
\[
q_{Z,d}(f)=\sup_{k\in K, m\in M, T\in \e} |R_{Z}f((\exp T)mk)|\|(\exp T)mk\|^{d},
\]
and define
\[
\Sch(N_{\f} \backslash G;\chi_{\gamma})=\{f\in C^{\infty}(N_{\f} \backslash G;\chi_{\gamma})\, | \, \mbox{$q_{Z,d}(f) <\infty$ for all $Z\in U(\g), d\in \N$}\}.
\]

\begin{lemma}
\label{lemma:equalityss}
As $G$-modules, we have $\Sch(N_{\f} \backslash G;\chi_{\gamma})\cong \Sch(N \backslash G;\S_{\chi_{\gamma}})$.
\end{lemma}

\begin{proof}
Given $f\in \Sch(N \backslash G;\S_{\chi_{\gamma}})$, define $\hat{f}\in C^{\infty}(N_{\f} \backslash G;\chi_{\gamma})$ by $\hat{f}(g)=f(g)(0)$. We claim that $\hat{f}\in \Sch(N_{\f} \backslash G;\chi_{\gamma})$. Effectively, given $Z\in U(\g)$, $d\in \N$,
\begin{eqnarray*}
q_{Z,d}(\hat{f})  & = & \sup_{k\in K, m\in M, T\in \e} |R_{Z}\hat{f}((\exp T)mk)|\|(\exp T)mk\|^{d}\\
             & = & \sup_{k\in K, m\in M, T\in \e} |R_{Z}f(nmk)(0)|\|(\exp T)mk\|^{d}\\
             & \leq & \sup_{k\in K, m\in M, T\in \e} |R_{Z}f(mk)(T)|\|mk\|^{d}\|\exp T\|^{d}\\
             & \leq & \sup_{k\in K, m\in M, T \in \e} C_{3}|R_{Z}f(mk)(T)|\|mk\|^{d}(1+\|T\|)^{2rd}\\
            & \leq & C_{3}q_{Z,1,d,2rd}(f) <\infty,
\end{eqnarray*}
where we have used equation (\ref{eq:nZinequality}) in the next to last equality.

Now given $f\in \Sch(N_{\f} \backslash G;\chi_{\gamma})$, set $\check{f}(g)(T)=f(\exp (T)\,g)$, for all $g\in G$, $T\in \e$. We claim that $\check{f}\in \Sch(N \backslash G;\S_{\chi_{\gamma}})$. Effectively, given $Z_{1}\in U(\g)$, $Z_{2}\in \D(\e)$, $d_{1}$, $d_{2}\in \N$, we have that
\begin{eqnarray*}
q_{Z_{1},Z_{2},d_{1},d_{2}}(\check{f}) & = &  \sup_{k\in K, m\in M, T\in \e} |Z_{2}(R_{Z_{1}}\check{f}(mk))(T)|  \|mk\|^{d_{1}}(1+\|T\|)^{d_{2}}\\
                & = &  \sup_{k\in K, m\in M, T\in \e} |(R_{\Ad(mk)^{-1} Z_{2}Z_{1}}f(\exp (T) mk )|  \|mk\|^{d_{1}}(1+\|T\|)^{d_{2}},
\end{eqnarray*}
where we have identified $Z_{2}$ with an element in $U^{l}(\g)$ for some $l$. Let $\{\tilde{Z}_{1},\ldots,\tilde{Z}_{s}\}$ be a basis of $U^{l}(\g)$. Then
\[
\Ad(mk)^{-1} Z_{2}=\sum_{j=1}^{s} a_{j}(mk)\tilde{Z}_{j},
\]
for some functions $a_{j}$. Since $(\Ad, U^{l}(\g))$ is finite
dimensional, it is of moderate growth, and hence
there exists  constants $C_{l}$, $d_{l}>0$ such that
$|a_{j}(mk)|\leq C_{l}\|mk\|^{d_{l}}$, for all $k\in K$, $m\in M$.
From this, equations (\ref{eq:ntildeZ}) and (\ref{eq:ntildenminequality}), we have
\begin{eqnarray*}
q_{Z_{1},Z_{2},d_{1},d_{2}}(\check{f}) & \leq & C_{l}\sum_{j=1}^{s} \sup_{k\in K, m\in M, T\in \e}  |R_{\tilde{Z}_{j}Z_{1}}f(\exp (T) mk )|  \|mk\|^{d_{1}+d_{l}}(1+\|T\|)^{d_{2}} \\
   & \leq & \frac{C_{l}}{C_{1}^{d_{2}}}\sum_{j=1}^{s} \sup_{k\in K, m\in M, T\in \e}  |R_{\tilde{Z}_{j}Z_{1}}f(\exp (T) mk )|  \|mk\|^{d_{1}+d_{l}}\|\exp T\|^{d_{2}} \\
   & \leq & \frac{C_{l}}{C_{1}^{d_{2}}}\sum_{j=1}^{s} \sup_{k\in K, m\in M, T\in \e}  |R_{\tilde{Z}_{j}Z_{1}}f(\exp (T) mk )|  \|\exp (T) mk\|^{d_{1}+d_{l}+2d_{2}} \\
   & \leq & \frac{C_{l}}{C_{1}^{d_{2}}}\sum_{j=1}^{s} q_{\tilde{Z}_{j}Z_{1},d_{1}+d_{l}+2d_{2}}(f) <\infty.
\end{eqnarray*}
Finally, we easily check that $\check{\hat{f}}=f$ and
$\hat{\check{f}}=f$.
\end{proof}

Given $f\in \Sch(G)$, define a function $f_{\chi_{\gamma}}\in C^{\infty}(N_{\f}\backslash G;\chi_{\gamma})$ by
\[
f_{\chi_{\gamma}}(g)=\int_{N_{\f}} f(\tilde{n}g)\chi_{\gamma}(\tilde{n})^{-1}\, d\tilde{n}.
\]

\begin{lemma}
\label{lemma:smFrobenius}
The map $f\mapsto f_{\chi_{\gamma}}$ defines a surjective $G$-intertwining map from $\Sch(G)$ to the space $\Sch(N_{\f} \backslash G;\chi_{\gamma})$.
\end{lemma}

\begin{proof}
We will first show that $f_{\chi_{\gamma}}\in \Sch(N_{\f} \backslash G;\chi_{\gamma})$. Given $Z\in U(\g)$, $d\in \N$, we have that
\begin{eqnarray*}
q_{Z,d}(f_{\chi_{\gamma}}) & = & \sup_{k\in K, m\in M, T \in \e}  |R_{Z}f_{\chi_{\gamma}}((\exp T)mk)|\|(\exp T)mk\|^{d} \\
                & \leq & \sup_{k, m, T} \int_{N_{\f}} |R_{Z}f_{\chi_{\gamma}}(\tilde{n} (\exp Z)mk)|\, d\tilde{n}\, \|(\exp T)mk\|^{d}.
\end{eqnarray*}
Now we know that for all $d_{1}\in \N$, $|R_{Z}f(\tilde{n} (\exp Z)mk)|\|\tilde{n} (\exp Z)mk\|^{d_{1}}\leq q_{Z,d_{1}}(f)$. Hence, from Corollary \ref{cor:nmktildeninequality},
\begin{eqnarray*}
q_{Z,d}(f_{\chi_{\gamma}}) & \leq & \sup_{k\in K, m\in M, T \in \e} q_{Z,d_{1}}(f) \int_{N_{\f}} \|\tilde{n} (\exp Z)mk\|^{-d_{1}}\, d\tilde{n}\|(\exp T)mk\|^{d} \\
                & \leq & \sup_{k\in K, m\in M, T \in \e} q_{Z,d_{1}}(f) \int_{N_{\f}} (C_{0}\|\tilde{n}\| \|\exp T\| \|m\|)^{-\frac{d_{1}}{d_{0}}} (\|\exp T\|\|m\|)^{d} \, d\tilde{n} \\
         & = & \sup_{k\in K, m\in M, T \in \e} q_{Z,d_{1}}(f)  C_{0}^{-\frac{d_{1}}{d_{0}}} (\|\exp T\|\|m\|)^{d-\frac{d_{1}}{d_{0}}}  \int_{N_{\f}} \|\tilde{n}\|^{-\frac{d_{1}}{d_{0}}} \, d\tilde{n}.
\end{eqnarray*}
But it is clear that, if $d_{1} \gg 0$, then the right hand side of the above equation is finite.

Now we will show that the map is surjective. Fix a function $\phi \in C_{c}^{\infty}(N_{\f})$ such that
\[
\int_{N_{\f}} \phi(\tilde{n})\chi_{\gamma}(\tilde{n})^{-1} \, d\tilde{n}=1.
\]
 Given $f\in \Sch(N_{\f} \backslash G;\chi_{\gamma})$, let $h\in C^{\infty}(G)$ be given by $h(\tilde{n} (\exp Z)mk)=\phi(\tilde{n})f(nmk)$. Then it is clear that $h\in \Sch(G)$ and $h_{\chi_{\gamma}}(nmk)=f(nmk)$ for all $k\in K$, $m\in M$ and $Z\in \e$. From all this we conclude that the map $f\mapsto f_{\chi_{\gamma}}$ is  surjective.
\end{proof}

\begin{remarks} (a) Lemmas \ref{lemma:equalityss} and \ref{lemma:smFrobenius} are a form of induction by stages and Frobenius reciprocity, in the ``rapidly decreasing" context.
 The key point is of course to establish the relevant estimates.

\noindent (b) When $\k$ is non-Archimedian, we set $\Sch(G)$ to be the space of locally constant, compactly supported functions, with analogous definitions for $\Sch(N\backslash G;\S_{\chi_{\gamma}})$ and $\Sch(N_{\f}\backslash G;\chi_{\gamma})$. With these definitions, it is straightforward to check that all the results in this section remain valid in the non-Archimedian case.
\end{remarks}

\subsection{Realizing a generalized Whittaker model on $N\backslash G$} In this section, we give the promised realization of the space of generalized Whittaker models for a Casselman-Wallach representation of $G$. But before stating this result, we need the following technical lemma.

\begin{lemma}\label{lemma:invariantSchwartz} Let $(\pi, \mathscr{V})$ be a Casselman-Wallach
 representation of $G$. Then
\[
(\mathscr{V}')^{N_{\f},\chi_{\gamma}}\cong \Wh_{\gamma}(\pi),
\]
where $(\mathscr{V}')^{N_{\f},\chi_{\gamma}}$ denotes the $(N_{\f},\chi_{\gamma})$-isotypic subspace of $\mathscr{V}'$.
\end{lemma}

\begin{proof}
Given $\lambda\in \Wh_{\gamma}(\pi)$, set
$\hat{\lambda}(v)=\lambda(v)(0)$, for $v\in \mathscr{V}$. Then
it is clear that $\hat{\lambda}\in
(\mathscr{V}')^{N_{\f},\chi_{\gamma}}$. On the other hand, given
$\lambda \in (\mathscr{V}')^{N_{\f},\chi_{\gamma}}$, set
$\check{\lambda}(v)(T)=\lambda(\pi(\exp T)\, v)$ for $v\in
\mathscr{V}$, $T\in \e$. Then it is clear that $\check{\lambda}(v)\in
C^{\infty}(\e)$, but we claim that it is actually in $\S(\e)$. To see
this, observe that if $R\in \e$, then
\begin{equation}\label{eq:derivative}
\check{\lambda}(d\pi(R)v)(T)=D_{R}\check{\lambda}(v)(T), \qquad \mbox{for all $T\in \e$,}
\end{equation}
where $D_{R}\in \D(\e)$ represents the derivative in the direction $R$. On the other hand, if $S \in \f$, then \begin{equation}\label{eq:multiplication}
\check{\lambda}(d\pi(S)v)(T)=d\chi_{\gamma}(\kappa(S,T))\check{\lambda}(v)(T), \qquad \mbox{for all $T\in \e$,}
\end{equation}
where $d\chi_{\gamma}$ is the linear functional given by the derivative of the character $\chi_{\gamma}$.
But now, since $(\pi, \mathscr{V})$ is a representation of moderate growth, we can find a constant $d>0$ such that for all $v\in \mathscr{V}$, there exists $C_{\lambda,v}>0$ such that
\begin{equation}\label{eq:moderategrowth}
|\check{\lambda}(v)(T)|=|\lambda(\pi(\exp T)\, v)|\leq C_{\lambda,v} \|\exp T\|^{d}\leq C_{3}C_{\lambda,v}(1+\|T\|)^{2dr},\qquad \mbox{for all $T\in \e$.}
\end{equation}
Here $C_3$ and $r$ are as in equation (\ref{eq:nZinequality}). Observe that, although the constant $C_{\lambda,v}$ depends on $v$
(and on $\lambda$), $d$ is independent of the element $v\in
\mathscr{V}$ chosen. Therefore, from equations
(\ref{eq:derivative}), (\ref{eq:multiplication}) and
(\ref{eq:moderategrowth}) for all $v\in  \mathscr{V}$ the function
$\check{\lambda}(v)$ is such that if we take any partial
derivatives, and multiply by any polynomial on $\e$, it still has growth
controlled by a polynomial of degree $2dr$. But this implies that
$\check{\lambda}(v)\in \S(\e)$, as we wanted to show.

To finish the proof we just have to show that $\hat{\check{\lambda}}=\lambda$, and $\check{\hat{\lambda}}=\lambda$, but this follow easily from the definitions.
\end{proof}

Given an $\sl_{2}$-triple $\gamma=\{X,H,Y\}\subset \g$, set
\begin{equation}
\label{eq:defgammacheck} \check{\gamma}=\{-X,H,-Y\}.
\end{equation}
Then, it is immediate to check that $\check{\gamma}$ is also an
$\sl_{2}$-triple and that
$\chi_{\check{\gamma}}=\check{\chi}_{\gamma}$, where
$\check{\chi}_{\gamma}$ is the character of $U$ given by
$\check{\chi}_{\gamma}(u)=\chi_{\gamma}(u)^{-1}$ for all $u\in U$.

\begin{proposition}\label{prop:SchwartzWhittakerModels}
Let $(\pi,\mathscr{V})$ be an irreducible Casselman-Wallach representation of $G$, and let $\Sch(N\backslash G;\S_{\chi_{\check{\gamma}}})_{G,\pi}$ denote the maximal $\pi$-isotypic quotient of
$\Sch(N\backslash G;\S_{\chi_{\check{\gamma}}})$. Then, as an $M_{\chi_{\gamma}}\times G$-module,
\[
\Sch(N\backslash G;\S_{\chi_{\check{\gamma}}})_{G,\pi}\cong  W_{\gamma}(\check{\pi}) \otimes \pi,
\]
where $\check{\pi}$ is the contragredient Casselman-Wallach representation of $\pi$ and $W_{\gamma}(\check{\pi})$ is the continuous dual of $\Wh_{\gamma}(\check{\pi})$. Here $\otimes$ stands for the completed projective tensor product (of two locally convex topological spaces).
\end{proposition}

\begin{proof} By Lemma \ref{lemma:equalityss}, we have $\Sch(N\backslash G;\S_{\chi_{\check{\gamma}}})\cong \Sch(N_{\f} \backslash G;\chi_{\check{\gamma}})$. Let $\lambda\in \Hom_{G}(\Sch(N_{\f} \backslash G;\chi_{\check{\gamma}}),\mathscr{V})$, and let $(\check{\pi},\check{\mathscr{V}})$ be the representation contragredient to $(\pi,\mathscr{V})$. Then we may think of $\lambda$ as an element in a space of $\check{\mathscr{V}}'$-valued distributions. Hence, according to \cite[Theorem 3.11]{KV96} there exists $\mu_{\lambda}\in (\check{\mathscr{V}}')^{N_{\f},\chi_{\gamma}}$ such that if $f\in C_{c}^{\infty}(N_{\f} \backslash G;\chi_{\gamma})$, then
\[
\lambda(f)(v)=\int_{N_{\f} \backslash G} f(g)\mu_{\lambda}(\pi(g)v)\, dg, \qquad \mbox{for} \ v \in \check{\mathscr{V}}.
\]
On the other hand, given $\mu \in (\check{\mathscr{V}}')^{N_{\f},\chi_{\gamma}}$ and $f\in \Sch(N_{\f} \backslash G;\chi_{\check{\gamma}})$, there exists $h\in \Sch(G)$ such that $f=h_{\chi_{\check{\gamma}}}$ and for all $v\in \check{\mathscr{V}}$
\begin{eqnarray*}
\int_{G} h(g)\mu_{\lambda}(\pi(g)v)\, dg & = & \int_{N_{\f} \backslash G}\int_{N_{\f}}h(\tilde{n}g)\mu_{\lambda}(\pi(\tilde{n})\pi(g)v)\, d\tilde{n} \,dg\\
& = & \int_{N_{\f} \backslash G}\int_{N_{\f}}h(\tilde{n}g) \chi_{\gamma}(\tilde{n})\, d\tilde{n}\,\, \mu_{\lambda}(\pi(g)v) \,dg\\
& = & \int_{N_{\f} \backslash G}\int_{N_{\f}}h(\tilde{n}g) \chi_{\check{\gamma}}(\tilde{n})^{-1}\, d\tilde{n}\,\, \mu_{\lambda}(\pi(g)v) \,dg\\
& = & \int_{N_{\f} \backslash G} h_{\chi_{\check{\gamma}}}(g)  \mu_{\lambda}(\pi(g)v) \,dg\\
& = & \int_{N_{\f} \backslash G} f(g)  \mu_{\lambda}(\pi(g)v) \,dg.
\end{eqnarray*}
Therefore, if we set
\[
\lambda_{\mu}(f)(v)=\int_{N_{\f} \backslash G} f(g)  \mu_{\lambda}(\pi(g)v) \,dg, \qquad \mbox{for} \, v \in \check{\mathscr{V}},
\]
then $\lambda_{\mu}\in \Hom_{G}(\Sch(N_{\f} \backslash
G;\chi_{\check{\gamma}}),\check{\mathscr{V}}')$. Method of \cite[Section 3]{SZ11} implies that $\lambda_{\mu}(f)$ extends to a continuous linear
functional on $\mathscr{V}'$, and since $V$ is irreducible, $\lambda_{\mu}(f)$ is actually
an element of $\mathscr{V}$ for all $f\in \Sch(N_{\f} \backslash
G;\chi_{\check{\gamma}})$, or equivalently $\lambda_{\mu}\in
\Hom_{G}(\Sch(N_{\f} \backslash G;\chi_{\check{\gamma}}),\mathscr{V})$. Now it
is easy to check that $\lambda_{\mu_{\lambda}}=\lambda$ and
$\mu_{\lambda_{\mu}}=\mu$. Therefore
\[
\Sch(N_{\f} \backslash G;\chi_{\check{\gamma}})_{G,\pi} \cong \pi\otimes \check{\pi}_{N_{\f},\chi_{\gamma}}.
\]
Finally, according to Lemma \ref{lemma:invariantSchwartz} and by taking the continuous dual, we have $\check{\pi}_{N_{\f},\chi_{\gamma}}\cong W_{\gamma}(\check{\pi})$. Therefore we obtain the required isomorphism, which is easily checked to be $M_{\chi_{\gamma}}$-equivariant. \end{proof}

\begin{remarks} (a) In the Archimedian setting whenever we take the tensor product of two locally convex topological vector spaces, we always mean the completed projective tensor product. Observe that if $(\pi,\V)$ is a Casselman-Wallach representation, then $\V$ is nuclear, and hence the completed projective tensor product of $\V$ with any locally convex topological vector space is equivalent to the completed injective tensor product \cite{G55}.

\noindent (b) Proposition \ref{prop:SchwartzWhittakerModels} also
holds for $\k$ non-Archimedian, if we replace in its statement
Casselman-Wallach representations by smooth, finitely generated,
admissible representations, and completed projective tensor product
by algebraic tensor product. The proof for this case follows the
same line, but it is more straightforward.
\end{remarks}

\section{Dual pairs and theta lifting of nilpotent orbits}
\label{sec:liftOrbit}

\subsection{Reductive dual pairs of type I} \label{reductivepairs}
Let $(V,B)$ be an $\varepsilon$-Hermitian module, and let
$(\tilde{V},\tilde{B})$ be  an $\tilde{\varepsilon}$-Hermitian
module such that $\varepsilon\tilde{\varepsilon}=-1$. Let
$G=G(V)$ and $\tilde{G}=G(\tilde{V})$ be its associated isometry
groups. We will use the $\varepsilon$-Hermitian form $B$ to
identify the $\k$-vector spaces $\tilde{V}\otimes_{D}V^{\ast}$ and
$\Hom_{D}(V,\tilde{V})$. Given $T\in \Hom_{D}(V,\tilde{V})$, define
$T^{\ast}\in \Hom_{D}(\tilde{V},V)$ by
\[
\tilde{B}(Tv,\tilde{v})=B(v,T^{\ast}\tilde{v}) \qquad
\mbox{for all $v\in V$, $\tilde{v}\in \tilde{V}$.}
\]
Define similarly the $*$ map from $\Hom_{D}(\tilde{V},V)$ to $\Hom_{D}(V,\tilde{V})$. Note that since $\varepsilon\tilde{\varepsilon}=-1$, we have $T^{\ast \ast}=-T$ for $T\in \Hom_{D}(V,\tilde{V})$.

We now define a
symplectic form $\langle\cdot , \cdot  \rangle$ on $\Hom_{D}(V,\tilde{V})$ by setting
\begin{equation}
\label{defsympro}
\langle T,S\rangle =\Tr(T^{\ast}S) \qquad \mbox{for all $T$, $S\in \Hom_{D}(V,\tilde{V})$,}
\end{equation}
where $\Tr(T^{\ast}S)$ is the trace of $T^{\ast}S$ as a $\k$-linear transformation.
Let
\[
 \Sp(\Hom_{D}(V,\tilde{V}))=\left\{g\in \GL(\Hom_{D}(V,\tilde{V}),\k) \, \left| \, \begin{array}{c} \mbox{$\langle g \cdot T, g\cdot S \rangle= \langle T, S \rangle$}\\ \mbox{for all $T$, $S \in \Hom_{D}(V,\tilde{V})$} \end{array} \right.\right\}.
\]
 Then there is a natural map $G \times \tilde{G}\longrightarrow \Sp(\Hom_{D}(V,\tilde{V}))$ given by
\[
 (g,\tilde{g})\cdot T = \tilde{g} T g^{-1} \qquad \mbox{for all $T \in \Hom(V,\tilde{V})$, $g\in G$, $\tilde{g}\in \tilde{G}$}.
\]
We will use this map to identify $G$ and $\tilde{G}$ with subgroups
of $\Sp(\Hom_{D}(V,\tilde{V}))$. These two subgroups are mutual
centralizers of each other, and form an example of a {\em reductive
dual pair} of type I. See \cite{Ho79}.

\subsection{Moment maps and lifting of nilpotent orbits}
\label{subsec:liftOrbit}

Given $T\in \Hom_{D}(V,\tilde{V})$, it is clear that $T^{\ast}T\in \g$ and $TT^{\ast}\in \tilde{\g}$. Following \cite{KP82, DKP} we define the \emph{moment maps} to be
\begin{eqnarray*}
\varphi:\Hom_{D}(V,\tilde{V}) & \longrightarrow & \g\\
               T & \mapsto  &  T^{\ast}T
\end{eqnarray*}
and
\begin{eqnarray*}
\tilde{\varphi}:\Hom_{D}(V,\tilde{V}) & \longrightarrow & \tilde{\g}\\
               T & \mapsto  &  TT^{\ast}.
\end{eqnarray*}
It is also clear that $\varphi (T)$ is nilpotent if and only if $\tilde{\varphi}(T)$ is nilpotent.

As in the introduction, let $\Max \Hom (V,\tilde{V})$ be the set of
full rank elements in $\Hom(V,\tilde{V})$. Without any loss of
generality, we assume that $\dim V \leq \dim \tilde{V}$, and
elements of $\Max \Hom (V,\tilde{V})$ are then represented by
injective maps from $V$ to $\tilde{V}$. Recall also our standing
assumption on the nilpotent orbit $\Orb \subset \g$:
\begin{equation}
\label{assumption}
\varphi ^{-1}(\Orb) \cap \Max \Hom (V,\tilde{V})\ne \emptyset.
\end{equation}

We first prove a result on dual pairs in the stable range, which says that in this case any (not just nilpotent) orbit satisfies the stated assumption.
Recall that the dual pair $(G,\tilde{G})$ is in the stable range, with $G$ the smaller member, if there is
a polarization $\tilde{V}=\tilde{E}\oplus \tilde{U}\oplus \tilde{F}$, where $\tilde{E}$, $\tilde{F}$ are
totally isotropic complementary subspaces with $\dim \tilde{E}=\dim \tilde{F}=\dim V$.

\begin{lemma}\label{lemma:existance} Assume that the dual pair $(G,\tilde{G})$ is in the stable range with $G$ the smaller member. Given $X\in \g$, there exists an injective map $T\in \Hom(V,\tilde{V})$ such that $T^{\ast}T=X$.
\end{lemma}

\begin{proof}
Fix a linear isomorphism $T_{\tilde{E}}:V\rightarrow \tilde{E}$ and define a linear map $T_{\tilde{F}}:V\rightarrow \tilde{F}$ by $T_{\tilde{F}} =(T_{\tilde{E}}^{\ast})^{-1} X/2$. Then, if we set $T=T_{\tilde{E}}+T_{\tilde{F}}$, we have that
\[
T^{\ast}T=(T_{\tilde{E}}^{\ast}+T_{\tilde{F}}^{\ast})(T_{\tilde{E}}+T_{\tilde{F}})=T_{\tilde{E}}^{\ast}T_{\tilde{F}}+T_{\tilde{F}}^{\ast}T_{\tilde{E}}=X/2+X/2=X.
\]
Note that we have used the fact that $T_{\tilde{F}}^{\ast}=X^{\ast}(T_{\tilde{E}}^{\ast \ast})^{-1}/2=XT_{\tilde{E}}^{-1}/2$. Now, since $T_{\tilde{E}}$ is injective, we conclude that $T$ is injective and $T^{\ast}T=X$ as we wanted to show.
\end{proof}

\begin{remark} The correspondence of nilpotent orbits for dual pairs in the stable range is well understood. See \cite{DKP,Pan}, and \cite{NOZ} (for $K_{\C}$-nilpotent orbits).
This is also covered by the next proposition, in view of Lemma \ref{lemma:existance}.
\end{remark}

\begin{proposition}\label{prop:Gammatilde}
Let $\Orb\subset \g$ be a nilpotent orbit corresponding to
the admissible $\varepsilon$-Hermitian Young tableau $\Gamma=(d^{\Gamma},(V^{\Gamma},B^{\Gamma}))$, where $d^{\Gamma}=[t_{1}^{i_{1}},\ldots,t_{l}^{i_{l}}]$.
Given $X\in
\Orb$, let $T\in \Hom(V,\tilde{V})$ be an injective map such
that $T^{\ast}T=X$.  Then the orbit of the nilpotent element $\tilde{X}:=TT^{\ast}\in \tilde{\g}$ corresponds to the unique
equivalence class of admissible $\tilde{\varepsilon}$-Hermitian Young tableaux $\tilde{\Gamma}=(d^{\tilde{\Gamma}},(V^{\tilde{\Gamma}},B^{\tilde{\Gamma}}))$, where
\begin{itemize}
\item $d^{\tilde{\Gamma}}=[(t_{1}+1)^{i_{1}},\ldots,(t_{l}+1)^{i_{l}},1^{s}]$, with $s=\dim \tilde{V}-\dim V-\sum_{j=1}^{l} i_{j}$, namely $\tilde{\Gamma}$ is obtained by adding a column of length $\dim \tilde{V}-\dim V$ to $\Gamma$;
\item $(V^{\tilde{\Gamma},t_{j}+1}_{t_{j}},B^{\tilde{\Gamma},t_{j}+1}_{t_{j}})\cong
(V^{\Gamma,t_{j}}_{t_{j}-1},B^{\Gamma,t_{j}}_{t_{j}-1})$, for all $j=1,\ldots,l$;
\item $(V^{\tilde{\Gamma},1}_{0},B^{\tilde{\Gamma},1}_{0})$ is an $\tilde{\varepsilon}$-Hermitian module of dimension $s$.
\end{itemize}
\end{proposition}

\begin{proof}
Let $\gamma=\{X,H,Y\}$ be an $\sl_{2}$-triple containing $X$. Define $V_{i}^{\gamma,t_{j}}$ as in Section \ref{subsection:nilpotentorbits}, and set
\begin{equation*}
\tilde{V}_{i+1}^{\tilde{\gamma},t_{j}+1}=T(V_{i}^{\gamma,t_{j}}).
\end{equation*}
(At this point, $\tilde{\gamma}$ has no meaning.) Since $T$ is injective, we have a direct sum decomposition:
\[
T(V)=\bigoplus _{j=1}^l T(V^{\gamma,t_{j}}),
\]
and for each $j$, we have
$$
T(V^{\gamma,t_{j}})=\bigoplus_{i=1}^{t_{j}}T(V^{\gamma,t_{j}}_{-t_{j}+2i-1})=\bigoplus_{i=1}^{t_{j}}\tilde{V}^{\tilde{\gamma},t_{j}+1}_{-t_{j}+2i}.
$$

Observe that for all $v$, $w\in V$,
\[
\tilde{B}(Tv,Tw)=B(v,T^{\ast}Tw)=B(v,Xw).
\]
This implies that $T(V^{\gamma,t_{i}})\bot T(V^{\gamma,t_{j}})$ for $i\ne j$. We also note the following: if $\mathcal{A},\mathcal{B}\subset V$ and $\mathcal{A}, X(\mathcal{B})$ form a perfect pairing under $B$, then $T(\mathcal{A}), T(\mathcal{B})$ form a perfect pairing under $\tilde{B}$.
This implies that, for a fixed $j$, $\bigoplus_{i=1}^{t_{j}-1}\tilde{V}^{\tilde{\gamma},t_{j}+1}_{-t_{j}+2i}$ is non-degenerate, and $\tilde{V}^{\tilde{\gamma},t_{j}+1}_{t_{j}}$ is orthogonal to $\bigoplus_{i=1}^{t_{j}-1}\tilde{V}^{\tilde{\gamma},t_{j}+1}_{-t_{j}+2i}$. Therefore there exists $\tilde{V}^{\tilde{\gamma},t_{j}+1}_{-t_{j}}$ such that
\begin{itemize}
\item $T(V^{\gamma,t_{j}})\oplus \tilde{V}^{\tilde{\gamma},t_{j}+1}_{-t_{j}}= \bigoplus_{i=0}^{t_{j}}\tilde{V}^{\tilde{\gamma},t_{j}+1}_{-t_{j}+2i}$ is non-degenerate; and
\item $\tilde{V}^{\tilde{\gamma},t_{j}+1}_{-t_{j}}\, \bot \, T(V^{\gamma,t_{i}})$, for all $i\ne j$.
\end{itemize}

By an inductive argument, we may thus find $\tilde{V}^{\tilde{\gamma},t_{j}+1}_{-t_{j}}$ for $1\leq j\leq l$ such that
\begin{itemize}
\item $\tilde{U}_j=: \oplus_{i=0}^{t_{j}}\tilde{V}^{\tilde{\gamma},t_{j}+1}_{-t_{j}+2i}$ is non-degenerate; and
\item $\oplus _{j=1}^l\tilde{U}_j$ is an orthogonal direct sum.
\end{itemize}

Let
$\tilde{V}^{\tilde{\gamma},1}_{0}$ be the orthogonal complement of $\oplus _{j=1}^l\tilde{U}_j$. Then, we have a decomposition
\[
\tilde{V}=\boxplus_{j=1}^{l}\bigg(\oplus_{i=0}^{t_{j}} \tilde{V}^{\tilde{\gamma},t_{j}+1}_{-t_{j}+2i}\bigg)
\boxplus \tilde{V}^{\tilde{\gamma},1}_{0},
\]
where $\boxplus$ denotes an orthogonal direct sum.

Let $\tilde{V}_{i}=\oplus_{j}\tilde{V}^{\tilde{\gamma},t_{j}+1}_{i}$, and
define an element $\tilde{H}\in \tilde{\g}$ (easily checked) by setting $\tilde{H} \tilde{v}= i\,\tilde{v}$
for all $\tilde{v}\in \tilde{V}_{i}$. Note that for $i<t_{j}-1$, $T^{\ast}|_{\tilde{V}^{\tilde{\gamma},t_{j}+1}_{i}}:\tilde{V}^{\tilde{\gamma},t_{j}+1}_{i}\rightarrow V^{\gamma,t_{j}}_{i+1}$ is an isomorphism and $TT^{\ast}=\tilde{X}:\tilde{V}^{\tilde{\gamma},t_{j}+1}_{i}\rightarrow
\tilde{V}^{\tilde{\gamma},t_{j}+1}_{i+2}$. In particular, $\tilde{X}$ is nilpotent.
Since $[\tilde{X},\tilde{H}]=2\tilde{X}$ and $\tilde{H}$ is semisimple, there exists an element $\tilde{Y}\in \tilde{\g}$ such that
$\{\tilde{X},\tilde{H}, \tilde{Y}\}$ is an $\sl_{2}$-triple. See \cite[Section 3.3]{CM92}. From
all this it is clear that if $\tilde{\Gamma}=(d^{\tilde{\Gamma}},(V^{\tilde{\Gamma}},B^{\tilde{\Gamma}}))$ is
the $\tilde{\varepsilon}$-Hermitian Young tableau associated to the orbit
of $\tilde{X}$, then $d^{\tilde{\Gamma}}$ is obtained by adding a column to
$d^{\Gamma}$. Finally observe that $\tilde{X}^{t_{j}}=TX^{t_{j}-1}T^{\ast}$, and so
\[\tilde{X}^{-t_{j}}|_{\tilde{V}^{\tilde{\gamma},t_{j}+1}_{t_{j}}}=(T^{-1})^*X^{-(t_{j}-1)}T^{-1}|_{\tilde{V}^{\tilde{\gamma},t_{j}+1}_{t_{j}}}.\]
If $\tilde{v}$, $\tilde{w}\in
\tilde{V}^{\tilde{\gamma},t_{j}+1}_{t_{j}}$, then
\begin{equation}\label{eq:isometry}
\begin{aligned}
\tilde{B}^{\tilde{\gamma},t_{j}+1}_{t_{j}}(\tilde{v},\tilde{w})&=\tilde{B}(\tilde{X}^{-t_{j}}\tilde{v},\tilde{w})\\
                                                               &=\tilde{B}((T^{-1})^*X^{-(t_{j}-1)}T^{-1}\tilde{v}, \tilde{w})\\
                                                               &=B(X^{-(t_{j}-1)}T^{-1}\tilde{v}, T^{-1}\tilde{w})\\
                                                               &=B_{t_{j}+1}^{\gamma,t_{j}}(T^{-1}\tilde{v},T^{-1}\tilde{w}),
\end{aligned}
\end{equation}
which implies that $(V^{\tilde{\Gamma},t_{j}+1}_{t_{j}},B^{\tilde{\Gamma},t_{j}+1}_{t_{j}})\cong
(V^{\Gamma,t_{j}}_{t_{j}-1},B^{\Gamma,t_{j}}_{t_{j}-1})$, for all $j=1,\ldots,l$.
\end{proof}

Let $\Orb\subset\g$ be a nilpotent orbit (satisfying \eqref{assumption}), which corresponds to an $\varepsilon$-Hermitian Young tableau $\Gamma=(d^{\Gamma},(V^{\Gamma},B^{\Gamma}))$. We define the \emph{theta lift} of $\Orb$, $\Theta(\Orb)\subset\tilde{\g}$, to be the nilpotent orbit corresponding to the $\tilde{\varepsilon}$-Hermitian Young tableau $\tilde{\Gamma}=(d^{\tilde{\Gamma}},(V^{\tilde{\Gamma}},B^{\tilde{\Gamma}}))$ specified in Proposition \ref{prop:Gammatilde}.

\begin{remark} It can be shown that, for the orbit $\Orb$ under consideration, $\tilde{\varphi}(\varphi^{-1}(\overline{\Orb}))=\overline{\Theta(\Orb)}$. Cf. \cite[Lemma 4.3]{KP82}. Therefore, our definition of the theta lift of $\Orb$ agrees with the usual one.
\end{remark}

\subsection{Lifting of $\sl_{2}$-triples}
\label{liftTriple} In this section we will introduce the concept of
lifting of $\sl_{2}$-triples. We start with the following

\begin{definition}
Let $\gamma=\{X,H,Y\}\subset \g$ and $\tilde{\gamma}=\{\tilde{X},\tilde{H},\tilde{Y}\}\subset \tilde{\g}$ be two $\sl_{2}$-triples of type $\Orb$ and $\tilde{\Orb}$, respectively. We say that $T\in \Hom(V,\tilde{V})$ \emph{lifts} $\gamma$ to $\tilde{\gamma}$ if $T^{\ast}T=X$, $TT^{\ast}=\tilde{X}$ and $T(V_{j})\subset \tilde{V}_{j+1}$ for all $j$.
\end{definition}

Set
\begin{equation}
\label{defogg}
\Orb_{\gamma,\tilde{\gamma}}=\{T\in \Hom(V,\tilde{V})\, | \, \mbox{$T$ lifts $\gamma$ to $\tilde{\gamma}$}\}.
\end{equation}

We prove the following elementary

\begin{lemma}
\label{lemma:injsur}
 Let $\gamma=\{X,H,Y\}\subset \g$ and $\tilde{\gamma}=\{\tilde{X},\tilde{H},\tilde{Y}\}\subset \tilde{\g}$ be two $\sl_{2}$-triples. Let $T\in \Orb_{\gamma,\tilde{\gamma}}$ and denote \[
T_j=T|_{V_{j}}\in \Hom(V_j,\tilde{V}_{j+1}).\]
Then $T_j$ is injective for $j<0$, and surjective for $j\geq 0$.
\end{lemma}

\begin{proof} We have $X|_{V_j}=T^*T|_{V_j}=T^*_{-j-2}T_j$. Since all highest weight vectors of a finite dimensional $\sl_{2}$ representation have non-negative weights, we know $X|_{V_j}$ is injective for $j<0$. This implies the first assertion.

For the second assertion, we first note that $T^*_j: \tilde{V}_{-j-1}\mapsto V_{-j}$. We have $\tilde{X}|_{\tilde{V}_{-j-1}}=TT^*|_{\tilde{V}_{-j-1}}=T_{-j}T^*_j$. By the same reasoning, we know that $T^*_j$ is injective for $j\geq 0$. The second assertion follows.
\end{proof}

From now on, we assume that $\Orb$ satisfies \eqref{assumption}, as
in the introduction. We let
\begin{equation}
\label{defoggMax}
\Orb^{\Max}_{\gamma,\tilde{\gamma}}=\Orb_{\gamma,\tilde{\gamma}}\cap \Max \Hom (V,\tilde{V}).
\end{equation}

\begin{lemma}
\label{lem:oggMax}
Let $\gamma=\{X,H,Y\}\subset \g$ and $\tilde{\gamma}=\{\tilde{X},\tilde{H},\tilde{Y}\}\subset \tilde{\g}$ be two $\sl_{2}$-triples of type $\Orb$ and $\Theta(\Orb)$, respectively. Then $\Orb^{\Max}_{\gamma,\tilde{\gamma}}\neq \emptyset$ and it is a single $M_{X}\times \tilde{M}_{\tilde{X}}$-orbit.
\end{lemma}

\begin{proof} From the proof of Proposition \ref{prop:Gammatilde}, we know $\Orb^{\Max}_{\gamma,\tilde{\gamma}}\neq \emptyset$.
Now let $T\in \Orb^{\Max}_{\gamma,\tilde{\gamma}}$ and set
\[
T_{i}^{\gamma,t_{j}}=T|_{V_{i}^{\gamma,t_{j}}} \qquad \mbox{ for all $1\leq j\leq l$, $-t_{j}+1\leq i \leq t_{j}-1$.}
\]
We claim that $T_{i}^{\gamma,t_{j}}$ defines a linear isomorphism between $V_{i}^{\gamma,t_{j}}$ and $\tilde{V}_{i+1}^{\tilde{\gamma},t_{j}+1}$ for all $1\leq j\leq l$ and $-t_{j}+1\leq i \leq t_{j}-1$. Since $TX=\tilde{X}T$, $T(V_{i})\subset \tilde{V}_{i+1}$ and $T$ is injective, we clearly have $T(V_{t_j-1}^{\gamma,t_{j}}) \subseteq
\tilde{V}_{t_{j}}^{\tilde{\gamma},t_{j}+1}$. By Lemma \ref{lemma:injsur}, $T_j$ is surjective for all $j\geq 0$. Thus any element of $\tilde{V}_{t_{j}}^{\tilde{\gamma},t_{j}+1}$ is of the form $T(v)$ for some $v\in V_{t_j-1}$. But then $\tilde{X}T(v)=TX(v)=0$ and since $T$ is injective, we must have $X(v)=0$. Thus $T_{t_j-1}^{\gamma,t_{j}}: V_{t_j-1}^{\gamma,t_{j}} \rightarrow \tilde{V}_{t_{j}}^{\tilde{\gamma},t_{j}+1}$ is a linear isomorphism. By applying an appropriate power of $\tilde{X}$, we may conclude that $T_{i}^{\gamma,t_{j}}$ is a linear isomorphism between $V_{i}^{\gamma,t_{j}}$ and $\tilde{V}_{i+1}^{\tilde{\gamma},t_{j}+1}$ for all $1\leq j\leq l$ and $-t_{j}+1\leq i \leq t_{j}-1$.

We set
\[
X_{i}^{\gamma,t_{j}}= X|_{V_{i}^{\gamma,t_{j}}} \qquad \mbox{and} \qquad \tilde{X}_{i}^{\gamma,t_{j}+1}=\tilde{X}|_{\tilde{V}_{i}^{\tilde{\gamma},t_{j}+1}}.
\]
Then it is clear that
$X_{i}^{\gamma,t_{j}}=(T_{-i-2}^{\gamma,t_{j}})^{\ast}T_{i}^{\gamma,t_{j}}$
and
$\tilde{X}_{i}^{\tilde{\gamma},t_{j}+1}=T_{i+1}^{\gamma,t_{j}}(T_{-i-1}^{\gamma,t_{j}})^{\ast}$.
From this we conclude that $T$ is completely determined by the maps
$T_{t_{j}-1}^{\gamma,t_{j}}$, for $j=1,\ldots,l$, but now, from
equation (\ref{eq:isometry}), all the $T_{t_{j}-1}^{\gamma,t_{j}}$'s
are isometries. From this and equation
(\ref{eq:productisometrygroups}) it follows immediately that
$\Orb^{\Max}_{\gamma,\tilde{\gamma}}$ is a single $M_{X}\times
\tilde{M}_{\tilde{X}}$-orbit.
\end{proof}

Let $\gamma=\{X,H,Y\}\subset \g$ and $\tilde{\gamma}=\{\tilde{X},\tilde{H},\tilde{Y}\}\subset \g'$ be two $\sl_{2}$-triples. Given $T\in \Orb^{\Max}_{\gamma,\tilde{\gamma}}$, we define a map $\phi_{T}:\tilde{M}_{\tilde{X}}\longrightarrow M_{X}$ by
\begin{equation}
\label{defphi}
\phi_{T}(\tilde{m})v=(T_{i}^{\gamma,t_{j}})^{-1}\tilde{m} T_{i}^{\gamma,t_{j}}v,
\end{equation}
for all $\tilde{m}\in \tilde{M}_{\tilde{X}}$, and $v\in V_{i}^{\gamma,t_j}$, where $1\leq j\leq l$ and $-t_{j}+1\leq i \leq t_{j}-1$. Note that as elements of $\Hom(V,\tilde{V})$, we have \[\tilde{m}T=T\phi_{T}(\tilde{m}),\]
for all $\tilde{m}\in \tilde{M}_{\tilde{X}}$.

\subsection{Isomorphism of $\g_{-1}\oplus \tilde{\g}_{-1}$ with a symplectic subspace $W_{\gamma,\tilde{\gamma}}$ of $\Hom(V,\tilde{V})$}
\label{subsection:embedding}

Let $\gamma=\{X,H,Y\}$ and $\tilde{\gamma}=\{\tilde{X},\tilde{H},\tilde{Y}\}$ be two $\sl_{2}$-triples of type $\Orb$ and $\Theta(\Orb)$, respectively.  From the proof of Proposition \ref{prop:Gammatilde}, we may assume that
\[
V=\bigoplus_{k=-r}^{r} V_{k} \qquad \mbox{and} \qquad \tilde{V}=\bigoplus_{k=-r-1}^{r+1} \tilde{V}_{k},
\]
for some $r$. We will set
\begin{equation}
\label{def:Wgg}
W_{\gamma,\tilde{\gamma}}=\bigoplus_{k=-r}^{r}
\Hom(V_{k},\tilde{V}_{k})\subset \Hom(V,\tilde{V}).
\end{equation}
Then it is clear that the restriction of the form $\langle \cdot,\cdot \rangle$, defined in Section \ref{reductivepairs}, to $W_{\gamma,\tilde{\gamma}}$ is non-degenerate. 
Now fix $T \in \Orb^{\Max}_{\gamma,\tilde{\gamma}}$, and define
\begin{equation}
\label{eq:sigmaT}
\begin{array}{rcl}
J_{T}: \ \ \g_{-1}\oplus \tilde{\g}_{-1} & \longrightarrow & W_{\gamma,\tilde{\gamma}}, \\
(R,\tilde{R}) \hspace{10pt} & \mapsto & \hspace{-5pt} TR+\tilde{R}T. 
\end{array}
\end{equation}
Obviously, we have $J_{T}(R,\tilde{R})(V_{k})\subseteq \tilde{V}_{k}$ for all $k$.

\begin{lemma}
\label{lem:sigma}
$J_{T}$ defines an isomorphism between
$\g_{-1}\oplus \tilde{\g}_{-1}$ and $W_{\gamma,\tilde{\gamma}}$. Furthermore,
\[
\langle J_{T}(R_{1},\tilde{R}_{1}),J_{T}(R_{2},\tilde{R}_{2})\rangle=-\kappa_{-1}(R_{1},R_{2})+\tilde{\kappa}_{-1}(\tilde{R}_{1},\tilde{R}_{2}),
\]
for all $R_{1}$, $R_{2}\in \g_{-1}$, $\tilde{R}_{1}$, $\tilde{R}_{2}\in \tilde{\g}_{-1}$. Here the bilinear form $\kappa$ on $\g$ is normalized as in \eqref{knormalized}, and likewise for $\tilde{\kappa}$.
\end{lemma}

\begin{proof}
First observe that if $R_{1}$, $R_{2}\in \g_{-1}$, then
\begin{eqnarray*}
\langle TR_{1},TR_{2} \rangle & = & \Tr ((TR_{1})^{\ast} TR_{2}) = \Tr (R_{1}^{\ast}T^{\ast} TR_{2})\\
                              & = & -\Tr (R_{1}XR_{2}) =\Tr(R_{2}XR_{1})\\
                              & = & -\kappa_{-1}(R_{1},R_{2}).
\end{eqnarray*}
Similarly $\langle
\tilde{R}_{1}T,\tilde{R}_{2}T \rangle=\tilde{\kappa}_{-1}(\tilde{R}_{1},\tilde{R}_{2})$. On the other hand, if $R\in \g_{-1}$, $\tilde{R}\in \tilde{\g}_{-1}$, then
\[
\langle TR,\tilde{R}T\rangle =\Tr(R^{\ast}T^{\ast}\tilde{R}T)=-\Tr(RT^{\ast}\tilde{R}T)
\]
and
\[
\langle\tilde{R}T,TR \rangle =\Tr(T^{\ast}\tilde{R}^{\ast}TR)=-\Tr(T^{\ast}\tilde{R}TR).
\]
Therefore,
$\langle \tilde{R}T,TR\rangle =0$, since
$\langle \cdot, \cdot \rangle $ is symplectic. We have thus shown that
 \[
\langle J_{T}(R_{1},\tilde{R}_{1}),J_{T}(R_{2},\tilde{R}_{2})\rangle=-\kappa_{-1}(R_{1},R_{2})+\tilde{\kappa}_{-1}(\tilde{R}_{1},\tilde{R}_{2}),
\]
for all $R_{1}$, $R_{2}\in \g_{-1}$, $\tilde{R}_{1}$, $\tilde{R}_{2}\in \tilde{\g}_{-1}$. But now, since $\kappa_{-1}$ and $\tilde{\kappa}_{-1}$ are non-degenerate, we conclude that $J_{T}$ is injective.

Now observe that, for all $S_k\in \Hom(V_{k},V_{k-1})$, $S_k-S_k^{\ast}\in\g_{-1}$. On the other hand, we can write $S \in \g_{-1}$ as the direct sum $\sum_{k=-r}^{r}S_k$, where $S_k\in \Hom(V_{k},V_{k-1})$. Then we have that $S_k^*=-S_{-k+1}$ for all $k$, and hence, the map $S\mapsto S-S^{\ast}$ defines a linear isomorphism between $\oplus_{k\leq 0} \Hom(V_{k},V_{k-1})$ and $\g_{-1}$. From this, we conclude that
\begin{eqnarray*}
\dim \g_{-1}& = & \sum_{k=0}^{-r+1} \dim V_{k}\cdot \dim V_{k-1} \\
                    & = & \sum_{k=0}^{-r+1} \dim \tilde{V}_{k-1}\cdot \dim V_{k-1} \\
                    & = & \sum_{k=-1}^{-r} \dim \tilde{V}_{k}\cdot \dim V_{k}.
\end{eqnarray*}
Here we have used the fact that $\dim V_{k} =\dim \tilde{V}_{k-1}$ for $k\leq 0$, in view of the linear isomorphism $T^*_{-k}: \tilde{V}_{k-1}\rightarrow V_{k}$. Similarly,
\begin{eqnarray*}
\dim \tilde{\g}_{-1}& = & \sum_{k=0}^{-r} \dim \tilde{V}_{k}\cdot \dim \tilde{V}_{k-1} \\
                    & = & \sum_{k=0}^{-r} \dim \tilde{V}_{k}\cdot \dim V_{k}\\
                    & = & \sum_{k=0}^{r} \dim \tilde{V}_{k}\cdot \dim V_{k}.
\end{eqnarray*}
Therefore,
\begin{eqnarray*}
\dim \g_{-1}+\dim \tilde{\g}_{-1} & = & \sum_{k=-1}^{-r}\dim \tilde{V}_{k}\cdot \dim V_{k}+\sum_{k=0}^{r}\dim \tilde{V}_{k}\cdot \dim V_{k} \\
                              & = & \sum_{k=-r}^{r} \dim \Hom(V_{k},\tilde{V}_{k})=\dim W_{\gamma,\tilde{\gamma}},
\end{eqnarray*}
and hence, $J_{T}$ must be a bijection.
\end{proof}

\section{Generalized Whittaker models and Howe correspondence}
\label{liftWhittaker}

\subsection{The smooth oscillator-Heisenberg representation $\HH_{\gamma,\tilde{\gamma}}$ associated to $W_{\gamma,\tilde{\gamma}}$}
\label{Heisenberg} Let
$H_{\gamma,\tilde{\gamma}}$ be the Heisenberg group associated to
the symplectic space $W_{\gamma,\tilde{\gamma}}$. That is,
$H_{\gamma,\tilde{\gamma}}=W_{\gamma,\tilde{\gamma}}\times \k$,
where $\{0\}\times \k$ is central, and $(R,0)(S,0) =(R+S,\langle R,
S\rangle/2)$. Let $\Mp(W_{\gamma,\tilde{\gamma}})$ be the
metaplectic group associated to $W_{\gamma,\tilde{\gamma}}$, and
$(\tau_{\gamma,\tilde{\gamma}},\HH_{\gamma,\tilde{\gamma}})$ be the
smooth oscillator-Heisenberg representation of
$\Mp(W_{\gamma,\tilde{\gamma}})\ltimes H_{\gamma,\tilde{\gamma}}$
associated to the character $\psi$.

Recall that we have fixed $T \in \Orb^{\Max}_{\gamma,\tilde{\gamma}}$, and we have an isomorphism $J_{T}$ from
$\g_{-1}\oplus \tilde{\g}_{-1}$ to $W_{\gamma,\tilde{\gamma}}$, given in equation (\ref{eq:sigmaT}).
Define a new invariant bilinear form $\kappa'$ on
$\g$ by setting
$\kappa'(R,S)=-\kappa(R,S)$, for all $R$, $S\in
\g$. This results in a new symplectic structure on
$\g_{-1}$, given by
$\kappa_{-1}'(R,S)=-\kappa_{-1}(R,S)$, for all
$R$, $S\in \g_{-1}$, and hence a new Heisenberg group
$H_{\gamma}'$. With this new symplectic structure,
$J_{T}|_{\g_{-1}}$ is a morphism of symplectic spaces,
and hence,  we may extend $J_{T}$ to an injective morphism of groups
$J_{T}: H_{\gamma}'\longrightarrow
H_{\gamma,\tilde{\gamma}}$. Observe that we can proceed similarly for $\tilde{\g}_{-1}$, but for this space the modification of the symplectic structure is unnecessary. We obtain in this way a map
\begin{equation}
\label{eq:extendJT}
J_{T}:H_{\gamma}'\times H_{\tilde{\gamma}}  \longrightarrow   H_{\gamma,\tilde{\gamma}}.
\end{equation}

As in Section \ref{subsec:GWM}, we let
$\alpha_{\gamma}':N\longrightarrow H_{\gamma}'$ be the map induced
by $\gamma$ and
the bilinear form $\kappa'$, 
that is
\[
\alpha_{\gamma}'(\exp R\exp Z)=(R,\kappa'(X,Z)), \ \ \mbox{for all $R\in \g_{-1}$, $Z\in \u$.}
\]
Similarly, define $\alpha_{\tilde{\gamma}}:\tilde{N}\longrightarrow H_{\tilde{\gamma}}$. 
By composing with the map $J_{T}$, we have a group
homomorphism
\begin{equation}
\label{eq:extendJT2}
\alpha_{T}=\alpha_{T,\gamma,\tilde{\gamma}}: N\times \tilde{N} \longrightarrow H_{\gamma,\tilde{\gamma}}
\end{equation}
given by $\alpha_{T}(n,\tilde{n})=J_{T}(\alpha_{\gamma}'(n),\alpha_{\tilde{\gamma}}(\tilde{n}))$.

Now observe that for all $\tilde{m}\in \tilde{M}_{\tilde{X}}$, $\tilde{R}\in \tilde{\g}_{-1}$,
\begin{equation}
J_{T}(\Ad(\tilde{m})\tilde{R})=\tilde{m}\tilde{R}\tilde{m}^{-1}T=\tilde{m}\tilde{R}T\phi_{T}(\tilde{m})^{-1}. \label{eq:tildeMXaction}
\end{equation}
On the other hand, since
$\phi_{T}:\tilde{M}_{\tilde{X}}\longrightarrow M_{X}$ is surjective,
then for any $m\in M_{X}$ we can find $\tilde{m}\in
\tilde{M}_{\tilde{X}}$ such that $\phi_{T}(\tilde{m})=m$. Moreover,
for all $R\in \g_{-1}$,
\begin{equation}
J_{T}(\Ad(\phi_{T}(\tilde{m}))R)=T\phi_{T}(\tilde{m})R\phi_{T}(\tilde{m})^{-1}=\tilde{m}TR\phi_{T}(\tilde{m})^{-1}. \label{eq:MXaction}
\end{equation}
Using equations (\ref{eq:tildeMXaction}) and (\ref{eq:MXaction}), we
define an action of $M_{X}\times \tilde{M}_{\tilde{X}}$ on
$W_{\gamma,\tilde{\gamma}}$, by
\[
\tilde{m}\cdot S =\left\{ \begin{array}{cl}
\tilde{m}S\phi_{T}(\tilde{m})^{-1} & \mbox{if $S\in J_{T}(\tilde{\g}_{-1})$}\\
S & \mbox{if $S\in J_{T}(\g_{-1})$,}\end{array}\right.
\]
if $\tilde{m}\in \tilde{M}_{\tilde{X}}$, and
\[
m \cdot S =\left\{ \begin{array}{cl}
S & \mbox{if $S\in J_{T}(\tilde{\g}_{-1})$}\\
\tilde{m}S\phi_{T}(\tilde{m})^{-1} & \mbox{if $S\in J_{T}(\g_{-1})$,}
\end{array}\right.
\]
if $m=\phi_{T}(\tilde{m})\in M_{X}$, for some $\tilde{m}\in
\tilde{M}_{\tilde{X}}$. Observe that in particular, for all
$\tilde{m}\in \tilde{M}_{\tilde{X}}$, and all $S\in
W_{\gamma,\tilde{\gamma}}$, we have
\begin{equation}
(\tilde{m},\phi_{T}(\tilde{m}))\cdot S = \tilde{m}S\phi_{T}(\tilde{m})^{-1}. \label{eq:combinedMXtildeMXaction}
\end{equation}
Using this action and the functorial property of the oscillator-Heisenberg representation \cite{HoUP}, we extend the map $\alpha_{T}$ to a group homomorphism
\begin{equation}
\alpha_{T}:M_{\chi_{\gamma}}N\times \tilde{M}_{\chi_{\tilde{\gamma}}}\tilde{N}\longrightarrow \Mp(W_{\gamma,\tilde{\gamma}})\ltimes H_{\gamma,\tilde{\gamma}}. \label{eq:extendedalphaTmap}
\end{equation}
By pulling back, this yields a representation
\begin{equation}
\label{eq:deftaut}
\tau_{\gamma,\tilde{\gamma}}^{T}:=\tau_{\gamma,\tilde{\gamma}}\circ \alpha_{T}
\end{equation}
of $M_{\chi_{\gamma}}N\times \tilde{M}_{\chi_{\tilde{\gamma}}}\tilde{N}$ on $\HH_{\gamma,\tilde{\gamma}}$.
Observe that for all $Z\in \u$, $v\in \HH_{\gamma,\tilde{\gamma}}$,
\begin{eqnarray*}
\tau_{\gamma,\tilde{\gamma}}^{T}(\exp Z,\tilde{e})v& = & \tau_{\gamma,\tilde{\gamma}}(0,\kappa'(X,Z))v \\
                                           &= & \psi(\kappa'(X,Z))v=\psi(-\kappa(X,Z))v \\
                                          & = & \chi_{\check{\gamma}}(\exp Z)v=\check{\chi}_{\gamma}(\exp Z)v,
\end{eqnarray*}
where $\tilde{e}$ is the identity element in $\tilde{M}_{\chi_{\tilde{\gamma}}}\tilde{N}$. Similarly, for all $Z\in \u$, $v\in \HH_{\gamma,\tilde{\gamma}}$,
\[
\tau_{\gamma,\tilde{\gamma}}^{T}(e,\exp \tilde{Z})v=\chi_{\tilde{\gamma}}(\exp \tilde{Z})v,
\]
where $e$ is the identity element in $M_{\chi_{\gamma}}N$. Using this, and the fact that $J_T: \g_{-1}\oplus\tilde{\g}_{-1}\rightarrow W_{\gamma,\tilde{\gamma}}$ is an isomorphism, we obtain the following result.

\begin{lemma}\label{lemma:tauTisomorphism} As $M_{\chi_{\gamma}}N\times
\tilde{M}_{\chi_{\tilde{\gamma}}}\tilde{N}$-modules, we have
\[
\tau_{\gamma,\tilde{\gamma}}^{T}\cong \S_{\check{\chi}_{\gamma}}\otimes
\S_{\chi_{\tilde{\gamma}}},
\]
where $\tau_{\gamma,\tilde{\gamma}}^{T}=\tau_{\gamma,\tilde{\gamma}}\circ \alpha_{T}$.
\end{lemma}

\subsection{The main result and the key proposition}
We recall some notation from Section \ref{sec:liftOrbit}. Let $(V,B)$ and $(\tilde{V},\tilde{B})$ be an $\epsilon$-Hermitian and an $\tilde{\epsilon}$-Hermitian $D$-module, respectively, with $\varepsilon\tilde{\varepsilon}=-1$. Let $G(V)$, $G(\tilde{V})$ be the corresponding isometry groups and $\g=\g(V)$, $\tilde{\g}=\g(\tilde{V})$ their Lie algebras. Then $(G(V),G(\tilde{V}))$ form a dual pair in $\Sp(\Hom(V,\tilde{V}))$ in the sense of Howe \cite{Ho79}.

 Let $\Mp(\Hom(V,\tilde{V}))$ be the metaplectic group. This is the unique topological central extension of $\Sp(\Hom(V,\tilde{V}))$ by $\{\pm 1\}$ such that it splits if $\k$ is $\C$, and it does not split otherwise. Denote the pullbacks of $G(V)$, $G(\tilde{V})$ in $\Mp(\Hom(V,\tilde{V}))$ by $G$ and $\tilde{G}$, respectively. Then $(G,\tilde{G})$ form a dual pair in $\Mp(\Hom(V,\tilde{V}))$.
 Let $(\omega,\Y)$ be the smooth oscillator representation of $\Mp(\Hom(V,\tilde{V}))$ associated to the character $\psi$.
It is well-known that after tensoring with a genuine character of $G\times \tilde{G}$, $(\omega , \Y)$ yields a representation of $G(V)\times G(\tilde{V})$, except when one of them, say $V$, is an odd dimensional quadratic space and $\k\ne \C$. In this exceptional case, $\tilde{G}$ is the metaplectic cover of $G(\tilde{V})$, and we shall only consider genuine representations of $\tilde{G}$. For all other cases,
representations of $G(V)$ are identified with genuine representations of $G$ (by tensoring with a fixed genuine character of $G$). With this convention/caveat in mind and for the rest of this article, we will make little distinction between $G$ and $G(V)$, and loosely speak of genuine representations of $G$. Likewise for $\tilde{G}$ and $G(\tilde{V})$.
This convention will lead to significant savings in notation and no confusion is expected for the expert reader.

We state the main result of this article, in a more concrete form than Theorem \ref{MainThm} of the introduction.

\begin{theorem}\label{thm:maintheorem}
Let $(G,\tilde{G})$ be a reductive dual pair, and let $(\omega , \Y
)$ be the smooth oscillator representation associated to the dual
pair $(G,\tilde{G})$ and to the character $\psi$ of $\k$. Let
$(\pi,\mathscr{V})$ be a smooth irreducible genuine representation
of $G$. Let $\Orb\subset \g$ be a
nilpotent $G$-orbit in the image of $\Max \Hom (V,\tilde{V})$ under
the moment map $\varphi$ and let $\tilde{\Orb}=\Theta(\Orb)\subset
\tilde{\g}$ be the corresponding nilpotent $\tilde{G}$-orbit. Then
\[
\Wh_{\tilde{\Orb}}(\Theta(\pi))\cong \Wh_{\Orb}(\check{\pi}),
\]
where $\check{\pi}$ is the contragredient representation of $\pi$. More precisely, let $\gamma$ and $\tilde{\gamma}$ be two $\mathfrak{sl}_{2}$-triples of type $\Orb$ and $\tilde{\Orb}$, respectively. Then, given any $T\in \Orb_{\gamma,\tilde{\gamma}}^{\Max}$, there exists an isomorphism
\[
\Phi_{T}:\Wh_{\gamma}(\check{\pi})\longrightarrow \Wh_{\tilde{\gamma}}(\Theta(\pi)),
\]
such that
\[
\tilde{m}\Phi_{T}(\lambda)=\Phi_{T}(\phi_{T}(\tilde{m})\lambda) \qquad \mbox{for all $\tilde{m}\in\tilde{M}_{\chi_{\tilde{\gamma}}}$, $\lambda\in \Wh_{\gamma}(\check{\pi})$},
\]
where $\phi_{T}:\tilde{M}_{\chi_{\tilde{\gamma}}}\twoheadrightarrow
M_{\chi_{\gamma}}$ is as in equation (\ref{defphi}).
\end{theorem}

The key to our proof of Theorem \ref{thm:maintheorem} is the following result, which computes $\Y_{\tilde{U},\chi_{\tilde{\gamma}}}$, the $(\tilde{U},\chi_{\tilde{\gamma}})$-isotypic quotient of $\Y$. See Section \ref{frechet} for the unexplained notation.

\begin{proposition}\label{prop:mainproposition} 
 Let $\gamma\subset \g$ and $\tilde{\gamma}\subset \tilde{\g}$ be two $\sl_{2}$-triples of type $\Orb$ and $\Theta(\Orb)$, respectively. Then, given any $T\in \Orb_{\gamma,\tilde{\gamma}}^{\Max}$, there exists a $ G\times \tilde{M}_{\chi_{\tilde{\gamma}}}\tilde{N}$-intertwining isomorphism
\begin{equation}
\label{covariants}
\Psi_{T}:\Y_{\tilde{U},\chi_{\tilde{\gamma}}}\longrightarrow \Sch(N\backslash G;\HH_{\gamma,\tilde{\gamma}}),
\end{equation}
where $M_{\chi_{\gamma}}N\times \tilde{M}_{\chi_{\tilde{\gamma}}}\tilde{N}$ acts on $\HH_{\gamma,\tilde{\gamma}}$ via the representation $\tau_{\gamma,\tilde{\gamma}}^{T}$ defined in equation (\ref{eq:deftaut}), and the action of $G\times \tilde{M}_{\chi_{\tilde{\gamma}}}\tilde{N}$ on $\Sch(N\backslash G;\HH_{\gamma,\tilde{\gamma}})$ is defined in the following way: given $f\in \Sch(N\backslash G;\HH_{\gamma,\tilde{\gamma}})$ and $g\in G$,
\begin{eqnarray}
(g'\cdot f)(g) & = & f(gg') \qquad \mbox{for all $g'\in G$,} \label{eq:tildemactionNG1}\\
(\tilde{n} \cdot f)(g) & = & \tau_{\gamma,\tilde{\gamma}}^{T}(\tilde{n})f(g) \qquad \mbox{for all $\tilde{n}\in \tilde{N}$} \label{eq:tildemactionNG2}\\
(\tilde{m}\cdot f)(g) & = &
\tau_{\gamma,\tilde{\gamma}}^{T}((\phi_{T}(\tilde{m}),\tilde{m}))f(\phi_{T}(\tilde{m})^{-1}g)
\qquad \mbox{for all $\tilde{m}\in
\tilde{M}_{\chi_{\tilde{\gamma}}}$.} \label{eq:tildemactionNG3}
\end{eqnarray}
\end{proposition}

Before starting with the proof of this proposition, let us show how it implies Theorem \ref{thm:maintheorem}.

\begin{proof}[Proof (of Theorem \ref{thm:maintheorem})]
By Lemma \ref{lemma:tauTisomorphism}, we have
\[
\Sch(N\backslash G;\HH_{\gamma,\tilde{\gamma}})\cong \Sch(N\backslash G; \S_{\check{\chi}_{\gamma}}\otimes \S_{\chi_{\tilde{\gamma}}})\cong \Sch(N\backslash G;\S_{\check{\chi}_{\gamma}})\otimes \S_{\chi_{\tilde{\gamma}}} .
\]
Now, given $f\in \Sch(N\backslash G;\S_{\check{\chi}_{\gamma}})$ and
$v\in\S_{\chi_{\tilde{\gamma}}}$, we identify $f\otimes v$ with an
element of $\Sch(N\backslash G;\HH_{\gamma,\tilde{\gamma}})$ via the
above isomorphism. With this identification, we may rewrite equation
(\ref{eq:tildemactionNG3}) as
\begin{eqnarray*}
\tilde{m}\cdot(f\otimes v)(g) & = &\tau_{\gamma,\tilde{\gamma}}^{T}(\phi_{T}(\tilde{m}))f(\phi_{T}(\tilde{m})^{-1}g)\otimes \tau_{\gamma,\tilde{\gamma}}^{T}(\tilde{m})v \\
& = & (\phi_{T}(\tilde{m})\cdot f)(g)\otimes \tau_{\gamma,\tilde{\gamma}}^{T}(\tilde{m})v,
\end{eqnarray*}
for all $\tilde{m}\in \tilde{M}_{\chi_{\tilde{\gamma}}}$. See
equation (\ref{eq:mactiononNG}).

By Propositions \ref{prop:mainproposition} and \ref{prop:SchwartzWhittakerModels}, we have the induced
isomorphisms:
\begin{equation}
\label{2qq1}
\Y_{(\tilde{U},\chi_{\tilde{\gamma}}),(G,\pi)}\cong \Sch(N\backslash G;\S_{\check{\chi}_{\gamma}})_{(G,\pi)}\otimes \S_{\chi_{\tilde{\gamma}}}\cong \pi\otimes W_{\gamma}(\check{\pi})\otimes \S_{\chi_{\tilde{\gamma}}},
\end{equation}
where $\tilde{M}_{\chi_{\tilde{\gamma}}}$ acts on $W_{\gamma}(\check{\pi})$ through the map $\phi_{T}:\tilde{M}_{\chi_{\tilde{\gamma}}}\twoheadrightarrow M_{\chi_{\gamma}}$. On the other hand, we have (from the definition of $\Theta(\pi)$) that
\begin{equation}
\label{2qq2}
\Y_{(G,\pi), (\tilde{U},\chi_{\tilde{\gamma}})}\cong \pi \otimes \Theta(\pi)_{(\tilde{U},\chi_{\tilde{\gamma}})}.
\end{equation}

Combining equations \eqref{2qq1} and \eqref{2qq2} and noting that
\[\Wh_{\tilde{\gamma}}(\Theta(\pi))=\Hom_{\tilde N}(\Theta(\pi),\S_{\chi_{\tilde{\gamma}}})= \Hom_{\tilde N}(\Theta(\pi)_{(\tilde{U},\chi_{\tilde{\gamma}})}, \S_{\chi_{\tilde{\gamma}}}),\]
we obtain the required isomorphism
\[
\Phi_{T}:\Wh_{\gamma}(\check{\pi})\longrightarrow \Wh_{\tilde{\gamma}}(\Theta(\pi)),
\]
such that
\[
\tilde{m}\Phi_{T}(\lambda)=\Phi_{T}(\phi_{T}(\tilde{m}\lambda)) \qquad \mbox{for all $\tilde{m}\in\tilde{M}_{\chi_{\tilde{\gamma}}}$, $\lambda\in \Wh_{\gamma}(\check{\pi})$.}
\]
\end{proof}

In the rest of this section we will closely examine the space
$\Y_{\tilde{U},\chi_{\tilde{\gamma}}}$. In fact we will prove a
refinement of Proposition \ref{prop:mainproposition} in which we
give an explicit isomorphism $\Psi_{T}$ in equation
(\ref{covariants}). This is Proposition
\ref{prop:tildechicoinvariants}. As we have already seen, our main
result follows immediately from this proposition.

\subsection{Strategy for the key proposition}

Let $\gamma=\{X,H,Y\}\subset\g$ and $\tilde{\gamma}=\{\tilde{X},\tilde{H},\tilde{Y}\}\subset\tilde{\gamma}$ be two $\sl_{2}$-triples of type $\Orb$ and $\Theta(\Orb)$, respectively. From the proof of Proposition \ref{prop:Gammatilde}, we may assume that
\begin{equation}
\label{gradingvv}
V=\bigoplus_{k=-r}^{r} V_{k}, \qquad  \mbox{and} \qquad \tilde{V}=\bigoplus_{k=-r-1}^{r+1} \tilde{V}_{k},
\end{equation}
for some $r\geq 0$.  ($r=0$ corresponds to the zero orbit.) Now fix $T_{\gamma,\tilde{\gamma}}\in \Orb^{\Max}_{\gamma,\tilde{\gamma}}$, that is, fix an injective element $T_{\gamma,\tilde{\gamma}}\in \Hom(V,\tilde{V})$, such that $T_{\gamma,\tilde{\gamma}}^{\ast}T_{\gamma,\tilde{\gamma}}=X$, $T_{\gamma,\tilde{\gamma}}T_{\gamma,\tilde{\gamma}}^{\ast}=\tilde{X}$ and $T_{\gamma,\tilde{\gamma}}(V_{k})\subset \tilde{V}_{k+1}$ for all $-r\leq k\leq r$. Then we have a decomposition
\begin{equation}
T_{\gamma,\tilde{\gamma}}=\oplus_{k=-r}^{r} T_{k},
\end{equation}
with $T_{k}\in \Hom(V_{k},\tilde{V}_{k+1})$, and $T_{k}$ is an isomorphism for $k\geq 0$. C.f. Lemma \ref{lemma:injsur}. Figuratively
\[
\begin{array}{c}   V_{-r}   \oplus    V_{-r+1}    \oplus    \cdots    \oplus    V_{r-1}   \oplus   V_{r}   \\
  \hspace{45pt} \searrow \hspace{-3pt} {\scriptstyle T_{-r}}  \hspace{10pt}  \searrow \hspace{-3pt} {\scriptstyle T_{-r+1}} \hspace{5pt}  \cdots\hspace{5pt}    \searrow \hspace{-3pt} {\scriptstyle T_{r-1}} \hspace{-5pt}   \searrow \hspace{-3pt} {\scriptstyle T_{r}}    \\
\tilde{V}_{-r-1}    \oplus     \tilde{V}_{-r}     \oplus     \tilde{V}_{-r+1}    \oplus    \cdots    \oplus    \tilde{V}_{r-1}   \oplus   \tilde{V}_{r}    \oplus   \tilde{V}_{r+1}.
\end{array}
\]

We will use the gradings of $V$ and $\tilde{V}$ associated to $H$ and $\tilde{H}$, respectively, to describe the so-called mixed model of the smooth oscillator representation $(\omega,\Y)$ associated to the dual pair $(G,\tilde{G})$. Using this model we will give a convenient description of the space of covariants $\Y_{\tilde{U},\chi_{\tilde{\gamma}}}$ leading to Proposition \ref{prop:tildechicoinvariants}. But before we outline the strategy, let us fix some notation. For $l\leq r$ and $m\leq r+1$, set
\begin{equation}
\label{vvbracket}
V_{(l)}=\bigoplus_{k=-l}^{l}V_{k}, \qquad \mbox{and} \qquad \tilde{V}_{(m)}=\bigoplus_{k=-m}^{m}\tilde{V}_{k}.
\end{equation}
Let $G_{(l)}=\{g\in G\, | \, \mbox{$g\cdot v =v$ for all $v\in V_{(l)}^{\perp}$}\}$, which is isomorphic with $G(V_{(l)})$.
Let $(\omega_{(l),(m)},\Y_{(l),(m)})$ be the smooth oscillator
representation associated to the dual pair
$(G_{(l)},\tilde{G}_{(m)})$ and the character $\psi$ of $\k$. Set
$\S_{(l),m}=\S(\Hom(V_{(l)},\tilde{V}_{m}))$, the Schwartz space of
$\Hom(V_{(l)},\tilde{V}_{m})$. Similarly, set
$\S_{l,(m)}=\S(\Hom(V_{l},\tilde{V}_{(m)}))$ and
$\S_{l,m}=\S(\Hom(V_{l},\tilde{V}_{m}))$. (Note that $\Y_{(l),(m)}$ can be identified with the
Schwartz space of a Lagrangian subspace of $\Hom(V_{(l)},\tilde{V}_{(m)})$, through the Schrodinger model.)

\begin{remark} In the current notation, our original dual pair $(G,\tilde{G})=(G_{(r)},\tilde{G}_{(r+1)})$, and the oscillator representation is
$(\omega,\Y)=(\omega_{(r),(r+1)},\Y_{(r),(r+1)})$.
\end{remark}

Now observe that
\begin{eqnarray*}
\Hom(V_{(r)},\tilde{V}_{(r+1)}) & = & \Hom(V_{(r)},\tilde{V}_{r+1})\oplus \Hom(V_{(r)},\tilde{V}_{-r-1})\oplus \Hom(V_{(r)},\tilde{V}_{(r)}) \\
&= & (\Hom(V_{(r)},\tilde{V}_{r+1})\oplus \Hom(V_{-r},\tilde{V}_{(r)}))\oplus (\Hom(V_{(r)},\tilde{V}_{-r-1}) \oplus \Hom(V_{r},\tilde{V}_{(r)}))\\
&  & {}\oplus \Hom(V_{(r-1)},\tilde{V}_{(r)})
\end{eqnarray*}
and $\Hom(V_{(r)},\tilde{V}_{r+1})\oplus \Hom(V_{-r},\tilde{V}_{(r)})$, $\Hom(V_{(r)},\tilde{V}_{-r-1}) \oplus \Hom(V_{r},\tilde{V}_{(r)})$ are totally isotropic, complementary subspaces. It then follows, from the standard theory of the oscillator representation \cite{HoUP}, that
\begin{equation}\label{eq:polarization}
\begin{aligned}
\Y_{(r),(r+1)}&\cong \S_{(r),r+1}\otimes \Y_{(r),(r)}\\
&\cong [\S_{(r),r+1}\otimes \S_{-r,(r)}]\otimes \Y_{(r-1),(r)}, \ \ \ \ r>0.
\end{aligned}
\end{equation}
What this equation is saying is that we may interpret the space
$\Y_{(r),(r+1)}$ as the space of Schwartz class functions on
$\Hom(V_{(r)},\tilde{V}_{r+1})$ with values in $\Y_{(r),(r)}$, and
the latter, in turn, can be interpreted as the space of Schwartz
class functions on $\Hom(V_{-r},\tilde{V}_{(r)})$ with values in
$\Y_{(r-1),(r)}$. More concretely, we describe this space in the
following way: given $\rho$ a seminorm on $\Y_{(r),(r)}$, $Z\in
\D(\Hom(V_{(r)},\tilde{V}_{r+1}))$ ($\D$ denotes the space of
constant-coefficient differential operators), $d\in \N$ and $f\in
C^{\infty}(\Hom(V_{(r)},\tilde{V}_{r+1});\Y_{(r),(r)})$, set
\[
q_{Z,d,\rho}(f)=\sup_{T\in \Hom(V_{(r)}\tilde{V}_{r+1})} \rho(Zf(T))(1+\|T\|)^{d},
\]
where $\|T\|$ is the operator norm of $T$. Then
\[
\S_{(r),r+1}\otimes \Y_{(r),(r)}\cong \left\{f\in C^{\infty}(\Hom(V_{(r)},\tilde{V}_{r+1});\Y_{(r),(r)})\, \left| \, \begin{array}{c} \mbox{$q_{Z,d,\rho}(f)<\infty$, for all  $d\in \N$,}\\ \mbox{$Z\in  \D(\Hom(V_{(r)},\tilde{V}_{r+1}))$,}\\ \mbox{  $\rho$ seminorm on $\Y_{(r),(r)}$}\end{array}  \right\}, \right.
\]
and likewise for $[\S_{(r),r+1}\otimes \S_{-r,(r)}]\otimes \Y_{(r-1),(r)}$. Proceeding inductively, we obtain the tensor product decomposition
\begin{equation}\label{eq:completepolarization}
\begin{aligned}
\Y_{(r),(r+1)}&\cong [\S_{(r),r+1}\otimes \S_{-r,(r)}]\otimes \cdots \otimes [\S_{(1),2}\otimes \S_{-1,(1)}]\otimes \Y_{(0),(1)}\\
&\cong [\S_{(r),r+1}\otimes \S_{-r,(r)}]\otimes \cdots \otimes [\S_{(1),2}\otimes \S_{-1,(1)}]\otimes \S_{(0),1} \otimes \Y_{(0),(0)}.
\end{aligned}
\end{equation}
Again we may interpret the space appearing in the right hand side of this equation as the space of Schwartz class functions on $\oplus _{k=r,...,1}[\Hom(V_{(k)},\tilde{V}_{k+1})\oplus \Hom(V_{-k},\tilde{V}_{(k)})]
\oplus \Hom(V_{(0)},\tilde{V}_{1})$ with values in $\Y_{(0),(0)}$. This is the mixed model of the smooth oscillator representation associated to the $\gamma$, $\tilde{\gamma}$-gradings. In the rest of this article we will make these identifications of spaces without further explanation.

For simplicity, assume for the moment that $\Orb$ and $\Theta(\Orb)$ are even orbits, that is $\g_{-1}=0$ and $\tilde{\g}_{-1}=0$.
(By Section \ref{subsection:embedding}, this is equivalent to $W_{\gamma,\tilde{\gamma}}=\bigoplus_{k=-r}^{r}
\Hom(V_{k},\tilde{V}_{k})=0$.) In this case we have that $\Y_{(0),(0)}=\C$ and we may define a continuous linear functional $\lambda_{(r)} \in \Y_{(r),(r+1)}'$ by setting
\[
\lambda_{(r)}(f)=f(T_{r},T_{-r},\ldots,T_{1},T_{-1},T_{0}),
\]
where $f\in \Y_{(r),(r+1)}\cong [\S_{(r),r+1}\otimes \S_{-r,(r)}]\otimes \cdots \otimes [\S_{(1),2}\otimes \S_{-1,(1)}]\otimes \S_{(0),1} \otimes \Y_{(0),(0)}$. As we will see later in this section, for such an $f$ we have that
\[
\lambda_{(r)}(\omega_{(r),(r+1)}(n) f)=\chi_{\gamma}^{-1}(n)\lambda_{(r)}(f) \qquad \mbox{for all $n\in U=N$,}
\]
and
\[
\lambda_{(r)}(\omega_{(r),(r+1)}(\tilde{n}) f)=\chi_{\tilde{\gamma}}(\tilde{n})\lambda_{(r)}(f) \qquad \mbox{for all $\tilde{n}\in \tilde{U}=\tilde{N}$.}
\]
Now given $f \in \Y_{(r),(r+1)}$, set
\[
f_{(r)}(g)=\lambda_{(r)}(\omega_{(r),(r+1)}(g) f),
\]
for $g\in G$. It is then immediate that $f_{(r)}\in
C^{\infty}(N\backslash G; \HH_{\gamma,\tilde{\gamma}})$ but we will
show that actually $f_{(r)}\in \Sch(N\backslash G;
\HH_{\gamma,\tilde{\gamma}})$. (Here we note that
$\HH_{\gamma,\tilde{\gamma}}$ is $1$-dimensional since
$W_{\gamma,\tilde{\gamma}}=0$.) On the other hand, since $M_{\chi}$
and $\tilde{M}_{\chi_{\tilde{\gamma}}}$ preserve the gradings on $V$
and $\tilde{V}$, respectively, it is straightforward to check in
this case that
\[
(\omega_{(r),(r+1)}(\tilde{m})f)_{(r)}(g)=f_{(r)}(\phi_{T}(\tilde{m})^{-1}g), \mbox{for all $f\in \Y_{(r),(r+1)}$, $\tilde{m}\in \tilde{M}_{\chi_{\tilde{\gamma}}}$, $g\in G$}.
\]
It then follows that the map $f\mapsto f_{(r)}$ induces a
$G\times \tilde{M}_{\chi_{\tilde{\gamma}}}\tilde{N}$-intertwining
map
\[
\Psi_{T}:
(\Y_{(r),(r+1)})_{\tilde{U},\chi_{\tilde{\gamma}}}\longrightarrow
\Sch(N\backslash G;\HH_{\gamma,\tilde{\gamma}}),
\]
that we will show to be an isomorphism.

To describe the map $\Psi_{T}$ in the general case, observe that if $(\tau_{\gamma,\tilde{\gamma}},\HH_{\gamma,\tilde{\gamma}})$ is the smooth Heisenberg representation associated to $W_{\gamma,\tilde{\gamma}}$ (as in Section \ref{Heisenberg}), then, following arguments similar to those leading to equation (\ref{eq:completepolarization}), we obtain the following tensor product decomposition:
\begin{equation}
\label{Schitchi}
\HH_{\gamma,\tilde{\gamma}} \cong \S_{-r,-r}\otimes
\S_{-r+1,-r+1}\cdots\otimes \S_{-1,-1}\otimes \Y_{(0),(0)}.
\end{equation}
Again we may identify the space $\HH_{\gamma,\tilde{\gamma}}$ with
the space of Schwartz class functions on
$\Hom(V_{-r},\tilde{V}_{-r})\oplus \cdots \oplus
\Hom(V_{-1},\tilde{V}_{-1})$ with values in $\Y_{(0),(0)}$. Now,
given $f\in \Y_{(r),(r+1)}$, set
\begin{equation}
\label{def:fr0}
f_{(r)}(g)(S_{-r},\ldots,S_{-1})=(\omega_{(r),(r+1)}(g) f)(T_{r},T_{-r}+S_{-r},\ldots,T_{1},T_{-1}+S_{-1},T_{0}),
\end{equation}
for all $g\in G$, $S_{-k}\in \Hom(V_{-k},\tilde{V}_{-k})$, $k=1,\ldots,r$.  Note that when $\Orb$ and $\Theta(\Orb)$ are both even this definition agrees with the one given before.

The following result is a refinement of Proposition \ref{prop:mainproposition}.

\begin{proposition}\label{prop:tildechicoinvariants}
 Let $\gamma\subset \g$ and $\tilde{\gamma}\subset \tilde{\g}$ be two $\sl_{2}$-triples of type $\Orb$ and $\Theta(\Orb)$, respectively. Fix $T\in \Orb_{\gamma,\tilde{\gamma}}^{\Max}$. Given $f\in \Y_{(r),(r+1)}$, define $f_{(r)}\in
C^{\infty}(G;\HH_{\gamma,\tilde{\gamma}})$ as in equation (\ref{def:fr0}). Then the map $f\mapsto f_{(r)}$ induces a $G\times \tilde{M}_{\chi_{\tilde{\gamma}}}\tilde{N}$-intertwining isomorphism
\[
\Psi_{T}:(\Y_{(r),(r+1)})_{\tilde{U},\chi_{\tilde{\gamma}}}\longrightarrow
\Sch(N\backslash G;\HH_{\gamma, \tilde{\gamma}}).
\]
where $G\times \tilde{M}_{\chi_{\tilde{\gamma}}}\tilde{N}$ acts on
$\Sch(N\backslash G;\HH_{\gamma, \tilde{\gamma}})$ via equations
(\ref{eq:tildemactionNG1})--(\ref{eq:tildemactionNG3}).
\end{proposition}

As we have already shown, the existence of such a $\Psi_{T}$ is enough to prove Theorem \ref{thm:maintheorem}.
 Before proceeding to the proof of this key proposition, 
 we introduce all the remaining notation.

Recall that we have set
 $G_{(l)}=\{g\in G\, | \, \mbox{$g\cdot v =v$ for all $v\in V_{(l)}^{\perp}$}\}.$ Let $P_{l}$ be the stabilizer of $V_{-l}$
in $G_{(l)}$. Then $P_{l}=M_{(l)}N_{l}$, where $N_{l}$ is the
unipotent radical of $P_{l}$, $M_{(l)}= M_{l}\times G_{(l-1)}$, and $M_{l}\cong \GL(V_{-l})$. Now observe that, if $S\in \Hom(V_{(l-1)},V_{-l})$, then $S-S^{\ast}\in \n_{l}$. Hence, we can use the map $S\mapsto S-S^{\ast}$ to define a Lie algebra isomorphism
\begin{equation}
\Hom(V_{(l-1)},V_{-l})\oplus \z_{l} \cong \n_{l}, \label{eq:nlisomorphism}
\end{equation}
where the Lie algebra structure on the left hand side is given as follows: $\z_{l}=\{Z\in \Hom(V_{l},V_{-l}) \, | \, Z^{\ast}=-Z\}$ is central, and $[T,S]=ST^{\ast}-TS^{\ast}$ for all $T$, $S\in
\Hom(V_{(l-1)},V_{-l})$.  In what follows we will make implicit use of this isomorphism to identify these two
spaces.

In a similar fashion, we define $\tilde{G}_{(m)}\cong
G(\tilde{V}_{(m)})$, and observe that if we set $\tilde{P}_{m}$ to be the
stabilizer of $\tilde{V}_{m}$ in $\tilde{G}_{(m)}$, then
$\tilde{P}_{m}=\tilde{M}_{(m)}\tilde{N}_{m}$, where $\tilde{N}_{m}$ is the
unipotent radical of $\tilde{P}_{m}$, $\tilde{M}_{(m)}=\tilde{M}_{m}\times \tilde{G}_{(m-1)}$, and
$\tilde{M}_{m}\cong \GL (\tilde{V}_{m})$. Similarly to (\ref{eq:nlisomorphism}), we use the map $\tilde{S}\mapsto \tilde{S}-\tilde{S}^{\ast}$ from $\Hom(\tilde{V}_{m},\tilde{V}_{(m-1)})$ to $\tilde{\n}_{m}$ to define a Lie algebra isomorphism
\begin{equation}
\Hom(\tilde{V}_{m},\tilde{V}_{(m-1)})\oplus \tilde{\z}_{m}\cong \tilde{\n}_{m}, \label{eq:tildenmisomorphism}
\end{equation}
where, similarly, the Lie algebra structure on the left hand side is given as follows: $\tilde{\z}_{m} = \{\tilde{Z}\in \Hom(\tilde{V}_{m},\tilde{V}_{-m}) \, | \, \tilde{Z}^{\ast}=-\tilde{Z}\}$
is central, and $[\tilde{T},\tilde{S}]=\tilde{S}^{\ast}\tilde{T}-\tilde{T}^{\ast}\tilde{S}$ for all $\tilde{T}$, $\tilde{S}\in \Hom(\tilde{V}_{m},\tilde{V}_{(m-1)})$.
We will also set
\begin{equation}
\label{nlul}
N_{(l)}=N\cap G_{(l)}, \qquad U_{l}=U\cap N_{l}, \qquad \mbox{and} \qquad U_{(l)}=U\cap N_{(l)},
\end{equation}
with similar definitions for $\tilde{N}_{(m)}$, $\tilde{U}_{m}$ and
$\tilde{U}_{(m)}$. Let $\chi=\chi_{\gamma}$ be the character of $U$
associated to the $\sl_{2}$-triple $\gamma$, and let
$\chi_{l}=\chi|_{U_{l}}$, $\chi_{(l)}=\chi|_{U_{(l)}}$. Similarly we
define $\tilde{\chi}=\chi_{\tilde{\gamma}}$, $\tilde{\chi}_{m}$ and
$\tilde{\chi}_{(m)}$. Finally, define $X_{(l)}\in \End(V_{(l)})$ in
the following way: for $k<l-1$ we define
$X_{(l)}|_{V_{k}}=X|_{V_{k}}$ and for $l-1\leq k \leq l$ set
$X_{(l)}|_{V_{k}}=0$. Then it is clear that $X_{(l)}\in
\g_{(l)}:=\g(V_{(l)})$ is a nilpotent element. Let
$H_{(l)}=H|_{V_{(l)}}$. Then $H_{(l)}\in \g_{(l)}$
is semisimple, and $[H_{(l)},X_{(l)}]=2X_{(l)}$. Therefore, there
exists $Y_{(l)}\in \g_{(l)}$ such that
$\gamma_{(l)}=\{X_{(l)},H_{(l)},Y_{(l)}\}$ is an $\sl_{2}$-triple.
Observe that the parabolic subgroup associated to this
$\sl_{2}$-triple is precisely $P_{\gamma_{(l)}}=P_{\gamma}\cap
G_{(l)}=M_{(l)}N_{(l)}$, where $M_{(l)}=M\cap G_{(l)}$. Furthermore,
observe that $\chi_{\gamma_{(l)}}=\chi_{(l)}$, and hence
$M_{\chi_{\gamma_{(l)}}}=M_{\chi_{(l)}}$. We make analogous
definitions for $\tilde{X}_{(l)}$, $\tilde{\gamma}_{(l)}$, etc.

\vsp
We recall the following explicit formulas for $\omega|_{G\times \tilde{P}_{r+1}}=\omega_{(r),(r+1)}|_{G\times \tilde{P}_{r+1}}$ \cite{HoUP}: for $f\in \Y_{(r),(r+1)}\cong \S_{(r),r+1}\otimes \Y_{(r),(r)}$ and $T\in \Hom(V_{(r)},\tilde{V}_{r+1})$,
\begin{eqnarray}
(\omega_{(r),(r+1)}(g)f)(T) & = & \omega_{(r),(r)}(g)[f(Tg)] \qquad \mbox{for all $g\in G=G(V_{(r)})$},\label{eq:action1}\\
(\omega_{(r),(r+1)}(\tilde{m})f)(T) & = & \nu (\tilde{m})[f(\tilde{m}^{-1}T)] \qquad \mbox{for all $\tilde{m}\in\tilde{M}_{r+1}\cong \GL(\tilde{V}_{r+1})$}, \label{eq:action2} \\
(\omega_{(r),(r+1)}(\tilde{g})f)(T) & = & \omega_{(r),(r)}(\tilde{g})[f(T)] \qquad \mbox{for all $\tilde{g}\in \tilde{G}_{(r)}\cong G(\tilde{V}_{(r)})$}, \label{eq:action3} \\
(\omega_{(r),(r+1)}(\exp \tilde{Z})f)(T) & = & \psi(\Tr \tilde{Z}TT^{\ast}/2)[f(T)] \qquad \mbox{for all $\tilde{Z}\in \tilde{\z}_{r+1}$}, \label{eq:action4} \\
(\omega_{(r),(r+1)}(\exp \tilde{R})f)(T) & = & \omega_{(r),(r)}(-\tilde{R}T)[f(T)] \qquad \mbox{for all $\tilde{R}\in \Hom(\tilde{V}_{r+1},\tilde{V}_{(r)})$}. \label{eq:action5}
\end{eqnarray}
Here $\nu $ is a certain character whose explicit form will not concern us. In the last equation we implicitly understand that $\Y_{(r),(r)}$ carries an action of the Heisenberg group associated to $\Hom(V_{(r)},\tilde{V}_{(r)})$. We also observe the following: since $\tilde{V}_{r+1}\subset \tilde{V}_{(r+1)}$ is stable under the action of $\tilde{P}_{r+1}$, we may consider $\Hom(V_{(r)},\tilde{V}_{r+1})$ as a $G\times \tilde{P}_{r+1}$-module with trivial $\tilde{G}_{(r)}\times \tilde{N}_{r+1}$ ($\subset \tilde{P}_{r+1}$) action. From equations (\ref{eq:action1})--(\ref{eq:action5}), it is then clear that the action $\omega_{(r),(r+1)}|_{G\times \tilde{P}_{r+1}}$ may be interpreted as coming from the usual vector bundle action of $G\times \tilde{P}_{r+1}$ on the trivial vector bundle with base $\Hom(V_{(r)},\tilde{V}_{r+1})$ and the fiber $\Y_{(r),(r)}$. (This action depends on $T$.) To be more precise, we may rewrite equations (\ref{eq:action1})--(\ref{eq:action5}) in terms of an action $\sigma_{T}$ of $G\times \tilde{P}_{r+1}$ on $\Y_{(r),(r)}$, as follows:
\begin{eqnarray}
(\omega_{(r),(r+1)}(g)f)(T)& = &\sigma_{T}(g)f(Tg) \qquad \mbox{for all $g\in G$} \label{eq:vectorbundleaction1}\\
(\omega_{(r),(r+1)}(\tilde{p})f)(T) & = & \sigma_{T}(\tilde{p})f(\tilde{p}^{-1}T) \qquad \mbox{for all $\tilde{p}\in \tilde{P}_{r+1}$,} \label{eq:vectorbundleaction2}
\end{eqnarray}
where $\sigma_{T}(g)=\omega_{(r),(r)}(g)$ for all $g\in G$, and
\[
\sigma_{T}(\tilde{p})=\psi(\Tr \tilde{Z}TT^{\ast}/2)\nu(\tilde{m})\omega_{(r),(r)}(\tilde{g})\omega_{(r),(r)}(\tilde{R}T),
\]
for $\tilde{p}=\tilde{m}\tilde{g}(\exp \tilde{R})(\exp \tilde{Z})$, with $\tilde{m}\in \tilde{M}_{r+1}$, $\tilde{g}\in \tilde{G}_{(r)}$, $\tilde{R}\in \Hom(\tilde{V}_{r+1},\tilde{V}_{(r)})$, $\tilde{Z}\in \tilde{\z}_{r+1}$.

\subsection{The inductive step: relating covariants of $\Y_{(r),(r+1)}$ with $\Y_{(r-1),(r)}$}
\label{subsection:induction}

Given a function $f\in \Y_{(r),(r+1)}\cong [\S_{(r),r+1}\otimes
\S_{-r,(r)}]\otimes \Y_{(r-1),(r)}$, define a new function $f_{r}\in
C^{\infty}(G;\S_{-r,-r}\otimes \Y_{(r-1),(r)})$ by
\begin{equation}
\label{def-fr}
f_{r}(g)(S)=[\omega_{(r),(r+1)}(g)f](T_{r},T_{-r}+S),\qquad \mbox{for all $g\in G$, $S\in\Hom(V_{-r},\tilde{V}_{-r})$},
\end{equation}
where $T_{\gamma,\tilde{\gamma}}=\sum_{k=-r}^{r}T_{k}$ is as before. Then, from equations (\ref{eq:action4}) and (\ref{eq:action5}), we have that for all $f\in \Y_{(r),(r+1)}$ $g\in G$, $S\in \Hom(V_{-r},\tilde{V}_{-r})$,
\begin{eqnarray}
(\omega_{(r),(r+1)}(\exp \tilde{Z})f)_{r}(g)(S) & = & (\omega_{(r),(r+1)}(\exp \tilde{Z})[\omega_{(r),(r+1)}(g)f])(T_{r},T_{-r}+S) \nonumber \\
& = & \psi(\Tr \tilde{Z}T_{r}T_{r}^{\ast}/2)(\omega_{(r),(r+1)}(g)f)(T_{r},T_{-r}+S) \nonumber \\
& = & f_{r}(g)(S), \label{f(r)center}
\end{eqnarray}
for all $\tilde{Z}\in \tilde{\z}_{r+1}$, and
\begin{eqnarray}
(\omega_{(r),(r+1)}(\exp \tilde{R})f)_{r}(g)(S) & = & (\omega_{(r),(r+1)}(\exp \tilde{R})[\omega_{(r),(r+1)}(g)f])(T_{r},T_{-r}+S) \nonumber \\
& = & (\omega_{(r),(r)}(-\tilde{R}T_{r})[\omega_{(r),(r+1)}(g)f](T_{r}))(T_{-r}+S) \nonumber \\
& = & \psi(\Tr ((T_{-r}+S)^{\ast}\tilde{R}T_{r}))[\omega_{(r),(r+1)}(g)f](T_{r},T_{-r}+S) \nonumber \\
& = & \psi(\Tr (T_{-r}^{\ast}\tilde{R}T_{r}))[\omega_{(r),(r+1)}(g)f](T_{r},T_{-r}+S) \nonumber \\
& = & \chi_{\tilde{\gamma}}(\exp \tilde{R})f_{r}(g)(S), \label{f(r)character}
\end{eqnarray}
for all $\tilde{R}\in \Hom(\tilde{V}_{r+1},\tilde{V}_{(r-1)}\oplus \tilde{V}_{-r})\subset \tilde{\u}_{r+1}$, in other words, for all $\tilde{u}\in \tilde{U}_{r+1}$,
\[
(\omega_{(r),(r+1)}(\tilde{u})f)_{r}(g)(S)=\chi_{\tilde{\gamma}}(\tilde{u})f_{r}(g)(S).
\]
 It is then clear that understanding the map $f\mapsto f_{r}$ is an important first step towards understanding the more complicated map $f\mapsto f_{(r)}$.

We start by observing that $G_{(r-1)}N_{r}$ is the stabilizer of $T_{r}$ in $G$. Hence, according to equation (\ref{eq:vectorbundleaction1}) we obtain a representation $(\sigma_{T_{r}},\Y_{(r),(r)})$ of $G_{(r-1)}N_{r}$, where we are viewing $\Y_{(r),(r)}$ as the fiber over the point $T_{r}\in \Hom(V_{(r)},\tilde{V}_{r+1})$. Now, since by definition $\sigma_{T_{r}}(g)\varphi=\omega_{(r),(r)}(g)\varphi$ for all $\varphi\in \Y_{(r),(r)}$, we may use the formulas given in \cite{HoUP} to describe this action explicitly: given $\varphi\in \Y_{(r),(r)}$, $S\in \Hom(V_{-r},\tilde{V}_{(r)})$,
\begin{eqnarray}
(\sigma_{T_{r}}(\exp Z)\varphi)(S)& = & \psi(\Tr S^{\ast}SZ/2)\varphi(S), \qquad \mbox{for all $Z\in \z_{r}$,} \label{eq:simgaTr1}\\
(\sigma_{T_{r}}(\exp R)\varphi)(S)& =&\omega_{(r-1),(r)}(SR)[\varphi(S)], \qquad \mbox{for all $R\in \Hom(V_{(r-1)},V_{-r})$,} \label{eq:simgaTr2} \\
(\sigma_{T_{r}}(g)\varphi)(S)& =&\omega_{(r-1),(r)}(g)[\varphi(S)], \qquad \mbox{for all $g\in G_{(r-1)}$}. \label{eq:simgaTr3}
\end{eqnarray}

Now, given $\varphi \in \Y_{(r),(r)}$, set $\bar{\varphi}(S)=\varphi(T_{-r}+S)$ for all $S\in \Hom(V_{-r},\tilde{V}_{-r})$. It is clear that the map $\varphi\mapsto \bar{\varphi}$ defines a surjection
\[
\bar{\,}: \ \ \Y_{(r),(r)}\twoheadrightarrow \S_{-r,-r}\otimes \Y_{(r-1),(r)}.
\]
 Furthermore, using equations (\ref{eq:simgaTr1})--(\ref{eq:simgaTr3}), we may thus define a representation $(\bar{\sigma}_{T_{r}},\S_{-r,-r}\otimes \Y_{(r-1),(r)})$ of $G_{(r-1)}N_{r}$ explicitly by the following formulas: given $\varphi\in \S_{-r,-r}\otimes \Y_{(r-1),(r)}$, $S\in \Hom(V_{-r},\tilde{V}_{-r})$,
\begin{eqnarray}
(\bar{\sigma}_{T_{r}}(\exp Z)\varphi)(S)& = & \psi(\Tr(T_{-r}+ S)^{\ast}(T_{-r}+S)Z/2)\varphi(S)=\varphi(S), \qquad \mbox{for $Z\in \z_{r}$,} \label{eq:barsimgaTr1}\\
(\bar{\sigma}_{T_{r}}(\exp R)\varphi)(S)& =&\omega_{(r-1),(r)}((T_{-r}+S)R)[\varphi(S)], \qquad \mbox{for $R\in \Hom(V_{(r-1)},V_{-r}),$} \label{eq:barsimgaTr2} \\
(\bar{\sigma}_{T_{r}}(g)\varphi)(S)& =&\omega_{(r-1),(r)}(g)[\varphi(S)], \qquad \mbox{for $g\in G_{(r-1)}$}. \label{eq:barsimgaTr3}
\end{eqnarray}

Observe that if $f\in \Y_{(r),(r+1)}$, then $f_{r}(g)=\overline{[\omega_{(r),(r+1)}(g)f](T_{r})}$. It follows immediately that
\[
f_{r}\in C^{\infty}(G_{(r-1)}N_{r}\backslash G;\S_{-r,-r}\otimes \Y_{(r-1),(r)}). 
\]
Now, given $\rho$ a seminorm in
$\Y_{(r-1),(r)}$, $Z_{1} \in \D(\Hom(V_{-r},\tilde{V}_{-r}))$, $Z_{2}\in
U(\g)$, $d_{1}$, $d_{2}\in \N$ and $f\in
C^{\infty}(G;\S_{-r,-r}\otimes \Y_{(r-1),(r)})$, let
\begin{equation}\label{eq:definitionseminorm}
q_{Z_{1},Z_{2},d_{1},d_{2},\rho}(f)=\sup_{\begin{array}{c} \mbox{\scriptsize $k\in K$, $m\in M_{r}$} \\ \ \mbox{\scriptsize $S\in \Hom(V_{-r},\tilde{V}_{-r})$}\end{array}} \rho(Z_{1}[Z_{2}f(mk)](S)) \|m\|^{d_{1}}(1+\|S\|)^{d_{2}},
\end{equation}
where $K\subset G$ is a maximal compact subgroup as in Section
\ref{norms}. Set
\begin{eqnarray*}
\lefteqn{\Sch(N_{r} G_{(r-1)}\backslash G;\S_{-r,-r}\otimes \Y_{(r-1),(r)})=} \\
& & \set{f\in C^{\infty}(N_{r} G_{(r-1)}\backslash G;\S_{-r,-r}\otimes \Y_{(r-1),(r)})}{\begin{array}{c}\mbox{$q_{Z_{1},Z_{2},d_{1},d_{2},\rho}(f) <\infty$, for all}\\ \mbox{   $q_{Z_{1},Z_{2},d_{1},d_{2},\rho}$ as in (\ref{eq:definitionseminorm})} 
\end{array}}.
\end{eqnarray*}

\begin{lemma}\label{lemma:outerlayer}

For $r>0$, the map $f\mapsto f_{r}$ induces a $G$-intertwining isomorphism
\[
\Psi_{r}:(\Y_{(r),(r+1)})_{\tilde{U}_{r+1},\tilde{\chi}_{r+1}} \longrightarrow
\Sch(N_{r} G_{(r-1)}\backslash G;\S_{-r,-r}\otimes\Y_{(r-1),(r)}),
\]
where $N_{r} G_{(r-1)}$ acts on $\S_{-r,-r}\otimes\Y_{(r-1),(r)}$ by the representation $\bar{\sigma}_{T_{r}}$, given in equations (\ref{eq:barsimgaTr1})--(\ref{eq:barsimgaTr3}).
\end{lemma}
\begin{proof}
Let $\lambda \in
(\Y_{(r),(r+1)}')^{\tilde{U}_{r+1},\tilde{\chi}_{r+1}}$, the $(\tilde{U}_{r+1},\tilde{\chi}_{r+1})$-isotypic subspace of $\Y_{(r),(r+1)}'$. As in the Appendix, we will identify $\lambda$ with the
$\Y_{(r),(r)}'$-valued distribution on
$\Hom(V_{(r)},\tilde{V}_{r+1})$ that sends $f\in \S_{(r),r+1}$ to the linear functional $v\mapsto \lambda(f\otimes v)$. In other words, and with some abuse of notation, we set
\[
\lambda(f)(v):=\lambda(f\otimes v) \qquad \mbox{for all $f\in \S_{(r),r+1}$, $v\in \Y_{(r),(r)}.$}
\]
Now recall that for $\tilde{Z}\in \tilde{\z}_{r+1}$, $f\in \S_{(r),r+1}$ and $T\in \Hom(V_{(r)},\tilde{V}_{r+1})$, we have $\omega_{(r),(r+1)}(\exp \tilde{Z}) f(T)=\psi(\Tr \tilde{Z}TT^{\ast}/2)f(T)$.
Hence, for all such $\tilde{Z}$ we have that
\begin{equation}\label{eq:tildezaction}
\omega_{(r),(r+1)}(\exp \tilde{Z}) \lambda = \psi(\Tr \tilde{Z}TT^{\ast}/2) \lambda,
\end{equation}
that is, the action of $\exp \tilde{Z}$ on $\lambda$ is given by multiplication by the function $T\mapsto \psi(\Tr \tilde{Z}TT^{\ast}/2)$.
On the other hand, since $\lambda\in
(\Y_{(r),(r+1)}')^{\tilde{U}_{r+1},\tilde{\chi}_{r+1}}$ and since $\tilde{\z}_{r+1} \subseteq \u_{r+1}$,
we must have that
\begin{equation}\label{eq:tildezequivariant}
\omega_{(r),(r+1)}(\exp \tilde{Z}) \lambda = \psi(\Tr \tilde{Z}\tilde{X}/2) \lambda=\lambda,
\end{equation}
where the latter equality is due to the fact that $\tilde{Z}\in \g_{-2r-2}(\tilde{V})$, $\tilde{X}\in \g_{2}(\tilde{V})$ and $r>0$.
Define $h_{\tilde{Z}}\in C^{\infty}(\Hom(V_{(r)},\tilde{V}_{r+1}))$ by
$h_{\tilde{Z}}(T)= \psi(\Tr \tilde{Z}TT^{\ast}/2) -1$ for all $T\in \Hom(V_{(r)},\tilde{V}_{r+1})$. Then we can reformulate equations (\ref{eq:tildezaction}) and (\ref{eq:tildezequivariant}) by saying that $ h_{\tilde{Z}}\lambda=0$, for all $\tilde{Z}\in \tilde{\z}_{r+1}$. It then follows immediately that
\[
\begin{aligned}
\supp \lambda\subseteq &\bigcap_{
\tilde{Z}\in \tilde{\z}_{r+1}}
\{T \in \Hom(V_{(r)},\tilde{V}_{r+1})\, | \, h_{\tilde{Z}}(T)=0\} \\
 &=\{T\in
\Hom(V_{(r)},\tilde{V}_{r+1}) \, | \, TT^{\ast}=0\}.
\end{aligned}
\]
Let
\begin{equation}
\label{defMN}
\Hom_{NM}(V_{(r)},\widetilde{V}_{r+1})=\{T\in
\Hom(V_{(r)},\widetilde{V}_{r+1})\, | \, \mbox{$TT^{\ast}=0$}\}, \ \ \text{and}
\end{equation}
\begin{equation}
\label{defGMN}
\Hom_{GNM}(V_{(r)},\widetilde{V}_{r+1})=\{T\in
\Hom(V_{(r)},\widetilde{V}_{r+1})\, | \, \mbox{$TT^{\ast}=0$ and $T$
has maximal rank}\}.
\end{equation}
(Here the subscripts NM and GNM stand for null mappings and generic null mappings, respectively. Note that both sets are locally closed.)

We claim
that the natural map
\[\Y_{(r),(r+1)} \twoheadrightarrow
(\Y_{(r),(r+1)})_{\tilde{U}_{r+1},\tilde{\chi}_{r+1}}
\] factors
through the intermediate space $\S(\Hom_{GNM}(V_{(r)},\tilde{V}_{r+1}))\otimes
\Y_{(r),(r)}$. In other words, if
$\lambda\in (\Y_{(r),(r+1)}')^{\tilde{U}_{r+1},\tilde{\chi}_{r+1}}$,
then we can identify $\lambda$ with a $\Y_{(r),(r)}'$-valued tempered distribution \emph{living on}
$\Hom_{GNM}(V_{(r)},\tilde{V}_{r+1})$. See the Appendix for the terminology ``\emph{living on}''.
So far we have only shown that if $\lambda \in (\Y_{(r),(r+1)}')^{\tilde{U}_{r+1},\tilde{\chi}_{r+1}}$, then $\supp \lambda \subseteq \Hom_{NM}(V_{(r)},\tilde{V}_{r+1})$. Let
\[\A =\{T \in \Hom(V_{(r)},\tilde{V}_{r+1}) \, | \, \mbox{$T$ is
singular}\}.\]
Clearly $\Hom_{GNM}(V_{(r)},\tilde{V}_{r+1})=\Hom_{NM}(V_{(r)},\tilde{V}_{r+1})\cap \A^{c}$, where $\A^{c}$ is the complement of $\A$. So, as a first step towards proving our claim, we will show that
\begin{equation}\label{eq:claimvanish}
\mbox{if $\lambda \in  (\Y_{(r),(r+1)}')^{\tilde{U}_{r+1},\tilde{\chi}_{r+1}}$} \quad \mbox{ and}\quad \mbox{ $\supp \lambda \subseteq \A,$} \qquad \mbox{then}\quad \mbox{ $\lambda=0.$}
\end{equation}

Observe that $\A \subset \Hom(V_{(r)},\tilde{V}_{r+1})$ is invariant under the natural action of $G(V_{(r)})\times \GL(\tilde{V}_{r+1})$ and decomposes as a finite union of orbits. (Actually, $\Hom(V_{(r)},\tilde{V}_{r+1})$ itself decomposes as a finite union of orbits.) Furthermore, if $\mathcal{C} \subset\A$ is a closed invariant subset, then there exists an orbit $\Orb\subset \mathcal{C}$ that is open in $\mathcal{C}$. Therefore, we may use a Bruhat type argument to prove statement (\ref{eq:claimvanish}). To be precise, we will show that if $\lambda \in (\Y_{(r),(r+1)}')^{\tilde{U}_{r+1},\tilde{\chi}_{r+1}}$, $\mathcal{C}\subset \A$ is a non-empty closed invariant subset containing $\supp \lambda$ and $\Orb \subset \mathcal{C}$ is a relatively open orbit, then $\lambda$ vanishes in the complement of $\mathcal{C}\backslash \Orb$. Once this is done, a straightforward inductive argument finishes the proof of statement (\ref{eq:claimvanish}).

If $\k$ is non-Archimedian then we may immediately identify $\lambda$ with a distribution living on $\mathcal{C}$. In this case, by restricting $\lambda$ to $\Orb$ we obtain a well-defined distribution on this orbit. So we will just need to focus on the case where $\k$ is Archimedian. Given any $x\in \Orb$, there exists a neighborhood $U_{x}\subset \Hom(V_{(r)},\tilde{V}_{r+1})$ of $x$ such that the order of $\lambda|_{U_{x}}$ is $\leq l$ for some $l\in \N$. Therefore, we may find a locally finite open cover $\{U_{i}\}_{i\in\N}$ of $\Hom(V_{(r)},\tilde{V}_{r+1})$ and a partition of unity $\{\rho_{i}\}$ subordinated to $\{U_{i}\}$ such that $\lambda =\sum_{i}\rho_{i} \lambda$ and each $\rho_{i}\lambda$ is of order $\leq l_{i}$, for some $l_{i}\in \N$. Now, since multiplication by a $C^{\infty}$-function clearly commutes with the action of $\tilde{U}_{r+1}$, we have that, if $\lambda \in (\Y_{(r),(r+1)}')^{\tilde{U}_{r+1},\tilde{\chi}_{r+1}}$, then $\rho_{i}\lambda$ is also in $(\Y_{(r),(r+1)}')^{\tilde{U}_{
r+1},\tilde{\chi}_{r+1}}$ for all $i$. The upshot is that in order to prove statement (\ref{eq:claimvanish}) it is enough to just consider distributions $ \lambda \in (\Y_{(r),(r+1)}')^{\tilde{U}_{r+1},\tilde{\chi}_{r+1}}$ of finite order. But in such a case, to prove the vanishing of $\lambda$ in the complement of $\mathcal{C}\backslash \Orb$, it suffices to show that $D'(\Orb;\Y_{(r),(r)}\otimes M^{(l)'})^{\tilde{U}_{r+1},\tilde{\chi}_{r+1}}=0$ for all $l\geq 0$ \cite[Vanishing Theorem 3.15]{KV96}. Here $M^{(l)}$ is the $l$-th transverse jet bundle of $\Orb$ (see the Appendix). More details on the transverse jet bundle in the context of a Lie group action can be found in \cite[Sections 2 and 3]{KV96}.
Now observe that, given $S\in \Hom(V_{(r)},\tilde{V}_{r+1})$ and $\tilde{Z}\in \tilde{\z}_{r+1}$,
\begin{eqnarray*}
\frac{\partial}{\partial S}(\exp \tilde{Z})\cdot f(T) & = & \left. \frac{d}{dt}\right|_{t=0}\psi(\Tr(T+tS)(T+tS)^{\ast}\tilde{Z}/2)f(T+tS)\\
& = & \psi(\Tr TT^{\ast}\tilde{Z}/2)\psi(\Tr ST^{\ast}\tilde{Z}/2)f(T)+\psi(\Tr TT^{\ast}\tilde{Z}/2)\psi(\Tr TS^{\ast}\tilde{Z}/2)f(T)\\ & & {}+\psi(\Tr TT^{\ast} \tilde{Z}/2)\frac{\partial f}{\partial S}(T),
\end{eqnarray*}
and if $R\in \Hom(\tilde{V}_{r+1},\tilde{V}_{(r)}) \subset \tilde{\n}_{r+1}$, then
\begin{eqnarray*}
\frac{\partial}{\partial S}(\exp R)\cdot f(T) & = & \left. \frac{d}{dt}\right|_{t=0}\omega_{(r),(r)}(R(T+tS))f(T+tS)\\
                             & = & \left. \frac{d}{dt}\right|_{t=0}\omega_{(r),(r)}(RT)\omega_{(r),(r)}(RtS)f(T+tS)\\
                             & = & \omega_{(r),(r)}(RT)d\omega_{(r),(r)}(RS)f(T)+\omega_{(r),(r)}(RT)\frac{\partial}{\partial S}f(T).
\end{eqnarray*}
What the above couple of equations imply is that the action of $\tilde{N}_{r+1}$ (and in particular the action of $\tilde{U}_{r+1}$) on the $l$-th transverse bundle is trivial, and hence $D'(\Orb;\Y_{(r),(r)}\otimes (M^{(l)})')^{\tilde{U}_{r+1},\tilde{\chi}_{r+1}}\cong M^{(l)}\otimes D'(\Orb;\Y_{(r),(r)})^{\tilde{U}_{r+1},\tilde{\chi}_{r+1}}$. From all this we see that, in order to prove statement (\ref{eq:claimvanish}), it suffices to show that
\begin{equation}
\label{vanish-orbit}
D'(\Orb;\Y_{(r),(r)})^{\tilde{U}_{r+1},\tilde{\chi}_{r+1}}=0
\end{equation}
for every orbit $\Orb \subset \A$.

At this point we come back to the general case, where $\k$ can be
Archimedean or non-Archimedean. Let $\lambda \in D'(\Orb;\Y_{(r),(r)})^{\tilde{U}_{r+1},\tilde{\chi}_{r+1}}$, where $\Orb \subset \A$ is a $G(V_{(r)})\times \GL(\tilde{V}_{r+1})$-orbit. As we have already seen, the equivariance of $\lambda$ implies that $\supp \lambda \subseteq \Orb\cap \Hom_{NM}(V_{(r)},\tilde{V}_{r+1})$. Since the latter set is $G(V_{(r)})\times \GL(\tilde{V}_{r+1})$ invariant, we see that actually $\Orb \subset \Hom_{NM}(V_{(r)},\tilde{V}_{r+1})$. From equations (\ref{eq:vectorbundleaction1}) and (\ref{eq:vectorbundleaction2}), it is clear that once we fix any $T\in \Orb$ we could describe $C_{c}^{\infty}(\Orb)\otimes \Y_{(r),(r)}$ as an induced representation. (Here we are using the identifications $G\cong G(V_{(r)})$ and $\tilde{M}_{r+1}\cong\GL(\tilde{V}_{r+1})$.) Translating under the action of $G(V_{(r)})$, if necessary, we may assume that $T\in \Orb$ satisfies  $T|_{\oplus_{k=-r}^{r-1}V_{k}}=0$. To see why, let $T\in \Hom_{NM}(V_{(r)},\tilde{V}_{r+1})$ ($\subseteq \Hom_{NM}(V_{(r)},\tilde{
V}_{(r+1)})$), and so $T^{\ast}|_{\tilde{V}_{r+1}^{\perp}}=0$. Since $\tilde{V}_{r+1}^{\perp}=\tilde{V}_{(r)}\oplus\tilde{V}_{r+1}$ we have that $\Im T^{\ast}$ ($\subset V_{(r)}$) is a totally isotropic subspace with dimension less than or equal to $\dim \tilde{V}_{-r-1}=\dim V_{-r}$. We may thus assume, by translating under $G(V_{(r)})$ if necessary, that $\Im T^{\ast}\subset V_{-r}$, which will imply that $T|_{\oplus_{k=-r}^{r-1}V_{k}}=0$. We identify such a $T$ with an element of $\Hom(V_{r},\tilde{V}_{r+1})\subset \Hom(V_{(r)},\tilde{V}_{r+1})$ in an obvious way. Now we have an isomorphism of $G\times \tilde{P}_{r+1}$-modules:
\begin{equation}\label{eq:compatlysupportedinduction}
C_{c}^{\infty}(\Orb)\otimes \Y_{(r),(r)}\cong C^{\infty}_c-\Ind_{\Stab_{T}}^{G\times
\tilde{P}_{r+1}} (\Y_{(r),(r)}).
\end{equation}
 Here $\Stab_{T}$ denotes the stabilizer of $T$ in $G\times \tilde{P}_{r+1}$ and the notation $C^{\infty}_c-\Ind$ means that
we are taking compactly supported induction. The isomorphism is given by the map $f\mapsto \hat{f}$, where we have set $\hat{f}(g,\tilde{p}) =[\omega_{(r),(r+1)}(g)\omega_{(r),(r+1)}(\tilde{p})f](T)$, for any given $f\in C_{c}^{\infty}(\Orb)\otimes \Y_{(r),(r)}$. Its inverse is given by the map $h\mapsto \check{h}$,  where $\check{h}(\tilde{p}^{-1}Tg)=\sigma_{T}(g)^{-1}\sigma_{T}(\tilde{p})^{-1}h(g,\tilde{p})$ for all $h\in C^{\infty}_c-\Ind_{\Stab_{T}}^{G\times
\tilde{P}_{r+1}} (\Y_{(r),(r)})$, $g\in G$, $\tilde{p}\in \tilde{P}_{r+1}$.
Let us describe the action on the right hand side of equation (\ref{eq:compatlysupportedinduction}) more precisely. For this we will use the isomorphism $\Y_{(r),(r)}\cong \S_{-r,(r)}\otimes \Y_{(r-1),(r)}$. According to equations (\ref{eq:action1})--(\ref{eq:vectorbundleaction2}), the action of $ \tilde{G}_{(r)}\tilde{N}_{r+1}$ on $\Y_{(r),(r)}$ is given by
\begin{eqnarray}
\sigma_{T}(\tilde{g})h(S) & = &\omega_{(r-1),(r)}(\tilde{g}) h(\tilde{g}^{-1}S) \qquad \mbox{for $\tilde{g}\in \tilde{G}_{(r)}$}, \label{eq:sigmaaction1} \\
\sigma_{T}(\exp \tilde{Z})h(S) & = &\psi(\Tr \tilde{Z}TT^{\ast}/2)h (S) \qquad \mbox{for $\tilde{Z}\in \tilde{\z}_{r+1}$}, \label{eq:sigmaaction2}\\
\sigma_{T}(\exp \tilde{R})h(S) & = & \omega_{(r),(r)}(\tilde{R}T)h(S)=\psi(\Tr\tilde{R}TS^{\ast})h(S) \qquad \mbox{for $\tilde{R}\in \Hom(\tilde{V}_{r+1},\tilde{V}_{(r)})$.} \label{eq:sigmaaction3}
\end{eqnarray}
In the last equality, we have used the fact that, since $T\in \Hom(V_{r},\tilde{V}_{r+1})\subset \Hom(V_{(r)},\tilde{V}_{r+1})$, then $\tilde{R}T$ is in $\Hom(V_{r},\tilde{V}_{(r)}) \subset \Hom(V_{(r)},\tilde{V}_{(r)})$ which is a totally isotropic subspace that is complementary to $\Hom(V_{-r},\tilde{V}_{(r)})$. (Here is where the choice of $T$ simplifies matters.)
Now observe that there is a continuous $G\times \tilde{P}_{r+1}$-intertwining map from $C_{c}^\infty(G\times  \tilde{P}_{r+1}\times \Hom(V_{-r},\tilde{V}_{(r)}); \Y_{(r-1),(r)})$ to  $C^{\infty}_c-\Ind_{\Stab_{T}}^{G\times
\tilde{P}_{r+1}} (\S_{-r,(r)}\otimes \Y_{(r-1),(r)})$, with dense image, given by
\[
\tilde{h}(g,\tilde{p},S)=\int_{\Stab_{T}}(\sigma_{T}(x)^{-1}h(x_{g}g,x_{\tilde{p}}\tilde{p})(S)\, dx, \qquad \mbox{ $h\in C_{c}^\infty(G\times  \tilde{P}_{r+1}\times \Hom(V_{-r},\tilde{V}_{(r)}); \Y_{(r-1),(r)})$.}
\]
Here $x_{g}$ and $x_{\tilde{p}}$ are the projections of $x\in \Stab_{T}$ to $G$ and $\tilde{P}_{r+1}$, respectively.
Therefore, we may define a $\tilde{U}_{r+1}$-equivariant $\Y_{(r-1),(r)}'$-valued distribution $\bar{\lambda}$ on $G\times \tilde{P}_{r+1}\times \Hom(V_{-r},V_{(r)})$ by setting
\[
\bar{\lambda}(h)=\lambda(\check{\tilde{h}}), \qquad \mbox{for all $h\in C_{c}^\infty(G\times  \tilde{P}_{r+1}\times \Hom(V_{-r},\tilde{V}_{(r)}); \Y_{(r-1),(r)})$.}
\]
It is clear that if $\tilde{p}=\tilde{m}\tilde{g}\tilde{n}$, with $\tilde{m}\in \tilde{M}_{r+1}\cong\GL(\tilde{V}_{r+1})$, $\tilde{n}\in \tilde{N}_{r+1}$, and $\tilde{g}\in \tilde{G}_{(r)}$, then, for all $h\in  C^{\infty}_c-\Ind_{\Stab_{T}}^{G\times
\tilde{P}_{r+1}} (\S_{-r,(r)}\otimes \Y_{(r-1),(r)})$, $S\in \Hom(V_{-r},\tilde{V}_{(r)})$, and $\tilde{R}\in \Hom(\tilde{V}_{r+1},\tilde{V}_{(r)})$, we have
\[
\exp \tilde{R} \cdot h(g,\tilde{m}\tilde{g}\tilde{n},S)=\psi(\Tr \tilde{R}\tilde{m}^{-1}TS^{\ast}\tilde{g})h(g,\tilde{m}\tilde{g}\tilde{n},S).
\]
Hence
\[
\exp \tilde{R}\cdot \bar{\lambda}=\psi(\Tr \tilde{R}\tilde{m}^{-1}TS^{\ast}\tilde{g})\bar{\lambda}.
\]
In other words, the action of $\exp \tilde{R}$ on $\bar{\lambda}$ is given by multiplication by the function $(g,\tilde{m}\tilde{g}\tilde{n},S) \mapsto \psi(\Tr \tilde{R}\tilde{m}^{-1}TS^{\ast}\tilde{g})$. On the other hand, since $\lambda$ is $\tilde{U}_{r+1}$-equivariant, we should have
\[
\exp \tilde{R} \cdot \bar{\lambda}=\psi(\Tr \tilde{R}\tilde{X})\bar{\lambda},  \qquad \mbox{for all $\tilde{R}\in \Hom(\tilde{V}_{r+1},\tilde{V}_{(r-1)}\oplus
\tilde{V}_{-r})\subset \tilde{\u}_{r+1}$}.
\]
Here for the trace function, we are considering $\tilde{X}$ and $\tilde{R}$ as elements of $\End(\tilde{V}_{(r+1)})$. Therefore,
\[
\supp \bar{\lambda}\subset \{(g,\tilde{m}\tilde{g}\tilde{n},S) \, | \, \mbox{$ \tilde{R}\tilde{m}^{-1}TS^{\ast}\tilde{g}=\tilde{R}\tilde{X}$, for all $\tilde{R}\in \Hom(\tilde{V}_{r+1},\tilde{V}_{(r-1)}\oplus
\tilde{V}_{-r})$}\},
\]
but the latter set is empty since $\Im \tilde{X}$ contains $\tilde{V}_{r+1}$, whereas the image of $\tilde{m}^{-1}TS^{\ast}\tilde{g}$ does not have the same property due to the fact that $T$ is singular. From all this, we conclude that if $\Orb\subset \A$, then $D'(\Orb;\Y_{(r),(r)})^{\tilde{U}_{r+1},\tilde{\chi}_{r+1}}=0$. This is statement (\ref{vanish-orbit}) which, as we have already seen, implies statement (\ref{eq:claimvanish}).

Having proved statement (\ref{eq:claimvanish}), we come back to our original assertion that the map $\Y_{(r),(r+1)} \twoheadrightarrow
(\Y_{(r),(r+1)})_{\tilde{U}_{r+1},\tilde{\chi}_{r+1}}$ factors
through $\S(\Hom_{GNM}(V_{(r)},\tilde{V}_{r+1}))\otimes
\Y_{(r),(r)}$. We have seen that any $\lambda\in
(\Y_{(r),(r+1)}')^{\tilde{U}_{r+1},\tilde{\chi}_{r+1}}$ is completely determined by its restriction to
$C_{c}^{\infty}(\A^{c})$, where $\A^{c}$ is the complement of $\A$ in
$\Hom(V_{(r)},\tilde{V}_{r+1})$. But on $\A^{c}$, $0$ is a regular
value of the map $T\mapsto TT^{\ast}$. Therefore, $\lambda|_{A^{c}}$ is a
distribution that \emph{lives on} $\Hom_{GNM}(V_{(r)},\tilde{V}_{r+1})$, that is, $\lambda$ can be identified with an
element of $(\S(\Hom_{GNM}(V_{(r)},\tilde{V}_{r+1}))\otimes
\Y_{(r),(r)})'$, as we wanted to show. C.f. \cite[page 49, Exercise 4.5]{FJ}. (Note that we may make a suitable change of variables, when we have a regular value.)

Given $f\in \S(\Hom_{GNM}(V_{(r)},\tilde{V}_{r+1}))\otimes
\Y_{(r),(r)}$, we set
\begin{equation}
\tilde{f}_{r}(g)=[\omega_{(r),(r+1)}(g)f](T_{r})=\omega_{(r),(r)}(g)[f(T_{r}g)],\qquad \mbox{for $g\in G$.}
\end{equation}
Then it is clear that $\tilde{f}_{r}\in
C^{\infty}(N_{r} G_{(r-1)}\backslash G;\Y_{(r),(r)})$ for all $f\in
\S(\Hom_{GNM}(V_{(r)},\tilde{V}_{r+1}))\otimes \Y_{(r),(r)}$. Let
\[
\S^{+}(N_{r} G_{(r-1)}\backslash G;\Y_{(r),(r)})=\{\tilde{f}_{r}\, | \, f\in
\S(\Hom_{GNM}(V_{(r)},\tilde{V}_{r+1}))\otimes \Y_{(r),(r)}\}.
\]
Now observe that $\Hom_{GNM}(V_{(r)},\tilde{V}_{r+1})$ is a single orbit
under the action of $G$. Therefore, since  $G_{(r-1)}N_{r}$ is the
stabilizer of $T_{r} \in
\Hom_{GNM}(V_{(r)},\tilde{V}_{r+1})$,  we have that as $G$-modules
\[
\S(\Hom_{GNM}(V_{(r)},\tilde{V}_{r+1}))\otimes \Y_{(r),(r)}\cong
\S^{+}(N_{r} G_{(r-1)}\backslash G;\Y_{(r),(r)}).
\]

Now we will show that
\begin{equation}\label{eqn:Sastchirplus1}
 \S^{+}(N_{r} G_{(r-1)}\backslash G;\Y_{(r),(r)})_{\tilde{U}_{r+1},\tilde{\chi}_{r+1}}\cong \S^{+}(N_{r} G_{(r-1)}\backslash G;\S_{-r,-r}\otimes \Y_{(r-1),(r)}),
\end{equation}
where
\[
\S^{+}(N_{r} G_{(r-1)}\backslash G;\S_{-r,-r}\otimes\Y_{(r-1),(r)})=\{f_{r} \,
| \, f\in \Y_{(r),(r+1)}\}.
\]
So let $\lambda \in
(\S^{+}(N_{r} G_{(r-1)}\backslash G;\Y_{(r),(r)})')^{\tilde{U}_{r+1},\tilde{\chi}_{r+1}}$.
We will identify $\lambda$ with a $\Y_{(r-1),(r)}'$-valued
distribution on $G\times \Hom(V_{-r},\tilde{V}_{(r)})$. Observe
that, \mbox{for all $ f\in \S^{+}(N_{r} G_{(r-1)}\backslash G;\Y_{(r),(r)})$}, $g\in G$, $S\in \Hom(V_{-r},\tilde{V}_{(r)})$, and $\tilde{R}\in
\Hom(\tilde{V}_{r+1},\tilde{V}_{(r)})$, we have
\[
\exp \tilde{R}\cdot f(g)(S)=\psi(\Tr \tilde{R}T_{r}S^{\ast})f(g)(S).
\]
Therefore, since $\lambda\in
(\S^{+}(N_{r} G_{(r-1)}\backslash G;\Y_{(r),(r)})')^{\tilde{U}_{r+1},\tilde{\chi}_{r+1}}$,
we must have
\[
\psi(\Tr \tilde{R}T_{r}S^{\ast})\lambda=\psi(\Tr \tilde{R}\tilde{X})\lambda,
\]
for all $\tilde{R}\in \Hom(\tilde{V}_{r+1},\tilde{V}_{(r-1)}\oplus
\tilde{V}_{-r}) \subset \tilde{\u}_{r+1}$. Now, since $\tilde{X}$ is a regular value of the
map $S\mapsto T_{r}S^{\ast}$, we conclude that $\lambda$ is a
distribution that \emph{lives on}
\[
G\times \{S \in \Hom(V_{-r},\tilde{V}_{(r)})\, | \,
\mbox{$\tilde{R}T_{r}S^{\ast}=\tilde{R}\tilde{X}$ for all $\tilde{R}\in
\Hom(\tilde{V}_{r+1},\tilde{V}_{(r-1)}\oplus
\tilde{V}_{-r})$}\}.
\]
But the last set is just $G\times [T_{-r}+\Hom(V_{-r},\tilde{V}_{-r})]$. The isomorphism given in equation (\ref{eqn:Sastchirplus1}) now follows immediately.

The only thing that remains to be checked is that
\begin{equation}\label{eq:SastS}
\S^{+}(N_{r} G_{(r-1)}\backslash G;\S_{-r,-r}\otimes\Y_{(r-1),(r)})\cong
\Sch(N_{r} G_{(r-1)}\backslash G;\S_{-r,-r}\otimes\Y_{(r-1),(r)}).
\end{equation}
For this, let $\rho$ be a seminorm in $\Y_{(r-1),(r)}$, $Z_{1}\in
\D(\Hom(V_{-r},\tilde{V}_{-r}))$, $Z_{2}\in U(\g)$, $d_{1}$, $d_{2}\in
\N$ and $f\in \S_{(r),r+1}\otimes \S_{-r,(r)}\otimes
\Y_{(r-1),(r)}$. Then
\begin{eqnarray*}
q_{Z_{1},Z_{2},d_{1},d_{2},\rho}(f_{r}) & = &\sup_{k\in K, m\in M_{r}, S\in \Hom(V_{-r},\tilde{V}_{-r})} \rho(Z_{1}[R_{Z_{2}}f_{r}(mk)](S))\|mk\|^{d_{1}}(1+\|S\|)^{d_{2}} \\
   & = &\sup_{k, m, S} \rho(Z_{1}[(k\cdot R_{Z_{2}}f)_{r}(m)](S))\|m\|^{d_{1}}(1+\|S\|)^{d_{2}} \\
                  & = & \sup_{k,m,S}\rho(\omega_{(r),(r)}(m)Z_{1}(k\cdot R_{Z_{2}}f)(T_{r}m)(T_{-r}+S))\|m\|^{d_{1}}(1+\|S\|)^{d_{2}}\\
                  & = & \sup_{k,m,S}\rho(Z_{1}(k\cdot R_{Z_{2}}f)(T_{r}m)(T_{-r}m+Sm))\|m\|^{d_{1}}(1+\|S\|)^{d_{2}}. \\
                  & \leq & \sup_{k,m,S} q_{Z_{1},c_{1},c_{2},\rho}(k\cdot R_{Z_{2}}f)(1+\|T_{r}m+T_{-r}m\|)^{-c_{1}}\\ & & {}\cdot (1+\|Sm\|)^{-c_{2}}    \|m\|^{d_{1}}(1+\|S\|)^{d_{2}},
\end{eqnarray*}
where $c_{1}$ ,$c_{2}\in \N$ are arbitrary and $q_{Z_{1},c_{1},c_{2},\rho}$  is a seminorm on $\S_{(r),r+1}\otimes \S_{-r,(r)}\otimes
\Y_{(r-1),(r)}$. But then there exists a constant $C>0$ such that
\begin{eqnarray*}
q_{Z_{1},Z_{2},d_{1},d_{2},\rho}(f_{r}) & \leq & \sup_{k,m,S}\, C\, q_{Z_{1},c_{1},c_{2},\rho}(k\cdot R_{Z_{2}}f)    \|m\|^{-c_{1}}(1+\|S\|)^{-c_{2}}\|m\|^{c_{2}}    \|m\|^{d_{1}}(1+\|S\|)^{d_{2}} \\
& \leq &\sup_{k,m,S}\, C\, q_{Z_{1},c_{1},c_{2},\rho}(k\cdot R_{Z_{2}}f) \|m\|^{d_{1}+c_{2}-c_{1}}(1+\|S\|)^{d_{2}-c_{2}}.
\end{eqnarray*}
Now observe that if we fix $c_{2} \gg d_{2}$ and $c_{1}\gg d_{1}+c_{2}$, then
\[
\sup_{m,S} C q_{Z_{1},c_{1},c_{2},\rho}(k\cdot R_{Z_{2}} f)\|m\|^{d_{1}+c_{2}-c_{1}}(1+\|S\|)^{d_{2}-c_{2}} <\infty,
\]
for all $k\in K$. Furthermore, the map $k\mapsto \sup_{m,S} C \,  q_{Z_{1},c_{1},c_{2},\rho}(k\cdot f)\|m\|^{d_{1}+c_{2}-c_{1}}(1+\|S\|)^{d_{2}-c_{2}}$ is continuous, and hence, since $K$ is compact, bounded. Therefore $q_{Z_{1},Z_{2},d_{1},d_{2},\rho}(f_{r})<\infty$ and (\ref{eq:SastS}) follows.
\end{proof}

Observe that the space
$(\Y_{(r),(r+1)})_{\tilde{U}_{r+1},\tilde{\chi}_{r+1}}$ carries a
natural action of the group $\tilde{N}$. As a consequence of the isomorphism
$(\Y_{(r),(r+1)})_{\tilde{U}_{r+1},\tilde{\chi}_{r+1}}\cong
\Sch(N_{r} G_{(r-1)}\backslash G;\S_{-r,-r}\otimes\Y_{(r-1),(r)})$, we may define an action of
$\tilde{N}$ on $\Sch(N_{r} G_{(r-1)}\backslash
G;\S_{-r,-r}\otimes\Y_{(r-1),(r)})$ by the formula $(\tilde{n}\cdot
f_{r})(g)(S):=(\omega_{(r),(r+1)}(\tilde{n})f)_{r}(g)(S)$, for
$f\in \Y_{(r),(r+1)}$, $\tilde{n}\in \tilde{N}$ and $S\in \Hom(V_{-r},\tilde{V}_{-r})$. We describe this action more concretely in the following

\begin{lemma}\label{lemma:inducedaction}
There is an induced action of $\tilde{N}$ ($=\tilde{N}_{(r)}\tilde{N}_{r+1}$) on $\Sch(N_{r} G_{(r-1)}\backslash
G;\S_{-r,-r}\otimes\Y_{(r-1),(r)})$, as follows:
for all $\phi\in  \Sch(N_{r} G_{(r-1)}\backslash G;\S_{-r,-r}\otimes\Y_{(r-1),(r)})$, $g\in G$, $S\in \Hom(V_{-r},\tilde{V}_{-r})$,
\begin{eqnarray*}
(\exp \tilde{Z}\cdot \phi)(g)(S)&= & \phi(g)(S) \qquad \mbox{for $\tilde{Z}\in \tilde{\z}_{r+1}$}\\
(\exp \tilde{R}\cdot \phi)(g)(S)&= & \psi(\Tr \tilde{R}T_{r}T_{-r}^{\ast})\psi(\Tr \tilde{R}T_{r}S^{\ast})\phi(g)(S) \qquad \mbox{for $\tilde{R}\in \Hom(\tilde{V}_{r+1},\tilde{V}_{(r)})$}\\
(\exp \tilde{Z}\cdot \phi)(g)(S)&= &\omega_{(r-1),(r)}(\exp \tilde{Z})[ \phi(g)(S)] \qquad \mbox{for $\tilde{Z}\in \tilde{\z}_{r}$}\\
(\exp \tilde{R}\cdot \phi)(g)(S)&= &\omega_{(r-1),(r)}(\exp \tilde{R})[ \phi(g)(S+\tilde{R}^{\ast}T_{-r})] \qquad \mbox{for $\tilde{R}\in \Hom(\tilde{V}_{r},\tilde{V}_{(r-1)})$}\\
(\tilde{n}\cdot \phi)(g)(S)&=&\omega_{(r-1),(r)}(\tilde{n})[\phi(g)(S)]  \qquad \mbox{for $\tilde{n}\in \tilde{N}_{(r-1)}.$}
\end{eqnarray*}
In particular, we have
\begin{equation}
\label{eq:ntildeaction6}
(\tilde{u}\cdot \phi)(g)(S)=\omega_{(r-1),(r)}(\tilde{u})[\phi(g)(S)]  \qquad \mbox{for $\tilde{u}\in \tilde{U}_{(r)}$.}
\end{equation}
\end{lemma}

\begin{proof} If $\tilde{Z}\in
\tilde{\z}_{r+1}$, we have already seen in equation (\ref{f(r)center}) that
\[
(\omega_{(r),(r+1)}(\exp \tilde{Z})f)_{r}(g)(S)  = f_{r}(g)(S),
\]
while if $\tilde{R} \in \Hom(\tilde{V}_{r+1},\tilde{V}_{(r)}) \subset \tilde{\n}_{r+1}$, similar to equation (\ref{f(r)character}), we have
\begin{eqnarray*}
(\omega_{(r),(r+1)}(\exp \tilde{R})f)_{r}(g)(S)&= &\psi(\Tr \tilde{R}T_{r}(T_{-r}+S)^{\ast})f_{r}(g)(S) \\
                                  &= &\psi(\Tr \tilde{R}T_{r}T_{-r}^{\ast})\psi(\Tr \tilde{R}T_{r}S^{\ast})f_{r}(g)(S).
\end{eqnarray*}
On the other hand, if $\tilde{Z}\in \tilde{\z}_{r}$, then
\begin{eqnarray*}
(\omega_{(r),(r+1)}(\exp \tilde{Z})f)_{r}(g)(S) & = & (\omega_{(r),(r+1)}(\exp \tilde{Z})[\omega_{(r),(r+1)}(g)f])(T_{r},T_{-r}+S) \\
& = & (\omega_{(r),(r)}(\exp \tilde{Z})[(\omega_{(r),(r+1)}(g)f)(T_{r})])(T_{-r}+S)  \\
& = & \omega_{(r-1),(r)}(\exp \tilde{Z})([\omega_{(r),(r+1)}(g)f](T_{r},(\exp \tilde{Z})^{-1}(T_{-r}+S))) \\
& = & \omega_{(r-1),(r)}(\exp \tilde{Z})([\omega_{(r),(r+1)}(g)f](T_{r},T_{-r}+S)) \\
& = & \omega_{(r-1),(r)}(\exp \tilde{Z})f_{r}(g)(S).
\end{eqnarray*}
Here we are using that $\tilde{Z}(T_{-r}+S)=0$ because $\Im (T_{-r}+S)\subset \tilde{V}_{-r+1}\oplus \tilde{V}_{-r}$ and $\tilde{Z}|_{\tilde{V}_{-r+1}\oplus \tilde{V}_{-r}}=0$. On the other hand, if $\tilde{R}\in \Hom(\tilde{V}_{r},\tilde{V}_{(r-1)})\subset \tilde{\n}_{r}$, then, according to equations (\ref{eq:sigmaaction1}), (\ref{eq:sigmaaction2}) and (\ref{eq:sigmaaction3})
\begin{eqnarray*}
(\omega_{(r),(r+1)}(\exp \tilde{R})f)_{r}(g)(S) & = & (\omega_{(r),(r+1)}(\exp \tilde{R})[\omega_{(r),(r+1)}(g)f])(T_{r},T_{-r}+S) \\
& = & (\omega_{(r),(r)}(\exp \tilde{R})[(\omega_{(r),(r+1)}(g)f)(T_{r})])(T_{-r}+S)  \\
& = & \omega_{(r-1),(r)}(\exp \tilde{R})([\omega_{(r),(r+1)}(g)f](T_{r},(\exp \tilde{R})^{-1}(T_{-r}+S))) \\
& = & \omega_{(r-1),(r)}(\exp \tilde{R})([\omega_{(r),(r+1)}(g)f](T_{r},T_{-r}+S+\tilde{R}^{\ast}T_{-r})) \\
& = & \omega_{(r-1),(r)}(\exp \tilde{R})f_{r}(g)(S+\tilde{R}^{\ast}T_{-r}).
\end{eqnarray*}
Here we are using the identification $\Hom(\tilde{V}_{r},\tilde{V}_{(r-1)}) \ni \tilde{R} \leftrightarrow \tilde{R}-\tilde{R}^{\ast}\in \tilde{\n}_{r}$ as in equation (\ref{eq:tildenmisomorphism}). We are also using that, since $T_{-r}\in \Hom(V_{-r},\tilde{V}_{-r+1})$, then $\tilde{R}^{\ast}T_{-r}\in \Hom(V_{-r},\tilde{V}_{-r})$ and that
\[
\tilde{R}^{\ast}S_{-r}=\tilde{R}S_{-r}=\tilde{R}T_{-r}=\tilde{R}^{\ast}\tilde{R}T_{-r}=\tilde{R}^{\ast}\tilde{R}S_{-r}=0,
\]
 for all $\tilde{R}\in \Hom(\tilde{V}_{r},\tilde{V}_{(r-1)})$, $S\in \Hom(V_{-r},\tilde{V}_{-r})$. We note, in particular, that $\tilde{R}^{\ast}T_{-r}=0$ if $\tilde{R}\in \Hom(\tilde{V}_{r},\tilde{V}_{(r-2)}\oplus V_{-r+1})\subset \tilde{\u}_{r}$.
Finally, if $\tilde{n}\in \tilde{N}_{(r-1)}\subset \tilde{G}_{(r)}$, then
\begin{eqnarray*}
(\omega_{(r),(r+1)}(\tilde{n})f)_{r}(g)(S) & = & (\omega_{(r),(r+1)}( \tilde{n})[\omega_{(r),(r+1)}(g)f])(T_{r},T_{-r}+S) \\
& = & (\omega_{(r),(r)}( \tilde{n})[(\omega_{(r),(r+1)}(g)f)(T_{r})])(T_{-r}+S)  \\
& = & \omega_{(r-1),(r)}( \tilde{n})([\omega_{(r),(r+1)}(g)f](T_{r}, \tilde{n}^{-1}(T_{-r}+S))) \\
& = & \omega_{(r-1),(r)}( \tilde{n})([\omega_{(r),(r+1)}(g)f](T_{r},T_{-r}+S)) \\
& = & \omega_{(r-1),(r)}( \tilde{n})f_{r}(g)(S).
\end{eqnarray*}
Here we are using that, considered as an element of $\tilde{G}_{(r)}$, $\tilde{n}\in \tilde{N}_{(r-1)}$ leaves $\tilde{V}_{-r+1}\oplus \tilde{V}_{-r}$ fixed.
\end{proof}

\subsection{From $\Y_{(r),(r+1)}$ to $\HH_{\gamma, \tilde{\gamma}}$: proof of Proposition \ref{prop:tildechicoinvariants}}
\label{subsection:keyprop}
 We are finally ready to complete the proof of the key proposition, to which we refer the reader for the unexplained notation below.

\begin{proof}[Proof (of Proposition \ref{prop:tildechicoinvariants})]
We will prove the result by induction on $r$.

We first look at the case where $r=0$. Observe that in this case $U_{1}=U$, $\tilde{\chi}_{1}=\tilde{\chi}$ and $\z_{1}=\u$. We want to show that the map $f\mapsto f_{(0)}$ induces an isomorphism between $(\Y_{(0),(1)})_{\tilde{U},\tilde{\chi}}$ and $\Sch(G,\HH_{\gamma,\tilde{\gamma}})$.

Let $\lambda\in (\Y'_{(0),(1)})^{\tilde{U},\tilde{\chi}}$. As in the proof of Lemma \ref{lemma:outerlayer}, we identify $\lambda$ with a $\Y'_{(0),(0)}$-valued distribution on $\Hom(V_{(0)},\tilde{V}_{1})$ by setting
\[
 \lambda(f)(v)=\lambda(f\otimes v) \qquad \mbox{ for all $f\in \S_{(0),1}$, $v\in \Y_{(0),(0)}$.}
\]
As in Equation (\ref{eq:tildezaction}) we have that for all $\tilde{Z}\in \u$, $f \in \S_{(0),1}$, $T\in \Hom(V_{(0)},\tilde{V}_{1})$,
\begin{equation}\label{eq:r0action}
 \omega_{(0),(1)}(\exp \tilde{Z}) \lambda =\psi(\Tr \tilde{Z}TT^{\ast}/2)\lambda.
\end{equation}
On the other hand, since $\lambda \in (\Y'_{(0),(1)})^{\tilde{U},\tilde{\chi}}$, we must have
\begin{equation}\label{eq:r0invariance}
 \omega_{(0),(1)}(\exp \tilde{Z})\lambda= \psi(\Tr \tilde{Z}\tilde{X}/2)\lambda.
\end{equation}
Observe that, in contrast with the situation in Lemma \ref{lemma:outerlayer}, $\tilde{X}:\tilde{V}_{-1}\longrightarrow \tilde{V}_{1}$ is a regular value of the map $T\mapsto TT^{\ast}$. Let
\[
 \Orb=\{T\in \Hom(V_{(0)},\tilde{V}_{1})\, | \, TT^{\ast}=\tilde{X}\}.
\]
From Equations (\ref{eq:r0action}), (\ref{eq:r0invariance}) and the above observation it follows that if $\lambda \in (\Y'_{(0),(1)})^{\tilde{U},\tilde{\chi}}$, then $\lambda$ is in fact a distribution that lives on $\Orb$. Since $r=0$, we have an isomorphism $\Y_{(0),(0)}\cong \HH_{\gamma,\tilde{\gamma}}$. On the other hand, if we fix a $T_{(0)}\in \Orb$ we have an isomorphism
\[
 \begin{array}{ccc} G & \cong & \Orb\\
                  g& \mapsto &   g\cdot T_{(0)}.
 \end{array}
\]
From all this it follows that $(\Y_{(0),(1)})_{\tilde{U},\tilde{\chi}}\cong \S^{+}(G;\HH_{\gamma,\tilde{\gamma}})$, where
\[
 \S^{+}(G;\HH_{\gamma,\tilde{\gamma}}):=\{f_{(0)}\, | \, f\in \Y_{(0),(1)}\}.
\]
The result now follows by noting that the growth conditions imposed on functions in $\Sch(G,\HH_{\gamma,\tilde{\gamma}})$ and $\S^{+}(G;\HH_{\gamma,\tilde{\gamma}})$ are in fact equivalent (c.f. the proof of Equation \eqref{eq:SastS}).

For the rest of proof, we assume that $r>0$.
Observe that $\tilde{U}=\tilde{U}_{(r)}\tilde{U}_{r+1}$ and $\tilde{\chi}_{(r+1)}=\tilde{\chi}_{(r)}\tilde{\chi}_{r+1}$. Therefore, from Lemma \ref{lemma:outerlayer}
\begin{eqnarray*}
(\Y_{(r),(r+1)})_{\tilde{U},\tilde{\chi}} & \cong & (\Sch(N_{r}G_{(r-1)}\backslash G;\S_{-r,-r}\otimes \Y_{(r-1),(r)}))_{\tilde{U}_{(r)},\tilde{\chi}_{(r)}} \\
 & \cong & \Sch(N_{r}G_{(r-1)}\backslash G;\S_{-r,-r}\otimes (\Y_{(r-1),(r)})_{\tilde{U}_{(r)},\tilde{\chi}_{(r)}}),
\end{eqnarray*}
where the last equation follows from the fact that $\tilde{U}_{(r)}$ acts only on the values. See equation (\ref{eq:ntildeaction6}).
Now, given $f\in \S_{-r,-r}\otimes \Y_{(r-1),(r)}$, $g\in G$ and $S_{-k}\in \Hom(V_{-k},\tilde{V}_{-k})$, $k=1,\ldots,l$, set
\begin{eqnarray*}
f_{(r-1)}(g)(S_{-r},\ldots,S_{-1})& = & (\bar{\sigma}_{T_{r}}(g)f)(S_{-r},T_{r-1},T_{-r+1}+S_{-r+1},\ldots,T_{0}) \\
& = & (\omega_{(r-1),(r)}(g)[f(S_{-r})])(T_{r-1},T_{-r+1}+S_{-r+1},\ldots,T_{0}).
\end{eqnarray*}
For convenience, we denote the representation $\tau_{\gamma,\tilde{\gamma}}^{T}$ on $\HH_{\gamma,\tilde{\gamma}}$ (defined in equation (\ref{eq:deftaut})) by the more concise symbol $\HH_{T}$.  Then, by induction hypothesis, the map $f\mapsto f_{(r-1)}$ induces
a $G_{(r-1)}$-intertwining isomorphism
\[
\Psi_{(r-1)}: \S_{-r,-r}\otimes
(\Y_{(r-1),(r)})_{\tilde{U}_{(r)},\tilde{\chi}_{(r)}}
\longrightarrow \Sch(N_{(r-1)}\backslash
G_{(r-1)}; \HH_{T_{(r)}}),
\]
where $T_{(l)}=\oplus_{k=-l}^{l} T_{k}$, for $l=0,\ldots,r$. Note that $T_{(r)}=T_{\gamma,\tilde{\gamma}}$. Also note that in the above equation we are identifying the spaces $\S_{-r,-r}\otimes \Sch(N_{(r-1)}\backslash
G_{(r-1)}; \HH_{T_{(r-1)}}) \cong \Sch(N_{(r-1)}\backslash
G_{(r-1)}; \HH_{T_{(r)}})$. We claim that $\Psi_{(r-1)}$ is actually an isomorphism of $G_{(r-1)}N_{r}$-modules:
\begin{equation}
\Psi_{(r-1)}: \S_{-r,-r}\otimes
(\Y_{(r-1),(r)})_{\tilde{U}_{(r)},\tilde{\chi}_{(r)}}
\longrightarrow  \Sch(N\backslash G_{(r-1)}N_{r}; \HH_{T_{(r)}}),
\label{eq:GNr1isomorphism}
\end{equation}
where
\[
\Sch(N\backslash G_{(r-1)}N_{r};
\HH_{T_{(r)}})=\{f \in C^{\infty}(N\backslash G_{(r-1)}N_{r};
\HH_{T_{(r)}})\,
| \, f|_{G_{(r-1)}}\in
\Sch(N_{(r-1)}\backslash G_{(r-1)}; \HH_{T_{(r-1)}}). \}
\]
Observe that if $f\in \Sch(N\backslash G_{(r-1)}N_{r};
\HH_{T_{(r)}})$, $g\in G_{(r-1)}$ and $n\in N_{r}$, then
\[
(n\cdot f)(g)=f(gn)=f((gng^{-1})g)=\tau_{T_{(r)}}(gng^{-1})f(g).
\]
Therefore, to prove our claim, we need to show that for all $f\in \S_{-r,-r}\otimes \Y_{(r-1),(r)}$, $g\in G_{(r-1)}$, $n\in N_{r}$,
\begin{equation}
(\bar{\sigma}_{T_{r}}(n)f)_{(r-1)}(g)=\tau_{T_{(r)}}(gng^{-1})f_{(r-1)}(g). \label{eq:compatibilitysigmatau}
\end{equation}
Now, by definition,
\begin{eqnarray*}
(\bar{\sigma}_{T_{r}}(n)f)_{(r-1)}(g)(S_{-r},\ldots,S_{-1}) & = &  (\bar{\sigma}_{T_{r}}(g)\bar{\sigma}_{T_{r}}(n)f)(S_{-r},T_{r-1},T_{-r+1}+S_{-r+1},\ldots,T_{0}) \\
 & = &  (\bar{\sigma}_{T_{r}}(gng^{-1})[\bar{\sigma}_{T_{r}}(g)f])(S_{-r},T_{r-1},T_{-r+1}+S_{-r+1},\ldots,T_{0}) \\
 & = &  (\bar{\sigma}_{T_{r}}(gng^{-1})[\bar{\sigma}_{T_{r}}(g)f])_{(r-1)}(e)(S_{-r},S_{-r+1},\ldots,S_{-1}).
\end{eqnarray*}
From all this we see that in order to prove the claim given in equation (\ref{eq:GNr1isomorphism}) it suffices to show that for all $f\in \S_{-r,-r}\otimes \Y_{(r-1),(r)}$, $n\in N_{r}$,
\begin{equation}
(\bar{\sigma}_{T_{r}}(n)f)_{(r-1)}(e)=\tau_{T_{(r)}}(n)f_{(r-1)}(e). \label{eq:simplecompatibilitysigmatau}
\end{equation}
Now, by equation (\ref{eq:barsimgaTr1}), if $Z\in \z_{r}$, then
\[
(\bar{\sigma}_{T_{r}}(\exp Z)f)_{(r-1)}(e)(S_{-r},\dots,S_{-1})=f(S_{-r},T_{r-1},\ldots,T_{0})=f_{(r-1)}(e)(S_{-r},\dots,S_{-1}).
\]
On the other hand, by equation (\ref{eq:barsimgaTr2}), if $R\in \Hom(V_{(r-2)}\oplus V_{r-1},V_{-r})\subset \u_{r}$, then
\begin{eqnarray*}
(\bar{\sigma}_{T_{r}}(\exp R)f)_{(r-1)}(e)(S_{-r},\dots,S_{-1}) & = & (\omega_{(r-1),(r)}(T_{-r}R)[f(S_{-r})])(T_{r-1},\ldots,T_{0}) \\
& = & \psi(-\Tr T_{-r}RT_{r-2}^{\ast})f(S_{-r},T_{r-1},\ldots,T_{0}) \\
& = & \chi_{\gamma}(\exp R)^{-1}f_{(r-1)}(e)(S_{-r},\ldots,S_{-1}),
\end{eqnarray*}
in other words, for all $u\in U_{r}$,
\[
(\bar{\sigma}_{T_{r}}(u)f)_{(r-1)}(e)(S_{-r},\dots,S_{-1})=\chi_{\gamma}(u)^{-1}f_{(r-1)}(e)(S_{-r},\ldots,S_{-1}).
\]
Finally, again by equation (\ref{eq:barsimgaTr2}), if $R\in \g_{-1}\cap \n_{r}\cong \Hom(V_{-r+1},V_{-r})$, then
\begin{eqnarray*}
(\bar{\sigma}_{T_{r}}(\exp R)f)_{(r-1)}(e)(S_{-r},\dots,S_{-1}) & = & (\omega_{(r-1),(r)}(T_{-r}R+S_{-r}R)[f(S_{-r})](T_{r-1},\ldots,T_{0}) \\
 & = & \psi(-\Tr S_{-r}RT^{\ast}_{r-1})f_{(r-1)}(e)(S_{-r},S_{-r+1}+T_{-r}R,\ldots,S_{-1}).
\end{eqnarray*}
Now, for such an $R$, $\sigma_{T_{(r)}}(R,0)=T_{-r}R-T_{r-1}R^{\ast}$. Therefore,
\begin{eqnarray*}
\tau_{T_{(r)}}(\exp R)f_{(r-1)}(e)(S_{-r},\ldots,S_{-1}) & = &\tau_{\gamma,\tilde{\gamma}}(T_{-r}R-T_{r-1}R^{\ast})f_{(r-1)}(e)(S_{-r},\ldots, S_{-1}) \\
 & = & \psi(-\Tr S_{-r}RT^{\ast}_{r-1})f_{(r-1)}(e)(S_{-r},S_{-r+1}+T_{-r}R,\ldots,S_{-1}).
\end{eqnarray*}
Therefore, equation (\ref{eq:simplecompatibilitysigmatau}) holds true, and hence we have proved our claim given in equation (\ref{eq:GNr1isomorphism}).

We will now show that
\[
\Sch(N\backslash G; \HH_{T_{(r)}})\cong \Sch(N_{r} G_{(r-1)}\backslash G;\Sch(N\backslash G_{(r-1)}N_{r};\HH_{T_{(r)}})).
\]
Given $f\in  \Sch(N_{r} G_{(r-1)}\backslash G;\Sch(N\backslash G_{(r-1)}N_{r};\HH_{T_{(r)}}))$, set $\check{f}(g)=f(g)(e)$, for all $g\in G$. Then it is clear that $\check{f}\in C^{\infty}(N\backslash G; \HH_{T_{(r)}})$, but we claim that $\check{f}$ is actually in $\Sch(N\backslash G;\HH_{T_{(r)}})$. Effectively, given a semi-norm $\rho$ of $\HH_{T_{(r)}}$, $Z\in U(\g)$ and $d\in \N$,
\begin{eqnarray*}
q_{Z,d,\rho}(\check{f})&= &\sup_{k\in K, \, m_{1}\in M_{r}, \, m_{2}\in G_{(r-1)}} \rho(R_{Z}\check{f}(m_{2} mk_{1} )) \|m_{1}m_{2}\|^{d}\\
             &= & \sup_{k,m_1,m_2}\rho(R_{Z}f(m_{2} mk_{1} )(e)) \|m_{1}m_{2}\|^{d}  \\
                            & \leq & \sup_{k,m_1,m_2}\rho(R_{Z}f(mk_{1})(m_{2}^{-1})) \|m_{1}\|^{d}\|m_{2}\|^{d}    \\
            & = & p_{Z,1,d,d,\rho}(f)<\infty.
\end{eqnarray*}
Now given $f\in \Sch(N\backslash G; \HH_{T_{(r)}})$, set $\hat{f}(g)(h)=f(hg)$. Then again it is clear that
\[
\hat{f} \in C^{\infty}(N_{r} G_{(r-1)}\backslash G;C^{\infty}(N\backslash N_{r}G_{(r-1)};\HH_{T_{(r)}})),
\]
 but we claim that, actually, $\hat{f}$ is a function in $ \Sch(N_{r} G_{(r-1)}\backslash G;\Sch(N\backslash G_{(r-1)}N_{r};\HH_{T_{(r)}}))$. Effectively, let $\rho$ be a semi-norm of $\HH_{T_{(r)}}$, $Z_{1}\in U(\g)$, $Z_{2}\in U(\g_{(r-1)})$ and $d_{1}$, $d_{2}\in \N$. Then $Z_{2}\in U^{m}(\g)$ for some $m$ and hence there exists a basis $\{Y_{j}\}_{j=1}^{l}$ of $U^{m}(\g)$ such that for all $g\in G$, $\Ad(g) Z_{2}=\sum_{j} a_{j}(g)Y_{j}$ for some functions $a_{j}$. Since $(\Ad,U^{m}(\g))$ is a finite dimensional representation, there exists a constant $d_{m}$ such that $|a(g)|\leq \|g\|^{d_{m}}$ for all $g\in G$. Taking this into account we have that
\begin{eqnarray*}
p_{Z_{1},Z_{2},d_{1},d_{2},\rho}(\hat{f})&= &\sup_{k\in K, \, m_{1}\in M_{r}, \, m_{2}\in G_{(r-1)}} \rho(R_{Z_{2}}(R_{Z_{1}}\hat{f}(mk_{1}))( m_{2})) \|m_{1}\|^{d_{1}}\|m_{2}\|^{d_{2}}\\
             &= & \sup_{k,m_1,m_2}\rho(R_{\Ad(mk_{1})^{-1}Z_{2}}R_{Z_{1}}f(m_{2} mk_{1} )) \|m_{1}\|^{d_{1}}\|m_{2}\|^{d_{2}}\\
                            & \leq & \sup_{k,m_1,m_2}\sum_{j=1}^{l}\rho(a(mk_{1})R_{Y_{j}}R_{Z_{1}}f(mk_{1}m_{2})) \|m_{1}\|^{d_{1}}\|m_{2}\|^{d_{2}}   \\
            & \leq & \sup_{k,m_1,m_2}\sum_{j=1}^{l}\rho(R_{Y_{j}}R_{Z_{1}}f(mk_{1}m_{2})) \|m_{1}\|^{d_{1}+d_{m}}\|m_{2}\|^{d_{2}}  \\
& \leq & \sup_{k,m_1,m_2}\sum_{j=1}^{l}\rho(R_{Y_{j}Z_{1}}f(mk_{1}m_{2})) \|m_{1}m_{2}\|^{d_{1}+d_{m}+d_{2}}    \\
& = & \sum_{j=1}^{l} q_{Y_{j}Z_{1},d_{1}+d_{2}+d_{m},\rho}(f)<\infty.
\end{eqnarray*}
Here we have used the fact that $\|m_{1}m_{2}\|=\max\{\|m_{1}\|,\|m_{2}\|\}\geq \|m_{1}\|,\|m_{2}\|$.
Now the only thing left to be check is that $\hat{\check{f}}=f$ and $\check{\hat{f}}=f$, but by definition
\[
\hat{\check{f}}(g)(h)=\check{f}(gh)=f(gh)(e)=f(g)(h),
\]
and similarly $\check{\hat{f}}=f$.
\end{proof}

\appendix
\section{Some facts on vector valued distributions}
\label{sec:vectordistributions}

In this appendix, we collect some elementary facts about vector valued distributions. For a more detailed presentation of the material, the interested reader should refer to \cite{KV96}, on  which this appendix is based.
\vsp

\subsection{Vector valued distributions}
Let $\X$ be a $C^{\infty}$-manifold and $\V$ a Fr\'echet space. We write $C^{\infty}(\X;\V)$ and $C_{c}^{\infty}(\X;\V)$ for the topological vector spaces of the $C^{\infty}$ mappings from $\X$ to $\V$, and those with compact support, respectively. Let $\V'$ be the dual of $\V$ provided with the strong dual topology, that is, the topology of uniform convergence on bounded subsets of $\V$. Finally, let $\Diff^{(l)}(\X)$ be the $C^{\infty}(\X)$-module of $C^{\infty}$ differential operators on $X$ of order $\leq l\in \Z_{\geq 0}$.

\begin{dfn} A $\V$-distribution on $\X$ is an element of the topological vector space
\[
D'(\X;\V):=C_{c}^{\infty}(\X;\V)'\cong (C_{c}^{\infty}(\X)\otimes \V)'
\]
consisting of the continuous linear forms on $C_{c}^{\infty}(\X;\V)$, provided with the strong dual topology. A $\V'$-valued distribution on $\X$ is an element of the space $\Lin(C_{c}^{\infty}(\X);\V')$ of continuous linear maps from $C_{c}^{\infty}(\X)$ to $\V'$. We will also equip $\Lin(C_{c}^{\infty}(\X);\V')$ with the strong dual topology.
\end{dfn}

A $\V$ distribution $\lambda$ defines canonically a $\V'$-valued distribution $\tilde{\lambda}$ on $\X$, that is, we have a canonical isomorphism
\[
\lambda\longleftrightarrow \tilde{\lambda}  \qquad D'(\X;\V) \longleftrightarrow \Lin(C_{c}^{\infty}(\X);\V'),
\]
given as follows: if $\lambda\in D'(\X;\V)$ and $f\in C_{c}^{\infty}(\X)$, then $\tilde{\lambda}(f)\in \V'$ is well-defined by $\tilde{\lambda}(f)(v)=\lambda(f\otimes v)$, for all $v\in \V$. In what follows we will frequently identify these two spaces.

\vsp

\subsection{Transverse order of a $\V$-distribution}

\begin{dfn} A $\V$-distribution $\lambda\in D'(\X;\V)$ is said to be of order $\leq l$ if for every compact set $\mathcal{K}\subset \X$ there exists a constant $c>0$, $Z_{1},\ldots, Z_{m}\in \Diff^{(l)}(\X)$ and a seminorm $\nu$ on $\V$ such that for all $f\in C_{c}^{\infty}(\X;\V)$
\[
|\lambda(f)|\leq c \sum_{j=1}^{m} \sup_{x\in \mathcal{K}} \nu(Z_{j}f(x)).
\]
We write $D'^{(l)}(\X;\V)$ for the linear subspace of $D'(\X;\V)$ consisting of elements of order $\leq l$.
\end{dfn}

It can be shown that every $\V$-distribution on $\X$ is of finite order over every relatively compact open subset of $\X$.

Suppose $\mathcal{C}$ is a closed $C^{\infty}$ submanifold that is regularly embedded in $\X$. We write $D_{\mathcal{C}}'(\X;\V)$ for the distributions in $D'(\X;\V)$ with support contained in $\mathcal{C}$.

\begin{dfn} A $\V$-distribution $\lambda\in D'(\X;\V)$ is said to have transverse order $\leq l$ at $x\in \mathcal {C}$ if the following holds: there exists an open neighborhood $U_{x}$ of $x$ in $\X$ such that
\[
\lambda(f)=0, \qquad \begin{array}{c} \mbox{if $f\in C_{c}^{\infty}(U_{x};\V)$ satisfies $Zf|_{\mathcal{C}\cap   U_{x}}=0$}\\
                                                      \mbox{for all $Z\in \Diff^{(l)}(U_{x})$}. \end{array}
\]
\end{dfn}

Denote by $D_{\mathcal{C}}'^{(l)}(\X;\V)$ the linear subspace of elements in $D'(\X;\V)$ which have transverse order $\leq l$ for all points of $\mathcal{C}$. Note that $\supp(\lambda) \subseteq \mathcal{C}$, for $\lambda \in D_{\mathcal{C}}'^{(l)}(\X;\V)$, justifying the notation.

We say that $\lambda$ \emph{lives on} $\mathcal{C}$ if $\lambda$ belongs to $D'(\mathcal{C};\V):=D'^{(0)}_{\mathcal{C}}(\X;\V)$. This is the same as saying that $\lambda(f)=0$ if $f\in C_{c}^{\infty}(\X;\V)$ satisfies $f|_{\mathcal{C}}=0$. In this case there exists a unique distribution $\mu$ on $\mathcal{C}$ such that $\lambda(f)=\mu(f|_{\mathcal{C}})$, for all $f\in C_{c}^{\infty}(\X;\V)$; and often we will identify $\lambda$ with $\mu$.

\vsp

\subsection{Transverse jet bundle}\label{subsec:transversebundle}
As before, let $\X$ be a $C^{\infty}$-manifold, and let $\mathcal{C}\subset \X$ be a closed $C^{\infty}$-manifold regularly embedded in $\X$. For any $x\in \mathcal{C}$, let $C^{\infty}_{x}$ be the algebra of germs of $C^{\infty}$ functions defined around $x$ and $\Diff_{x}^{(l)}$ the $C^{\infty}_{x}$-module of germs of differential operators of order $\leq l$ defined around $x$. Let $V_{x}^{(l)}$ be the $C^{\infty}_{x}$-submodule of $\Diff_{x}^{(l)}$ generated by germs of $l$-tuples $(v_{1},\ldots,v_{l})$ of vector fields around $x$ for which at least one of the $v_{i}$'s is tangent to $\mathcal{C}$; let $I_{x}^{(l)}=\Diff_{x}^{(l-1)}+V_{x}^{(l)}$. By choosing local coordinates at $x$, it is easy to see that $I_{x}^{(l)}$ actually is the stalk at $x$ of a subsheaf $I^{(l)}$ of the sheaf $\Diff^{(l)}$ of differential operators of order $\leq l$ on $\X$. Hence we have a well-defined quotient sheaf
\[
M^{(l)}=\Diff^{(l)}/I^{(l)},
\]
with stalk at $x$ equal to $M_{x}^{(l)}=\Diff_{x}^{(r)}/I_{x}^{(l)}$. We write $\partial \mapsto \overline{\partial}$ for the projection  $\Diff_{x}^{(l)}\longrightarrow M_{x}^{(l)}$. Since $\Diff_{x}^{(l)}$ and $I_{x}^{(l)}$ are both $C^{\infty}_{x}$-modules, $M_{x}^{(l)}$ also has the structure of a $C^{\infty}_{x}$-module, and the projection is a mapping of $C^{\infty}_{x}$-modules.

Given $x\in \mathcal{C}$ we can select an open neighborhood $U_{x}$ of $X$ and local coordinates $(t:u)=(t_{1},\ldots,t_{p},u_{1},\ldots,u_{q})$ on $U_{x}$ such that
\[
U_{x}\cap \mathcal{C}=\{(t:u)\, | \, t_{1}=\ldots=t_{p}=0\}.
\]
It is easy to show that the elements $\overline{\partial^{\alpha}_{t}}$ ($|\alpha|=l$) form a free basis for the sections of $M^{(l)}$ around $x$. Here $\alpha=(\alpha_{1},\ldots,\alpha_{p})$ is a multi-index, $|\alpha|=\alpha_{1}+\cdots+ \alpha_{p}$ and $\partial^{\alpha}_{t}=\partial^{|\alpha|}/\partial_{t_{1}}^{\alpha_{1}}\cdots\partial_{t_{p}}^{\alpha_{p}}.$ Thus $M^{(l)}$ is a vector bundle over $\mathcal{C}$ of finite rank. This is the $l$-th graded part of the transverse jet bundle on $\mathcal{C}$. We have that $M^{(l)}$ is the $l$-th symmetric power of $M^{(1)}$. We write $M^{(l)'}$ for the dual to $M^{(l)}$.

\end{document}